\documentclass[12pt, reqno]{amsart}

\usepackage[english]{babel}
\usepackage{amsthm} 
\usepackage{amssymb}
\usepackage{mathtools, amsmath}
\usepackage{xcolor}
\usepackage{amsfonts}
\usepackage{stmaryrd}
\usepackage{fancyhdr}
\usepackage[numbers, sort]{natbib}
\RequirePackage[colorlinks,citecolor=blue,urlcolor=blue]{hyperref}
\usepackage{graphicx}
\usepackage{epstopdf}
\usepackage[left=1.1in, right=1.1in, top=1.1in, bottom=1.1in]{geometry}
\usepackage{mathrsfs}  
\usepackage{subfig}
\usepackage{booktabs} % for much better looking tables
\usepackage{array} % for better arrays (eg matrices) in maths
\usepackage{paralist} % very flexible & customisable lists (eg. enumerate/itemize, etc.)
\usepackage{verbatim} % adds environment for commenting out blocks of text & for 
\usepackage{fancyvrb}
\usepackage{float}
\usepackage{caption}
\usepackage{bbm} %For blackboard bold 1

\usepackage{physics}
\usepackage{amsmath}
\usepackage{tikz}
\usepackage{mathdots}
\usepackage{yhmath}
\usepackage{cancel}
\usepackage{color}
\usepackage{array}
\usepackage{multirow}
\usepackage{amssymb}
\usepackage{tabularx}
\usepackage{extarrows}
\usepackage{booktabs}
\usetikzlibrary{fadings}
\usetikzlibrary{patterns}
\usetikzlibrary{shadows.blur}
\usetikzlibrary{shapes}

\usepackage{hyperref}
\usepackage{url}
\usepackage[toc,page]{appendix}
\usepackage{color}

\usepackage[shortlabels]{enumitem}

\definecolor{mygreen}{HTML}{006622}

\theoremstyle{plain}
\newtheorem{thm}{Theorem}[section] % reset theorem numbering for each chapter
\newtheorem{lemma}[thm]{Lemma}

\newtheorem{prop}[thm]{Proposition}
\theoremstyle{definition}

\newtheorem{assumption}[thm]{Assumption}
\theoremstyle{remark}
\newtheorem*{remark}{Remark}
\newtheorem{rmk}[thm]{Remark}
\newtheorem{note}[thm]{Note}
\newtheorem{example}[thm]{Example}

\renewcommand{\a}{\alpha}
\renewcommand{\b}{\beta}
\newcommand{\e}{\varepsilon}

\newcommand{\E}{\mathbb{E}}

\newcommand{\N}{\mathbb{N}}

\newcommand{\R}{\mathbb{R}}

\newcommand{\CC}{\mathcal{C}}

\newcommand{\LL}{\mathcal{L}}
\newcommand{\MM}{\mathcal{M}}

\newcommand{\PP}{\mathcal{P}}
\newcommand{\QQ}{\mathcal{Q}}

\newcommand{\XX}{\mathcal{X}}

\newcommand{\1}{\mathbbm{1}}

\definecolor{mygreen}{HTML}{006622}

\begin{document}

\title[Uniform convergence of numerical methods  for SDEs]{Uniform in time convergence of numerical  schemes for stochastic differential equations via Strong Exponential stability: Euler methods, Split-Step  and Tamed Schemes.}

\author{Letizia Angeli}\address{Maxwell Institute for Mathematical Sciences and Mathematics Department, Heriot-Watt University, Edinburgh EH14 4AS, UK. letizia.angeli92@gmail.com}
\author{Dan Crisan}\address{Department of Mathematics, Imperial College London, London, SW7 2AZ, UK. d.crisan@imperial.ac.uk}
\author{Michela Ottobre}\address{Maxwell Institute for Mathematical Sciences and Mathematics Department, Heriot-Watt University, Edinburgh EH14 4AS, UK. m.ottobre@hw.ac.uk}
%
%\date{July 2024}
%
\begin{abstract} We prove a general criterion providing sufficient  conditions under which a time-discretiziation of a given Stochastic Differential Equation (SDE) is a {\em uniform in time} approximation of the SDE.  The criterion is also, to a certain extent, discussed in the paper, necessary. 
Using such a criterion we then analyse the convergence properties of numerical methods for solutions of SDEs; we consider  
Explicit  and Implicit Euler, split-step and (truncated) tamed
Euler methods. 
In particular, we show that, under mild conditions on the coefficients of the SDE (locally Lipschitz and  strictly monotone), these methods produce approximations of the law of the solution of the SDE that converge {uniformly in time}. The bounds we provide are non-asymptotic.
The theoretical results are verified by  numerical examples.\end{abstract}

\maketitle

%{\hypersetup{linkcolor=black}  \tableofcontents }

\maketitle
\section{Introduction}\label{section_setting}

Stochastic differential equations (SDEs) are popular mathematical models for the evolution of dynamical systems that are randomly perturbed. 
Explicitly solvable SDEs are rare in practical applications, so  one often  needs to rely on numerical approximations in order to be able to study the underlying physical phenomenon. There are two sets of numerical approximations for solutions of SDEs: pathwise (or strong) approximations and approximations of the law of the solution of the SDE, also known as weak approximations. 
  Let  $\{x_t, t\ge 0 \}$ denote the solution of the SDE at hand,  let $\{\mathcal{X}^{\delta}_{t_n}\}_{n \in \mathbb N}$ be some time-discretisation (with step $\delta$, so that $t_n=n\delta$) of $\{x_t, t\ge 0 \}$ and let $\{\mathcal{X}^{\delta}_{t}, t\geq 0\}$ be a continuous-time interpolant of  $\{\mathcal{X}^{\delta}_{t_n}\}_{n \in \mathbb N}$.
An approximating process $\{\mathcal{X}^{\delta }_t, t \geq 0\}$ converges in weak sense with order $\beta>0$ to the solution of the SDE if there exist positive constants $C$ and $\delta _{0}$ such that 
\begin{equation}\label{eqn:boundintro}
\sup_{t \in [0,T]}\left| \mathbb E\left[ g\left( x_{t}\right) \right] -\mathbb E\left[ g\left( \mathcal{X}^{\delta
}_t \right) \right] \right| \leq C(T) \,\delta ^{\beta } \ ,
\end{equation}
for any time partition with maximum step size $\delta \in \left(
0,\delta _{0}\right)$ and for any function  $g$ belonging to a suitably chosen class, see \cite[Chapter 7]{kloedenplaten}; in this paper we will consider the class $\mathcal C_b^2$ of real-valued twice differentiable functions with bounded first and second derivative. In general,  the constant $C=C(T)$ in \eqref{eqn:boundintro} is a function of time and, more often than not, theoretical estimates of the constant depend exponentially on time.  In other words, due to the use of Gronwall Lemma or other similar tools, we often end up producing bounds of the form  \eqref{eqn:boundintro} with $C=C(T)$ a growing function of time. 
In this paper, we consider various approximation schemes for SDEs and we show that, under mild conditions on the coefficients of the SDE (locally-Lipschitz and  strictly monotone drift, we will be more precise below), the bound \eqref{eqn:boundintro} holds with a constant $C$ independent of time; that is, we show that the following holds
\begin{equation}\label{eqn:boundintrouniform}
\sup_{t \geq 0}\left| \mathbb E\left[ g\left( x_{t}\right) \right] -\mathbb E\left[ g\left( \mathcal{X}^{\delta
}_t \right) \right] \right| \leq K \,\delta ^{\beta } \, , \quad g \in \mathcal C_b^2\,,
\end{equation}
where now $K>0$ is a constant independent of time (but dependent on $g$), and $\beta$ will depend on the scheme at hand.\footnote{For the sake of clarity,  $\beta$ does not depend on $t$ nor on $g$. }  When the bound \eqref{eqn:boundintrouniform} holds for some constant $K>0$ we say that the approximating scheme  $\mathcal{X}^{\delta}_t$ is a {\em Uniform in Time} (UiT) weak approximation of the given SDE. 

Bounds of the form \eqref{eqn:boundintrouniform} are {\em non-asymptotic} in time, in the sense that they allow one to control in one go both the transient phase of the dynamics and the long time behaviour. We will be more explicit on this point when we discuss relation to literature in Subsection  \ref{subsec:rel to lit}. 
We will obtain such bounds for SDEs whose drift needs {\em not}  be in gradient form. In particular, the method of proof of this paper does not require any a priori knowledge of the invariant measure of the SDE at hand - our assumptions will imply that such processes have an invariant measure but we never need to use this fact explictly and we never require any knowledge on the invariant measure of the SDE itself nor on the invariant measure of the approximation. Hence the work of this paper provides theoretical guarantees on the approximation of both  the transient and the equilibrium state  of a large number of models.

%%%%%%%%%%%%%%%%%%%%%%%%%%5
%%%%%%%%%%%%%%%%%%%%%%%%%%%%%%%%%%%
%%%%%%%%%%%%%%%%%%%%FROM HERE

Clearly,  uniform in time bounds of the form \eqref{eqn:boundintrouniform} cannot hold in general; whether the bound \eqref{eqn:boundintrouniform} holds or not will depend  both on the dynamics to approximate and on the chosen approximation. In this paper we state sufficient conditions,  both on the SDE and on the approximating scheme,  to guarantee the validity of the UiT bound \eqref{eqn:boundintrouniform}. We will phrase such criteria both in terms of abstract conditions and then in terms of easily verifiable, explicit conditions on the coefficients of the SDE. Note that if \eqref{eqn:boundintrouniform} holds,
one does not need to adapt the time-step during the simulation to keep a given threshold
accuracy. So this line of research is in a different spirit to adaptive time-stepping methods
such as those introduced in \cite{lamberton2003recursive, kelly2018adaptive}.  Furthermore, the solution $x_t$ of the SDE will depend on its initial datum $x_0=x$ --
to emphasise such a dependence we will use the notation $x_t^{(x)}$ -- hence, in general, the constant $K$ on the right hand side (RHS) of \eqref{eqn:boundintrouniform} will depend on (the euclidean norm of) $x$ too. Our results carefully track this dependence, making it possible for the user to fix an appropriate time-step at the start of the simulation.   

\subsection{Main results and relation to literature}\label{subsec:rel to lit}For a clearer description of the results in this paper, let us introduce some essential notation:  let $x_t^{(x)} \in \mathbb \R^N$ be the solution of the SDE
\begin{equation}\label{SDE_ito}
d x^{(x)}_{t} =U_{0}\left(x^{(x)}_{t} \right) d t+\sqrt{2} \sum_{i=1}^{d} V_{i}\left(x^{(x)}_{t} \right) d B_{t}^{i}, \quad x_{0}^{(x)} =x \in \mathbb{R}^{N}\ ,
\end{equation}
where $U_0,V_1,\dots,V_d$ are  vector fields on $\R^N$, $V_i(x)=\left(V_i^1(x),\dots,V_i^N(x)\right)^T$, assumed smooth throughout the paper,  and $B^1_t,\dots,B^d_t$ are 1-dimensional independent standard Brownian motions.

%where $U_0$ denotes the drift term in the corresponding \ito form, i.e.
%\begin{equation}
%    U_{0}^{i}(x)=V_{0}^{i}(x)+\sum_{k=1}^{d} \sum_{j=1}^{N} V_{k}^{j}(x)\, \partial_{j} V_{k}^{i}(x)\ .
%\end{equation}

We denote by $\{\mathcal{P}_{ {t }}\}_{\{t\geq 0\}}$ the associated Markov semigroup, i.e.~the operator defined on measurable bounded functions $f:\R^N \rightarrow \R$ as 
\begin{equation}\label{MarkovSemigroup}
    (\mathcal{P}_{ {t }} f)(x):=\mathbb{E}\left[f\left(x^{(x)}_{t}\right) \right], \quad x \in \mathbb{R}^{N}\, . \footnote{The semigroup associated with a Markov process is generally well defined for bounded and measurable functions $f$. Under our assumptions on the drift of the SDE - see later sections - this semigroup is well defined also for functions in $\mathcal C_b^2$ (which grow linearly). Using the methods of \cite{crisan2022poisson} we could show that this semigroup, as much as the semigroup \eqref{def:discrete semigr} are well defined also for functions $f$ that grow polynomially, but we refrain from doing so here. }
\end{equation}

We first give a brief description of the content of this paper and then comment on the results of each section, in turn.  

\begin{itemize}
    \item In Section \ref{section_general}, we study  the weak error for generic time-discretisation schemes $\mathcal{X}_t^{\delta}$ for  SDEs of the form \eqref{SDE_ito}. The main result of this section,  Proposition \ref{prop_weak}, provides a general, abstract criterion, stating sufficient conditions   for a (any) time-discretisation $\mathcal{X}_t^{\delta}$ to be a UiT approximation of the SDE \eqref{SDE_ito}. While this criterion is an abstract one, all the results of subsequent sections are stated in terms of explicit, easily verifiable and rather general conditions on the coefficients of the SDE. 	
    \item In Section \ref{section_EM}, we apply the abstract criterion of Section \ref{section_general}, Proposition \ref{prop_weak},  to the case in which the numerical approximation is given by an explicit Euler-Maruyama scheme (introduced in equation \eqref{EM}) and the  SDE  \eqref{SDE_ito} has globally Lipschitz coefficients. We show that, in this setting, the explicit Euler-Maruyama scheme is a  UiT approximation of the SDE \eqref{SDE_ito},  see Proposition \ref{prop_EM} for a precise statement. 
    \item In Section \ref{section_split}, we consider split-step (see \eqref{SS2}-\eqref{SS1}) and the Implicit Euler scheme (see \eqref{general_implicit}) for SDEs of the form \eqref{SDE_ito}, when the drift is just locally Lipschitz and strictly monotone, see Assumption \ref{ass_onesidedLip} and Assumption \ref{ass_for_modified_expdecay}, and the diffusion coefficient is bounded. Under these assumptions we provide UiT  estimates for the weak error  associated to such schemes, see Theorem \ref{thm_split} for  split-step and Theorem \ref{thm_implicit} for Implicit Euler.  These are the main results of Section \ref{section_split}. The proofs in this section do not leverage directly the abstract criterion of Section \ref{section_general}; rather,  they use the fact that split-step and implicit Euler (for SDEs with locally Lipschitz drift) can be interpreted  as explicit Euler-Maruyama schemes for a modified SDE, see \cite[Section 3.3]{higham2002strong}. The modified SDE has globally Lipschitz coefficients (however with Lipschitz constants depending on $\delta$, which leads to technical subtleties). Using this fact, the proof is based on  the result of Section \ref{section_EM} for the Euler-Maruyama scheme for SDEs with globally Lipschitz coefficients.  
    \item In Section \ref{section_tamed}, we prove  UiT results for a truncated version of the tamed Euler scheme, see \eqref{tamed}, assuming again that the drift $U_0$ of the SDE \eqref{SDE_ito} is only locally Lipschitz and strictly monotone, see Theorem \ref{thm_weak_tamed}, which is the main result of Section \ref{section_tamed}.  We refer to scheme \eqref{tamed}  as to  {\em truncated tamed Euler} or {\em Newton-tamed Euler}, we explain the reason for this name in Section \ref{section_tamed}. We had derived this scheme independently of other works and  then found out that an analogous truncation of tamed Euler had already been developed in  \cite{sabanis2016euler}. So the scheme we consider here has been introduced (to the best of our knowledge) in \cite{sabanis2016euler}.  We prove UiT convergence for this scheme by applying directly  the abstract criterion of Section \ref{section_general}, Proposition \ref{prop_weak}. The variant of the tamed scheme that we consider appears to be `more stable' than standard tamed, we make more comments on this below and, more in detail, in Section \ref{section_tamed}.
\end{itemize}
Examples of the drift coefficients  satisfying our assumptions can be found in Example \ref{toy_eg} (and in other examples in subsequent sections). Throughout the paper, various remarks are devoted to showing that these running examples satisfy the assumptions of our main results.  

{\bf Relation to literature. } Let us clarify that in this paper we consider  time-discretisation schemes with constant time step $\Delta t = \delta$ (for which the $n$-th step is $t_n=n\delta$), which can be described as a family of discrete-time Markov processes $\{\mathcal{X}^\delta_{t_n}\}_{n\in\N}$, $\delta>0$. The analysis of schemes with varying time-step has been considered e.g.~in \cite{pages2023unadjusted} (for Langevin algorithms), and schemes with random step size have been considered in \cite{alfonsi2021generic}.    The focus of the former is on approximating the invariant measure of the SDE, not the law at time $t$,  for each positive time, which is what we are concerned with here; the focus of the latter is on building higher order schemes (however the estimates there are not necessarily UiT). 

As we have already remarked, the bound \eqref{eqn:boundintrouniform} is uniform in time and  non-asymptotic. To give broader context, let us observe that,  in contrast,   the type of uniform in time  bounds most  often found in the literature are asymptotic (in time), see e.g. \cite{mattingly2002ergodicity, bou2023mixing, schuh2024convergence, duncan2017using, brehier2014approximation} (without any claim to completeness of references) and references therein,  aimed at understanding how well $\mathcal X_t^{\delta}$ approximates the invariant measure $\pi$ of the process $x_t$ (when such an invariant measure exists).   Generally speaking, by `asymptotic' bounds we mean bounds of the form   
\begin{equation}\label{asymptboundintro}
\left|  \mathbb E\left[ g\left( \mathcal{X}^{\delta
}_t \right) - \pi(g) \right] \right| \leq K e^{-\lambda t} + \tilde K \delta^{\hat \beta}
\end{equation}
for some $K, \tilde K, \hat \beta>0$ (and independent of time), though $K, \tilde K, \lambda$ may depend on the approximation parameter $\delta$. Here $\pi(g)$ denotes the integral of $g$ with respect to $\pi$. These bounds are usually obtained by considering the invariant measure of the dynamics $\mathcal X_t^{\delta}$, say $\pi^{\delta}$,  so that, by triangular inequality, the first added on the RHS of \eqref{asymptboundintro} comes from studying exponential convergence of $\mathcal X_t^{\delta}$ to  $\pi^{\delta}$ while the second comes from estimating the asymptotic bias, i.e. the distance between $\pi^{\delta}$ and the invariant measure $\pi$ of $x_t$. Depending on the approximation $\mathcal X_t^{\delta}$ at hand, \footnote{Here we are thinking of general types of approximation, not necessarily just restricting to numerical approximations. } one might not know explicitly how $K$ and $\lambda$ depend on $\delta$. In some instances it might be the case that such constants do not depend on $\delta$ at all, see e.g. \cite{stoltz2018longtime}. Either way the bound \eqref{asymptboundintro} may not imply that the true dynamics $x_t$ and its approximation $\mathcal X_t^{\delta}$ are `close', for $\delta$
small enough, when $t$ is finite, i.e. when both processes are in their transient state. This is not a  problem if one is interested in sampling, but it is an actual limitation if one would like to simulate the whole dynamics (for any $t\geq 0$) accurately -- task which is relevant when one is interested in general modelling e.g. of biological processes rather than in sampling. Moreover bounds of the form \eqref{asymptboundintro} are commonly studied for the case when the invariant measure of the process $x_t$ is known explicitly. In contrast, as we have already remarked, the method of proof in this paper does not require any a priori knowledge on the form of the invariant measure of the process $x_t$, nor on the invariant measure of the approximation $\mathcal X_t^{\delta}$. {\footnote{To be precise, the assumptions we make on the drift coefficients in Section \ref{section_EM}, Section \ref{section_split} and Section \ref{section_tamed} will imply that the process $x_t$ has an invariant measure, but we never need to use the explicit form of such a measure. More comments on this in Note \ref{note_on_main_ass}. } }

For clarity, {we emphasize that in comparing \eqref{eqn:boundintrouniform} with \eqref{asymptboundintro} we primarily mean to compare the form of time-dependence on the RHS of such bounds. In particular we observe that results of the form \eqref{eqn:boundintrouniform}  imply  a bound on the asymptotic bias  (by just letting $t$ to infinity on the LHS of \eqref{eqn:boundintrouniform}) and they also imply that the dynamics $\mathcal X_t^{\delta}$ will  approximate the invariant measure of the SDE correctly (assuming the SDE does admit one, see \cite[Corollary 3.6]{crisan2021uniform} for a more careful statement of this fact).  So, in this sense, and in view of the discussion so far (on the fact that \eqref{asymptboundintro} may fail to capture the transient state),   they are stronger than bounds of the form \eqref{asymptboundintro}. Nonetheless, as far as the approximation of the invariant measure is concerned, the methods employed to prove bounds of the form \eqref{asymptboundintro} may give a better order of convergence with less work, i.e. those proofs may result in  $\hat \beta> \beta$ \cite{talay1990expansion}.}

Let us now comment on the  results of this paper, section by section, relating them to existing literature. 

The abstract criterion we state in Section \ref{section_general} is rather general, and it can be in principle applied to any approximation of SDEs of type \eqref{SDE_ito}, whether produced numerically or not (and indeed Proposition \ref{prop_weak} has recently inspired the developments in \cite{schuh2024conditions}, see in particular the very nice `Informal theorem' therein). Such a criterion shows that, besides a local consistency condition on the approximating scheme, see (\ref{cond_Euler}), one needs two main conditions to hold in order for the UiT bound \eqref{eqn:boundintrouniform} to be satisfied,  namely:  
exponential decay in time of the space-derivatives of the semigroup   $\PP_t$ \eqref{MarkovSemigroup} associated with the SDE \eqref{SDE_ito} (see condition \eqref{cond_SDE}) and some a-priori uniform in time control on the discretisation scheme (see condition \eqref{cond_bound_Phi}). We refer to the first of these two conditions as {\em Strong Exponential Stability} (SES). We explain below the reason for this nomenclature and in Section \ref{section_general} we will be  precise on how many space-derivatives we control. \footnote{We use the term SES somewhat generically, i.e. irrespective of the number of space derivatives for which time-decay is shown. We could be more precise and say `SES of order $k$' if $k$ derivatives are controlled but we refrain from doing so in this introduction.}  Consistently with the earlier remark that the validity of the UiT estimate \eqref{eqn:boundintrouniform} will depend both on the dynamics to approximate and on the approximating scheme, the first of these conditions is a condition on the SDE, the second is a condition on the approximation. 
We will make further technical  remarks on these conditions in Note \ref{note_on_main_ass}. For the time being, we make some more general observations which will hopefully help intuition. It is easy to show (see e.g.~\cite{cass2021long}) that, if SES holds then the semigroup $\mathcal P_t$ enjoys  the following property (hence the name Strong Exponential Stability): 
\begin{equation}\label{H}
\left\vert (\PP_tf)(x) - (\PP_tf)(y) \right\vert =\left \vert 
\mathbb E f(x_t^{(x)}) - \mathbb E f(x_t^{(y)})
\right \vert \leq C e^{-\lambda t}\ ,
\end{equation}
    for some constants $ \lambda>0$ and $C>0$, the latter generally dependent on $x,y \in \R^N$.  
That is, if we start the dynamics from two different initial conditions, then the associated laws will converge to each other {exponentially fast}. We note in passing that this is related to the concept of pullback attractor, see e.g.~\cite{kloedenplaten}, and to the definition of Lyapunov exponent, see \cite{castro2022lyapunov} and references therein. 
  The second and third authors of this paper have pushed a programme to show that SES is conceptually key and technically a flexible tool to produce uniform in time results, whether the approximation is produced numerically or not, and indeed  SES has been successfully used as a key ingredient in \cite{crisan2022poisson}, where the approximation is produced via averaging methods; in \cite{barre2021fast}, where the approximation is given by a multi-scale particle system;  and in \cite{crisan2021uniform} for other numerical approximations. So, within this stream of literature, the general purpose of this paper is to show that this approach is successful also when considering a larger class of (numerical) approximations.  Furthermore, to make this tool ready for practical use, we point out that  the literature by now contains a number of user-friendly criteria  for SES, see e.g.~\cite{lorenzi2006analytical, crisan2016pointwise,crisan2021uniform, cass2021long} and references therein and in Section \ref{s:s} we will provide further criteria for SES.  

Regarding the second condition, i.e.~the UiT a priori control on the numerical scheme,  this is gained in this paper through UiT estimates for the moments  of the considered discretisations, see Lemma \ref{standard_eg} and Note \ref{note_on_main_ass} on the matter. Conquering such estimates is one of the technical difficulties we tackled in this paper. We flag up the reference \cite{tretyakov2013fundamental}, which also contains an interesting scheme of proof for  obtaining uniform in time moment bounds.  \\
Further comments on these two conditions, on the extent to which they are `necessary' and on relations to other frequently considered sets of assumptions have been made in \cite{crisan2021uniform}, so we don't repeat here this kind of observations, however see Note \ref{note_on_main_ass} for further comments. 
We also point out that a different approach to proving UiT results has been proposed in \cite{del2022backward}. 

 The main result of Section \ref{section_EM}, Proposition \ref{prop_EM}, shows UiT convergence of the explicit Euler scheme, i.e.~the validity of the bound \eqref{eqn:boundintrouniform} when $\mathcal{X}^{\delta}$ is the explicit Euler scheme,  assuming the coefficients of the SDE are globally Lipschitz. Besides perhaps some intrinsic interest, this result is for us a necessary technical step in view of the strategy of proof that we adopt in Section \ref{section_split}.  Proposition \ref{prop_EM} can also be seen as an extension of  \cite{crisan2021uniform}, where the authors  prove the UiT result \eqref{eqn:boundintrouniform}  for explicit Euler in the case in which the coefficients of the SDE are bounded. However there is an important difference between  Proposition \ref{prop_EM} and \cite{crisan2021uniform}: the result in \cite{crisan2021uniform}, while formulated under more restrictive assumptions, is optimal, in the sense that \eqref{eqn:boundintrouniform}  is not only shown to hold for a constant $K>0$ independent of  time, but also for weak order  $\beta=1$ (which is optimal for explicit Euler). Proposition \ref{prop_EM} holds under a more general assumption on the coefficients but shows only order of convergence $\beta=1/2$. This is due to the fact that in \cite{crisan2021uniform} the authors controlled four space-derivatives of the Markov semigroup, whereas here we `stop the expansion' at the second derivative and control only the first two derivatives. As our main focus is on the UiT aspects of the estimates, this strikes a compromise between simplicity (shortening proofs) and optimality. However, at the price of controlling more derivatives, the method we use could still be employed to gain results that are optimal in $\beta$. Finally, it is known that when the coefficients of the SDE are non globally Lipschitz, explicit Euler `blows up' (for large initial data), see \cite[Section 6.3]{mattingly2002ergodicity}, hence the global Lipschitz assumption in  Proposition \ref{prop_EM} cannot be further relaxed.

As for literature related to the content of Section \ref{section_split}, in \cite{liu2022backward} the authors provide UiT strong error estimates for implicit Euler, with order of convergence $\beta=1/2$, see \cite[Theorem 4.2]{liu2022backward}. The set of criteria under which they work is in spirit similar to ours (though not directly comparable because a specific structure of the drift is assumed there), though their assumptions on the diffusion coefficient are more general. For convenience, we assumed the diffusion coefficient to be bounded throughout; this could be relaxed within our scheme of proof, we don't do so again to keep  the length of proofs at bay. However our result keeps track of the dependence of the constant $K$ in \eqref{eqn:boundintrouniform} on the size of the initial datum; such a dependence is not studied in \cite{liu2022backward}.

In \cite{buckwar2022splitting} the authors illustrate how the choice of the initial state $X_0=x\in\R^N$ influences the behaviour of different Euler schemes, under one-sided Lipschitz conditions on the coefficients of the SDE. In particular, they show that the standard tamed scheme is very sensitive to the size of the initial datum (and for this reason introduce splitting schemes which ameliorate this problem and are more stable in this sense). %\textcolor{mygreen}{I am referring at \cite[Figure 2, p.22]{buckwar2022splitting}: it seems to me that in their work they only compare splitting methods with tamed Euler methods, not implicit. Moreover, they consider standard tamed Euler schemes \eqref{standard_tamed}, not truncated. So, they are just pointing out our same issues with the standard tamed. In our work, we propose a different tamed Euler scheme to overcome this issue (in particular, our result is not sub-optimal).} 
We also  observed this behaviour for standard tamed Euler schemes, as illustrated in Section \ref{section_tamed} and, with this motivation, we consider the truncated tamed scheme  \eqref{tamed} which is significantly less sensitive to the choice of initial datum, as shown both by Figure \ref{fig_tamed} and by Theorem \ref{thm_weak_tamed}.  Related to this fact, we point out that, under analogous conditions to those we impose in Section \ref{section_tamed},
\cite{brehier2020approximation} provides weak error estimates for tamed Euler schemes, with a constant $C(t)$ in \eqref{eqn:boundintro} which grows at most 
polynomially in time. We could not improve on these bounds (and we are not sure whether it is possible to achieve such an improvement).  However we explain why the truncated tamed scheme is more stable than the standard tamed scheme and prove a UiT result for the latter  scheme, showing the (good) dependence of the weak error on the size of the initial datum.

%In \cite{liu2022backward}, the authors provide time-uniform strong error estimates for implicit Euler schemes with order of convergence $\delta^{1/2}$ \citep[see][Theorem 4.2]{liu2022backward}. The result is proved for implicit Euler schemes only.   They work under a similar set of criteria as ours (Assumption \ref{ass_onesidedLip}), the main difference is that they don't require $V_k$ bounded, and introduce weaker assumptions for $V_k$. {\color{blue}M: shall we say here that this is achievable also using our approach, we don't do it to keep the length of proofs at bay, but we make remarks on this in Note...}\\

\section{A general result for uniform weak Convergence}\label{section_general}

In this section, we present a general approach for proving the uniform in time weak convergence of time-discretisation schemes. The main result of this section, Proposition \ref{prop_weak}, will be applied in subsequent sections to prove the weak convergence of various numerical schemes. We first state the proposition, then make a number of comments on the assumptions under which it holds (Assumption \ref{ass_weak}), see Note \ref{note_on_main_ass} and Lemma \ref{standard_eg}, and postpone the proof to later in the section.    

Let $x_t$ be the solution of \eqref{SDE_ito} with associated Markov semigroup $\PP_t$ (see \eqref{MarkovSemigroup}). For each $\delta>0$,   let $\{\mathcal{X}^\delta_{t_n}\}_{n\in\N}$ be a  discrete-time Markov processes on  $\R^N$, $\{\mathcal{X}^\delta_{t_n}\}_{n\in\N} \subset \R^N$,   with time step $\delta$ (so that $t_n=n\delta$) and with associated  discrete-time Markov semigroup $\MM^\delta_n$, defined as
\begin{equation}\label{def:discrete semigr}
    \MM^\delta_n f(x)\,:=\, \E_x\left[f(\mathcal{X}^\delta_{t_n})\right]\ ,\quad x\in\R^N\ ,\; f \in \mathcal C_b^2(\R^N),
\end{equation}
where in the above and throughout $\E_x$ denotes expectation with respect to the law of the process, conditional on the process starting at $x \in \R^N$, $\mathcal X^{\delta}_0=x$.
In what follows, we denote by $\nabla f$ and $\nabla^2 f$ respectively the gradient vector and the Hessian matrix of a function $f$ in $\CC^2(\R^N)$. Furthermore, we consider the seminorm
\begin{equation}\label{norm}
   \|f\|_{\CC^2_b} := \sup_x\left(\left| \nabla f(x)\right|\,+\,\left\| \nabla^2  f(x)\right\|\right)\ ,
\end{equation}
on the space of functions $\CC^2_b(\R^N)$, where $\|A\|:=\sqrt{\sum_{i,j}|a_{ij}|^2}$ denotes the Frobenius norm of a matrix.

\newcommand{\laa}{\tilde{\lambda}}
\begin{assumption}\label{ass_weak}
 
We assume the following:
\begin{itemize}
\begin{subequations}
    %\item There exists a constant $K_1>0$ independent of $\delta$ and $t$ such that
    %\begin{equation}\label{discretisation_moment_bound}
    %\sup_{t\geq 0}\E_x\left[|\mathcal{X}^\delta_{t}|\right]\,\leq\,K_1\,\left(1+|x|\right)\ ;
    %\end{equation}
    \item[(i)] {\em Strong Exponential Stability}: there exist constants $K_0,\,\laa>0$ such that
    \begin{equation}\label{cond_SDE}
        \|\PP_t f\|_{\CC^2_b}\,\leq\,K_0\, \|f\|_{\CC^2_b}\cdot e^{-\laa t}\ ,
    \end{equation}
    for any $t\geq 0$, and $f\in\CC_b^2(\R^N)$;
    \item[(ii)] {\em Local consistency and a-priori control}: there exist  positive functions $\phi, \Phi: \R^N\times \R^+\to [0,+\infty)$ and a constant $K_1>0$ such that the following two conditions hold:
    \begin{equation}\label{cond_Euler}
        \left|\E_x\left[f(\mathcal{X}^\delta_{\delta})\,-\, f(x_\delta)\right]\right|\,\leq\, K_1\|f\|_{\CC^2_b}\cdot\phi(x,\delta)\ ,
    \end{equation}
    for every $x\in\R^N$, $f\in\CC_b^2(\R^N)$, and $\delta>0$, and 
    \begin{equation}\label{cond_bound_Phi}
        \sup_{n\in\N} \E_x[\phi(\mathcal X^{\delta}_{t_n},\delta)] = \sup_{n\in\N}\,\left(\MM^\delta_n \phi\right)(x,\delta)\,\leq\,\Phi(x,\delta)\ ,
    \end{equation}
    for every $x\in\R^N$ and $\delta>0$ small enough. 
\end{subequations}
\end{itemize}

\end{assumption}

In the above an throughout, by  $(\MM^{\delta}\phi)(x,\delta)$ we mean the  semigroup $\MM^{\delta}$ applied to the function $y\rightarrow \phi(y, \delta)$ (with $\delta$ viewed as fixed). We could have simply used the notation $(\MM^{\delta}\phi)(x)$ instead of 
$(\MM^{\delta}\phi)(x,\delta)$; we use the latter to emphasize that this quantity depends on $\delta$ both through $\MM^{\delta}$ and through $\phi(x,\delta)$. 

\begin{prop}\label{prop_weak}
%Let $\{\mathcal{X}^\delta_{t_n}\}_{n\in\N}$ be a family of discrete-time Markov processes with Markov semigroup denoted by $\MM^\delta_n$ and let $x_t$ be the solution of the SDE \eqref{SDE} with Markov semigroup denoted by $\PP_t$. 
With the notation introduced so far, under Assumption \ref{ass_weak}, the following bound holds for any $f\in\CC^2_b(\R^N)$ and $\delta>0$ small enough:
\begin{equation}\label{weak_conv}
    \sup_{n\in\N} \,\left| \E_xf(\mathcal{X}^{\delta}_{t_n})-\E_x f(x_{t_n})\right|\,\leq\, \frac{K\|f\|_{\CC^2_b}\cdot \Phi(x,\delta)}{1-e^{-\laa\,\delta}}\  ,
\end{equation}
with $K:=K_1( K_0\vee 1 )$, where $K_0, K_1$ and $\laa$ are as in Assumption \ref{ass_weak}.
\end{prop}

In the note below we make comments to help the reader interpret the conditions in Assumption \ref{ass_weak};  in Lemma \ref{standard_eg} we state a criterion to check \eqref{cond_bound_Phi} in practice. We will discuss conditions under which SES (the bound \eqref{cond_SDE}) holds in Section \ref{s:s}.  Examples of SDEs and approximation schemes satisfying these assumptions are provided throughout the paper, see e.g. Example \ref{toy_eg} in the next section. 

\begin{note}\label{note_on_main_ass}
Some observations:
\begin{itemize}
\item Assumption \ref{ass_weak} (i) is an assumption on the SDE itself, while Assumption \ref{ass_weak} (ii) is an assumption on the discretization.  For the sake of clarity we emphasize that condition \eqref{cond_bound_Phi} does not follow from \eqref{cond_Euler}, these are two distinct requirements. The positive function $\phi$ from condition \eqref{cond_Euler} encodes the bound for the local weak error of the time approximation (local in the sense that it is the weak error after one time-step); clearly, such an error depends both on the initial state $x\in\R^N$ and on the time-step $\delta>0$. The schemes we consider in this paper satisfy   \eqref{cond_Euler}   with functions $\phi$ of the form $\phi(x,\delta)=\delta^\a |x|^q +\delta^\beta$, for some constants $q,\a,\beta>0$. In this case, condition \eqref{cond_bound_Phi}  boils down to requiring time-uniform $q$-moment bounds. A practical way of checking that \eqref{cond_bound_Phi} holds is by using  Lemma \ref{standard_eg} below. Using such a lemma  one obtains  that, if $\phi$ is as in the above (and if the assumption of the lemma is satisfied), then $\Phi$ (appearing in \eqref{cond_bound_Phi}) is substantially  of the form $\Phi(x,\delta) = \delta^\a |x|^q +\delta^{\gamma} $, for some $\gamma>0$. 
Since the term $1-e^{-\laa \delta}$ at the denominator of \eqref{weak_conv} is 
$o(\delta)$, if $\a, \gamma>1$ then Proposition \ref{prop_weak} yields the desired UiT result. Note that the term $1-e^{-\laa \delta} \approx o(\delta)$  makes one lose one power of $\delta$, which is in fact expected in going from local to global error. 

    \item The difference $\E_xf(\mathcal{X}^{\delta}_{t_n})-\E_x f(x_t)$ in \eqref{weak_conv} can be written  as the telescopic sum of $n$ contributions, namely
    \begin{equation}\label{basic_idea}
       \E_xf(\mathcal{X}^{\delta}_{t_n})-\E_x f(x_t)\,=\, \sum_{k=1}^n \left( \E_xf(\mathcal{X}^{\delta,k}_{t_n})-\E_x f(\mathcal{X}^{\delta,k-1}_{t_n})\right)\ ,
    \end{equation}
    where $\mathcal{X}^{\delta,k}_{t_n}$ denotes the process that evolves according to the time discretisation until time $t_k$ and then evolves according to the SDE \eqref{SDE_ito} from time $t_k$ to time $t_n$; in particular, $\mathcal{X}^{\delta,0}_{t_n}=x_{t_n}$ and $\mathcal{X}^{\delta,n}_{t_n}=\mathcal{X}^{\delta}_{t_n} $. {We don't use explicitly the decomposition \eqref{basic_idea} in the proof of Proposition \ref{prop_weak} (as written,  this decomposition is more pathwise in nature than we will need), but we find it to be a useful way of explaining each of the requirements in Assumption \ref{ass_weak} and the intuition behind  the proof of Proposition \ref{prop_weak},  which substantially requires bounding each of the addends  in \eqref{basic_idea} and showing that \eqref{basic_idea} is a convergent series.} \\
     The processes $\mathcal X_{t_n}^{\delta,k}$ and $\mathcal X_{t_n}^{\delta,k-1}$ coincide up to time $t_{k-1}$ (included), as up to $t_{k-1}$ they are  copies of the time-discretisation 
     $\mathcal{X}^\delta_{t_n}$, but in general differ for any time $t>t_{k-1}$. Bounding each of the terms $\E_xf(\mathcal{X}^{\delta,k}_{t_n})-\E_x f(\mathcal{X}^{\delta,k-1}_{t_n})$ requires understanding the difference between such processes at time $t_n$. To understand such a difference observe that on the time interval $[t_{k-1}, t_k]$ the process $\mathcal X_{t_n}^{\delta,k-1}$ evolves according to the SDE \eqref{SDE_ito} (with initial condition at time $t_{k-1}$ equal to $\mathcal X_{t_{k-1}}^{\delta}$) while $\mathcal X_{t_n}^{\delta,k}$ evolves according to the time-discretization. From time $t_k$ on they both evolve according to the SDE, but with different initial conditions at time $t_k$, namely  $\mathcal{X}^{\delta,k-1}_{t_k}$ and $\mathcal{X}^{\delta}_{t_k}$, respectively.  In view of the observation \eqref{H},  this is where SES,  condition \eqref{cond_SDE}, comes into play, with the exponential decay ensuring summability. However, the constant $C$ in \eqref{H} depends on the distance between  the initial conditions $x$ and $y$, which are $\mathcal{X}^{\delta,k-1}_{t_k}$ and $\mathcal{X}^{\delta}_{t_k}$ in the case at hand. That is, using \eqref{H} and the definition of Markov semigroup,  for $n\geq k$, we have
$$\left\vert \mathbb E f(\mathcal X_{t_n}^{\delta,k}) -  \mathbb E f(\mathcal X_{t_n}^{\delta,k-1})\right\vert \leq C( \mathcal X_{t_k}^{\delta}, \mathcal X_{t_n}^{\delta,k-1}) e^{-\laa (t_n-t_k)} \,,$$
where again $C( \mathcal X_{t_k}^{\delta}, \mathcal X_{t_n}^{\delta,k-1})$ is a constant depending on the distance between $ \mathcal X_{t_k}^{\delta} $ and $ \mathcal X_{t_n}^{\delta,k-1}$. To estimate such a distance, we observe that, 
by definition of the two processes on the time interval $[t_{k-1}, t_k]$ the (expected value of the) difference between $\mathcal{X}^{\delta}_{t_k}$ and $\mathcal{X}^{\delta,k-1}_{t_k}$ is precisely the local weak error, i.e. the error over one time step, when both the discretization and the SDE are initialised at $\XX^\delta_{t_{k-1}}$;  by \eqref{cond_Euler},  this error is bounded by $\phi(\XX^\delta_{t_{k-1}},\delta)$. Hence, in order for \eqref{basic_idea} to be converging, we enforce $\phi(\XX^\delta_{t_{k-1}},\delta)$ to be independent of $k\in\N$: condition \eqref{cond_bound_Phi} ensures 
    we can bound these errors uniformly in $k\in\N$.   
\item The above discussion gives one way of understanding the role of  SES, condition \eqref{cond_SDE}, within our scheme of proof.  A different and more general way of intuitively interpreting SES has been presented in \cite{crisan2022poisson, crisan2021uniform} and we will briefly recall this different interpretation here, for completeness. 
To explain in a simplified setting why SES is key to proving uniform in time estimates, let us consider two Markov semigroups, say $\mathcal T_t$ and $\bar{\mathcal {T}}_t$. With standard manipulations, the difference between such semigroups can be expressed in terms of the difference between their respective generators, say $\mathcal G$ and 
			$\bar{\mathcal G}$, as follows
			\begin{align*}
				(\bar{\mathcal {T}}_t \varphi)(z) -  (\mathcal{T}_t \varphi)(z) & =  \int_0^t ds \frac{d}{ds} \mathcal{T}_{t-s}\bar{\mathcal{T}}_s \varphi (z)  =  
				\int_0^t \!\!\!ds \,  \mathcal{T}_{t-s} (\bar{\mathcal G} - \mathcal G) \bar{\mathcal T}_s \varphi (z) \\
				&\leq \int_0^t ds \| \mathcal T_{t-s} (\bar{\mathcal G} - \mathcal G) \bar{\mathcal T}_s \varphi \|_{\infty}
				\leq \int_0^t ds \|  (\bar{\mathcal G} - \mathcal G) \bar{\mathcal T}_s \varphi \|_{\infty} \,.
			\end{align*}
		If $\mathcal G$ and $\bar{\mathcal{G}}$ are differential operators then the latter difference involves derivatives of the semigroup $\bar{\mathcal T}_t$. If such derivatives decay exponentially fast in time, then the difference between such semigroups can be estimated by a constant (independent of time) rather than with exponential growth, which is what would happen by using Gronwall-type arguments. This line of reasoning, when applied to $\PP_t$ and $\MM_t^{\delta}$ rather than $\bar{\mathcal T}_t$ and $\mathcal T_t$, inspires our approach - though, as we have already explained,  the precise proof does not  follow the above calculation.

     %evolve with a different dynamics from  on the time step $[t_{k-1}, t_k]$.  Condition \eqref{cond_SDE} ensures that the weak error of two solutions of the SDE \eqref{SDE_ito}, with initial state respectively $x$ and $y\in\R^N$, converges exponentially fast as $t\to\infty$ .  This is a necessary requirement because, we need the quantity , with fixed $k\in\N$, to converge exponentially fast to $0$ as $n\to\infty$, so to have \eqref{basic_idea} bounded uniformly in $n\in\N$. The processes $\mathcal{X}^{\delta,k}_{t_n}$ and $\mathcal{X}^{\delta,k-1}_{t_n}$ evolve according to the original SDE \eqref{SDE_ito} from time $t_k$, but at $t_k$ they are at different states (namely $\mathcal{X}^{\delta,k}_{t_k}$ and $\mathcal{X}^{\delta,k-1}_{t_k}$). Hence we need this error at $t=t_k$ to decrease exponentially fast as $t\to\infty$.
%\item  Using the notation of the previous bullet points, 
    %the processes $\mathcal{X}^{\delta,k}_{t_n}$ and $\mathcal{X}^{\delta,k-1}_{t_n}$ are Hence, by \eqref{cond_Euler}, the local weak error $\left|\E\left[\,f(\mathcal{X}^{\delta,k}_{t_k})-f(\mathcal{X}^{\delta,k-1}_{t_k})\,\Big|\,\mathcal{X}^\delta_{t_{k-1}}\right]\right|$ 
    
    \item The conditions of Proposition \ref{prop_EM} are substantially necessary,  in the sense that they could be quantitatively relaxed (e.g. the exponential decay in \eqref{cond_SDE} can be replaced by any other integrable in time decay), but they are qualitatively necessary.  In \cite{crisan2021uniform} the authors used SES and control on moments of the SDE to prove that the Explicit Euler scheme provides a UiT approximation of the SDE (when the SDE has bounded coefficients). In that setting they showed (constructively, through examples and counterexamples) that one assumption does not imply the other and made several observations on the extent to which these assumptions are necessary. Such observations apply more broadly here, but we don't repeat them and refer to \cite{crisan2021uniform}. However, as further corroborating evidence for the key role of SES type of conditions, we do recall that Lyapunov conditions have been shown to be in general neither sufficient nor necessary in order to obtain uniform in time control \cite{crisan2021uniform, mattingly2002ergodicity}. The relation between SES and Lyapunov conditions has been investigated in \cite{crisan2021uniform}. 
    \item {SES, i.e. Assumption \ref{ass_weak} (i), does imply the contractivity property \eqref{H}. However it does not imply that the process $x_t$ admits an invariant measure. It only implies that, if the invariant measure exists, then it is unique.  This has been shown in  \cite{cass2021long} and discussed in \cite{crisan2021uniform} as well. Nonetheless the explicit assumptions on the coefficients of the SDE that we will enforce in subsequent sections in order to prove either SES or the a priori control on the approximation do imply that the process has a unique invariant measure. So one might conclude that the use we make in subsequent sections of Proposition \ref{prop_weak} is suboptimal. Proving UiT convergence for processes which do not have an invariant measure seems to be a big open problem, to the best of our knowledge. }
\end{itemize}
\end{note}

\begin{lemma}\label{standard_eg}
Assume that condition \eqref{cond_Euler} is satisfied, and that the function $\phi(x,\delta)$ defined in \eqref{cond_Euler} can be written in the form $\phi(x,\delta)=\delta^\a\, g(x)+\delta^\b$, for some $\a,\b>0$ and $g:\R^N\to\R$ such that
\begin{equation}\label{M1}
    \MM^\delta_1 g(x) \leq \e\, g(x)+c\ ,
\end{equation}
for some $\e=\e_{\delta}\in(0,1)$ and $c=c_{\delta}>0$ (both constants possibly depending on $\delta$). Then, condition \eqref{cond_bound_Phi} is satisfied with
\begin{equation*}
    \Phi(x,\delta)\,=\, \delta^\a\,\,g(x)\,+\, \frac{c_\delta\, \delta^\a}{1-\e_\delta}\,+\, \delta^\beta\ .
\end{equation*}

\end{lemma}
The proof of Lemma \ref{standard_eg} is standard, and is achieved by iteratively applying  the inequality \eqref{M1}, so we omit it.

\begin{proof}[Proof of Proposition \ref{prop_weak}]
Since $\E_xf(\mathcal{X}^{\delta}_{t_n})-\E_x f(x_{t_n})=\MM^\delta_n f(x)-\PP_{n\delta} f(x)$, let us consider the basic identity
\begin{align*}
     \MM^\delta_n f(x)-\PP_{n\delta} f(x)=\underbrace{ \MM^\delta_n f(x)-\left(\MM^\delta_m \PP_{(n-m)\delta}f\right)(x)}_{=:A_{m,n}}+\underbrace{ \left( \MM^\delta_m \PP_{(n-m)\delta}f\right)(x) -\PP_{n\delta}f(x) }_{=:B_{m,n}}\ ,
\end{align*}
for any $m\leq n$. The first term is related to the local weak error, whereas the second is essentially a term of the form \eqref{H} and is bounded by using the SES condition. Choosing $m=1$, by \eqref{cond_Euler} and then \eqref{cond_SDE}, we have
\begin{align}
    \left|B_{1,n}\right|\,=\, \left|\E_x\left[ \PP_{(n-1)\delta}f(\mathcal{X}^\delta_{t_1})\,-\, \PP_{(n-1)\delta}f(x_\delta)\right]\right|\,&\leq K_1\left\|\PP_{(n-1)\delta}f\right\|_{\CC^2_b}\cdot \phi(x,\delta)\nonumber\\
    &\leq\,K_1\,K_0\,\|f\|_{\CC^2_b}\cdot\phi(x,\delta) \,e^{-\laa (n-1)\delta}\ .\label{B1n_bound}
\end{align}
Moreover, for $n=2$, $m=1$, by \eqref{cond_Euler} we have 
\begin{align*}
    \left|A_{1,2}\right|\,=\, \left|\E\left[\MM^\delta_1 f(\XX^\delta_{t_1}) -\PP_\delta f(\XX^\delta_{t_1})\,\big|\,\XX^\delta_0=x\right]\right|\,&\leq\, K_1\|f\|_{\CC^2_b}\,\E_x\left[\phi\left(\XX^\delta_{t_1},\,\delta\right)\,\big|\,\XX^\delta_0=x\right]\\
    &=\,K_1\|f\|_{\CC^2_b}\,\MM^\delta_1\phi\left(x,\,\delta\right)\ .
\end{align*}
Thus, for $n=2$, we obtain
\begin{align}
    \left| \MM^\delta_2 f(x)\,-\,\PP_{2\delta} f(x)\right|\,&=\,\left|A_{1,2}+B_{1,2}\right|\nonumber\\
    &\leq K\,\|f\|_{\CC^2_b}\cdot\left(\MM^\delta_1\phi\left(x,\,\delta\right)\,+\, \phi(x,\delta) \,e^{-\laa \delta}\right)\nonumber\\
    &= K\,\|f\|_{\CC^2_b}\sum_{k=0}^1\MM^\delta_{k}\phi(x,\delta)\,e^{-\laa(1-k)\delta}  \ ,\label{basecase}
\end{align}
where $M^\delta_0\phi=\phi$ and $K=K_1(K_0\vee 1)$. Now, by induction on $n\in\N$, $n> 2$, assuming 
\begin{equation}
    \left| \MM^\delta_{n-1} f(x)\,-\,\PP_{(n-1)\delta} f(x)\right|\,
    \leq\,K\,\|f\|_{\CC^2_b}\,\sum_{k=0}^{n-2}\MM^\delta_{k}\phi(x,\delta)\,e^{-\laa(n-2-k)\delta} \ ,\label{inductive_hp}
\end{equation}
for any $x\in\R^N$ (which is satisfied for the base case $n=2$, as shown in \eqref{basecase}), we show that the following holds
\begin{equation}
    \left| \MM^\delta_{n} f(x)\,-\,\PP_{n\delta} f(x)\right|\,
    \leq\,K\,\|f\|_{\CC^2_b}\,\sum_{k=0}^{n-1}\MM^\delta_{k}\phi(x,\delta)\,e^{-\laa(n-1-k)\delta} \ ,\label{induction}
\end{equation}
for any $x\in\R^N$. Indeed, by the inductive assumption \eqref{inductive_hp}, $A_{1,n}$ is bounded by
\begin{align*}
    \left|A_{1,n} \right|\,&=\, \left|\E\left[\MM^\delta_{n-1} f(\XX^\delta_{t_1}) -\PP_{(n-1)\delta} f(\XX^\delta_{t_1})\,\big|\,\XX^\delta_0=x\right]\right|\\
    &\leq\, K\|f\|_{\CC^2_b}\,\E\left[ \sum_{k=0}^{n-2}\MM^\delta_{k}\phi(\XX^\delta_1,\delta)\,e^{-\laa(n-2-k)\delta} \,\bigg|\,\XX^\delta_0=x\right]\\
    &=\, K\|f\|_{\CC^2_b}\sum_{k=1}^{n-1} \MM^\delta_{k}\phi(x,\delta)\,e^{-\laa(n-1-k)\delta} \ ,
\end{align*}
while $B_{1,n}$ is bounded by \eqref{B1n_bound}, thus we obtain \eqref{induction}.

To conclude, it is enough to observe that, under condition \eqref{cond_bound_Phi},
\begin{align*}
    \sum_{k=0}^{n-1}\MM^\delta_{k}\phi(x,\delta)\,e^{-\laa(n-1-k)\delta} \,\leq\, \Phi(x,\delta)\, \sum_{k=0}^{n-1} e^{-\laa\,k\delta}\,\leq\,  \,\frac{ \Phi(x,\delta)}{1-e^{-\laa\,\delta}}\ .
\end{align*}

\end{proof}

%\begin{figure}[h]
%    \centering
%    \includegraphics[scale=0.4]{pictures/proba.png}
%    \caption{Illustration of the key idea of proof of Proposition \ref{prop_weak}. The blue line represents the original process and the red line represents the time-discretisation. ....}
%    \label{fig:my_label}
%\end{figure}

\subsection{Sufficient conditions for Strong Exponential Stability (condition (\ref{cond_SDE}))}\label{s:s}
A number of criteria for SES have been presented in \cite{lorenzi2006analytical,crisan2016pointwise,crisan2021uniform, crisan2022poisson} and references therein. Here we provide further criteria, namely Lemma \ref{lemma_expdecay1} and Lemma \ref{lemma_expdecay2}, expressed in a way which is more suited to the use we make of them in this paper. 
The proofs of the criteria below can be found in the appendix, and follow the structure of  \cite[Prop. 4.5]{crisan2022poisson} and \cite[Proposition 7.1.5]{lorenzi2006analytical}.

\begin{example}\label{toy_eg}
To illustrate the assumptions presented in this and the next sections, we consider the following two basic examples:
    \begin{itemize}
        \item for $N=d=1$, $U_0(x)=-x^3 -a\,x$ and $V(x)=b\,\arctan(x)$, for some constants $a,b>0$.
        \item for $N=d=2$, $$U_0(x_1,\,x_2)= \binom{- x_1 - a x_1^3-x_2^2 x_1}{-x_2-bx_2^3-x_1^2 x_2 }\ ,$$
        with $a,b>0$, and $V_1(x_1,\,x_2)=\sigma_1\in\R^2,\;V_2(x_1,\,x_2)=\sigma_2\in\R^2$ constant in $(x_1,\,x_2)\in\R^2$.
    \end{itemize}
\end{example}

\begin{assumption}\label{ass_expdecay}
\begin{subequations}
Assume that there exists a positive function $\lambda:\R^N\to\R$ such that $\lambda(x)>0$ for every $x\in\R^N$, $\inf_{x\in\R^N} \lambda(x) >0$, and
\begin{align}\label{cond_nablaU0}
    v^T\, \nabla U_0(x)\, v
    \,\leq \, -\lambda(x)\, |v|^2\ ,
\end{align}
for every $x,v\in\R^N$. Moreover, assume  
\begin{equation}\label{cond_expdecay_V}
\sup_{i=1,\dots,N}\,\sum_{k=1}^d \,|\partial_i V_k(x)|^2 \,\leq \,\frac{1}{N}\left(\lambda(x)-\gamma\right)\ ,\qquad x\in\R^N\ ,
\end{equation}
for some constant $\gamma>0$ independent of $x\in\R^N$.
\end{subequations}
\end{assumption}

\begin{rmk}\label{eg2}
    We show that Assumption \ref{ass_expdecay} is satisfied by the vector fields of Example \ref{toy_eg}. In the first example, we have $U_0'(x)=-3x^2-a$, hence \eqref{cond_nablaU0} is satisfied with $\lambda(x)=3x^2 +a$, provided $a>0$. Moreover, $V'(x)=\frac{b}{1+x^2}$, thus condition \eqref{cond_expdecay_V} is satisfied for any $0<\gamma\leq 3x^2+a-\frac{b}{1+x^2}$, if $0<b<a$.

    In the second example, we have
    \begin{align*}
        \nabla U_0 (x_1,x_2)\,=\, 
        \begin{pmatrix}
            -1-3a x_1^2-x_2^2 &- 2x_1x_2\\
            -2x_1x_2&-1-3b x_2^2-x_1^2
        \end{pmatrix}
        \ .
    \end{align*}
    We can note that it is a symmetric matrix with eigenvalues given by
    \begin{equation*}
        -\frac{1}{2}\left(2+(3a+1)x_1^2+(3b+1)x_2^2\right)\,\pm\,\frac{1}{2}\sqrt{\left((3a-1)x_1^2-(3b-1)x_2^2\right)^2+16x_1^2 x_2^2}\ .
    \end{equation*}
    If $a,b>0$ are such that
    \begin{align*}
        (1+3ax_1^2+x_2^2)(1+3bx_2^2+x_1^2) \geq 4x_1^2 x_2^2\ ,
    \end{align*}
    we have that the eigenvalues of $\nabla U_0(x_1,x_2)$ are real and negative, for every $(x_1,x_2)\in\R^2$.
    Thus, condition \eqref{cond_nablaU0} is satisfied with
    \begin{equation*}
        \lambda(x_1,x_2) = 1+\frac{1}{2}\left((3a+1)x_1^2+(3b+1)x_2^2\right)\,-\,\frac{1}{2}\sqrt{\left((3a-1)x_1^2-(3b-1)x_2^2\right)^2+16x_1^2 x_2^2}\ .
    \end{equation*}
    Moreover, since $V_1,\,V_2$ are constant, condition \eqref{cond_expdecay_V} is also satisfied with $\gamma=\min_{x\in\R^2} \lambda(x)$, which is strictly positive (for suitable values of $a,b$).
\end{rmk}

\begin{lemma}\label{lemma_expdecay1}
    Under Assumption \ref{ass_expdecay}, there exists a constant $\gamma>0$ such that
    \begin{equation*}
         \|\nabla \PP_t f\|^2_{\infty}\,\leq\,e^{-2\gamma t}\,\|\nabla f\|^2_{\infty}\ ,
    \end{equation*}
    for every $t\geq 0$ and $f\in \CC^1_b(\R^N)$.
\end{lemma}

See Appendix \ref{appendix_expdecay} for the proof.

\begin{assumption}\label{ass_expdecay2}
\begin{subequations}
Assume $\nabla U_0$ satisfies condition \eqref{cond_nablaU0} and  
\begin{equation}\label{cond_expdecay2_2d}
    \sum_{j,i=1}^N \left|\partial_j \partial_i U_0(x)\right|\,\leq\, \a\left(1+\lambda(x)\right)\ ,\qquad x\in\R^N\ ,
\end{equation}
for some constant $\a>0$ (possibly depending on $N$) and where $\lambda$ is as in \eqref{cond_nablaU0}.
Moreover, assume that there exist constants $\gamma>0$ and $\rho\in(0,\frac{1}{5})$ such that
\begin{equation}\label{cond_expdecay_Vbis}
    \sup_{i=1,\dots,N}\,\sum_{k=1}^d |\partial_i V_k(x)|^2 \,\leq \,\frac{1}{N}\left(\rho\,\lambda(x)-\gamma\right)\ ,
\end{equation}
and
\begin{equation}\label{cond_expdecay_V2}
    \sup_{i,j=1,\dots,N}\,\sum_{k=1}^d | V_k(x)| \,|\partial_i\partial_j V_k(x)| \,\leq\,\rho\,\lambda(x)-\gamma\ .
\end{equation}
\end{subequations}
\end{assumption}

\begin{remark}
    Note that in dimension $N=1$, condition \eqref{cond_expdecay2_2d} is simply
    \begin{equation*}
         \left| U''_0(x)\right|\leq \a\left(1 +|U'_0(x)|\right)\ ,
    \end{equation*}
    and it is satisfied for any polynomial.
    
\end{remark}

\begin{remark}
   Continuing from Remark \ref{eg2}, for the first example in dimension $N=1$, we simply have $U_0''(x)=-6x$. Hence, recalling from remark \ref{eg2} that $\lambda(x)=3x^2+a$, we see that condition \eqref{cond_expdecay2_2d} is satisfied with constant $\a\geq 3/(1+a)$. Moreover, since $V'(x)=\frac{b}{1+x^2}\leq b$, condition \eqref{cond_expdecay_Vbis} is satisfied if there exist constants $\rho\in(0,\frac{1}{5})$ and $\gamma>0$ such that $b^2\leq \rho a-\gamma$, thus condition \eqref{cond_expdecay_Vbis} holds if $a> 5b^2$. Note that in this case also \eqref{cond_expdecay_V2} holds, since $|V(x)V''(x)|=\left|b^2 \arctan(x)\,\frac{2x}{(1+x^2)^2}\right|\leq b^2$.

    Now, we show that Assumption \ref{ass_expdecay2} is satisfied by the second example of Eg. \ref{toy_eg}. Clearly, conditions \eqref{cond_expdecay_Vbis} and \eqref{cond_expdecay_V2} hold since $V_1$ and $V_2$ are constant. Thus, we only need to check condition \eqref{cond_expdecay2_2d}. The second derivatives of $U_0$ are
    \begin{equation*}
        \partial_{1,1} U_0(\underline{x})= -2\,\binom{3a\,x_1}{x_2},\quad \partial_{1,2} U_0(\underline{x})=\partial_{2,1} U_0(\underline{x})= -2\,\binom{x_2}{x_1},\quad \partial_{2,2} U_0(\underline{x})=-2\,\binom{x_1}{3b\,x_2}.
    \end{equation*}
    Thus,
    \begin{equation*}
        \sum_{i,j=1}^2 |\partial_i \partial_j U_0(\underline{x})|\,\leq\, 2\left(3a+3b+4\right)\,|\underline{x}|\ ,
    \end{equation*}
   and comparing this with the function $\lambda(\underline{x})$ from Remark \ref{eg2} we can see that condition \eqref{cond_expdecay2_2d} is satisfied for some constant $\a>0$.
\end{remark}

\begin{lemma}\label{lemma_expdecay2}
Under Assumption \ref{ass_expdecay2}, there exist constants $\sigma,\, \tilde{K}_0>0$ such that
    \begin{equation*}
         \|\nabla^2 \PP_t f\|^2_{\infty}\,\leq\,\tilde{K}_0\,e^{-\sigma t}\,\left(\|\nabla f\|^2_{\infty}+\|\nabla^2 f\|^2_{\infty}\right)\ ,
    \end{equation*}
    for every $t\geq 0$ and $f\in\CC^2_b(\R^N)$. In particular, condition \eqref{cond_SDE} holds.
\end{lemma}

See Appendix \ref{appendix_expdecay} for the proof.

\section{Euler-Maruyama Schemes for SDEs with Lipschitz coefficients}\label{section_EM}

{In this Section, we provide UiT estimates of the weak error of the Euler-Maruyama scheme \eqref{EM} for the SDE \eqref{SDE_ito},  under a  globally Lipschitz assumptions on the coefficients (see Assumption \ref{ass_globalLip}).} 

Let $\{X^\delta_{t_n}\}_{n\in \N}$ be the Euler–Maruyama approximation with time-step $\delta$ of the SDE \eqref{SDE_ito}, i.e.
\begin{equation}\label{EM}
    X_{t_{n+1}}^{\delta}=X_{t_{n}}^{\delta}+U_{0}\left(X_{t_{n}}^{\delta}\right) \delta+\sqrt{2} \sum_{k=1}^{d} V_{k}\left(X_{t_{n}}^{\delta}\right) \Delta B_{t_{n}}^{k}, \quad X_{0}^{\delta}=x\ ,
\end{equation}
where $t_n= n\delta$ and $\Delta B^k_{t_n}=B^k_{t_{n+1}}-B^k_{t_n}$, $k\in\{1,\dots,d\}$, with $B^1_t,\dots,B^d_t$ 1-dimensional independent standard Brownian motions. 

Let $\{X^\delta_t\}_{t\geq 0}$ be the continuous-time interpolant of $\{X^\delta_{t_n}\}_{n\in \N}$, i.e.
\begin{equation}\label{interpolant}
    dX^\delta_t\,=\, U_0\left(X^\delta_{t_{n(t)}}\right)\,dt\,+\, \sqrt{2}\,\sum_{k=1}^d V_k\left(X^\delta_{t_{n(t)}}\right)\, dB^k_t\ ,\qquad X^\delta_0\,=\,x\ ,
\end{equation}
with $t_{n(t)}=t_i$ for $t\in[t_i,\,t_{i+1})$. 
%Furthermore, we introduce the discrete-time Markov semigroup associated to \eqref{EM}, $P^\delta_n(f)(x):= \E_x\left[f\left(X^\delta_{t_n}\right) \right]$, and the continuous-time Markov semigroup associated to the interpolation \eqref{interpolant}, $$\fP^\delta_s(f)(x):=\E_x\left[f(X^\delta_s) \right]=x+U_0(x)\,s\ ,$$ defined on the time interval $s\in[0,\delta]$. This way, for $t=n\delta+s$, with $0\leq s<\delta$, $n\in\N$, we can write
%\begin{equation*}
%    \E_x\left[f(X^\delta_t)\right]\,=\, P^\delta_n\left(\fP^\delta_s (f)\right)(x)\,=\, \E_x\left[ \fP^\delta_s (f)(X^\delta_{t_n}) \right]\ .
%\end{equation*}
%However, it is important to stress that $\E_x\left[f(X^\delta_t)\right] $ is not a Markov semigroup, and $P^\delta_n\left(\fP^\delta_s(f)\right)\,\neq\, \fP^\delta_s\left(P^\delta_n(f)\right)$.\\
By Ito's formula (the way this is applied is spelled out in e.g. ~\cite[Lemma 3.5]{crisan2021uniform}), if $\varphi\in \CC^2(\R^+\times\R^N)$, then 
\begin{align}
    \varphi(t,\,X^\delta_t)\,=\, &\varphi(0,\,X^\delta_0)\,+\, \int_0^t \left(\partial_s \varphi(s,\,X^\delta_s)\,+\,\LL_{\left(X^\delta_{t_{n(s)}}\right)}\varphi(s,\,X^\delta_s)\right)\, ds\nonumber\\
    &\;+\,\sqrt{2}\sum_{k=1}^d\sum_{i=1}^N\int_0^t V_k^i\left(X^\delta_{t_{n(s)}}\right)\,\partial_i \varphi(s,\,X^\delta_s)\,dB_s^k\ ,\label{3.5}
\end{align}
for any $t\geq 0$, where 
\begin{equation}\label{freezedL}
    \left(\mathcal{L}_{(v)} f\right)(x)\,:=\,\sum_{i=1}^{N} U_{0}^{i}(v)\,\partial_{i} f(x)\,+\,\sum_{i, j=1}^{N} V_{k}^{i}(v) V_{k}^{j}(v)\,\partial_{i,j} f(x) \,.
\end{equation}
Recalling that the generator of the SDE \eqref{SDE_ito} is the  second order differential operator $\LL$ defined on a set of suitably smooth functions $f:\R^N \to \R$ as
\begin{align}\label{LL}
    \mathcal{L} f(x):=\sum_{i=1}^{N} U_{0}^{i}(x)\, \partial_{i} f(x)+\sum_{k=1}^d\sum_{i, j=1}^{N} V_{k}^{i}(x) V_{k}^{j}(x) \,\partial_{i}\, \partial_{j} f(x) \, ,
\end{align}
 the operator \eqref{freezedL} can be seen as obtained from \eqref{LL} by `freezing' the value of the coefficients to $v\in \R^N$. In particular, we can write 

\begin{equation}\label{25}
    \mathbb{E}_x\left[ f\left(x_{t}\right)\right]-\mathbb{E}_x\left[ f\left(X_{t}^{\delta}\right)\right]=\mathbb{E}_x\left[ \int_{0}^{t}\Big(\mathcal{L}_{\left(X_{s}^{\delta}\right)}-\mathcal{L}_{\big(X_{t_{n(s)}}^{\delta}\big)}\Big)\left(\PP_{t-s} f\right)\left(X_{s}^{\delta}\right)\, d s\right]\ ,
\end{equation}
for any $f\in\CC_b(\R^N)$, $t\geq 0$ (see Lemma \ref{lemma_P-Q}).

\begin{assumption}\label{ass_globalLip}
The vector fields in \eqref{SDE_ito} satisfy the following:
\begin{itemize}
    \item there exist constants $c_0,\,c_1>0$ such that
\begin{subequations}
    \begin{align}
        \left| U_0(x)-U_0(y)\right|\,&\leq\, c_0\,|x-y|\ ,\label{U0_globalLip}\\
        \sum_{k=1}^d \left| V_k(x)-V_k(y)\right| \,&\leq\, c_1\,|x-y|\ ,\label{Vk_globalLip}
    \end{align}
    for every $x,\,y\in\R^N$;
    \item the vector fields $V_k$, $k=1,\dots,d$, are bounded, namely there exists a constant $K>0$ such that
    \begin{equation}\label{Vk_bounded}
        \sum_{k=1}^d \left| V_k(x)\right|^2\,\leq\, K\ ,
    \end{equation}
    for any $x\in\R^N$;
    
    \item there exist constants $b_0>0$ and $b_1\geq 0$ such that
\begin{equation}\label{cond_Lyap}
    \langle U_0(x),\,x\rangle\, \leq \, -b_0\,|x|^{2}\,+\, b_1 \ .
\end{equation}
\end{subequations}
\end{itemize}
\end{assumption}

\begin{prop}\label{prop_EM}
Let $X^\delta_t$ be the Euler-Maruyama scheme for the SDE \eqref{SDE_ito}, see  \eqref{EM}. Under Assumption \ref{ass_globalLip}, conditions \eqref{cond_Euler} and \eqref{cond_bound_Phi} are satisfied by the Euler-Maruyama scheme $X^\delta_t$.

If in addition Assumption \ref{ass_expdecay2} holds, then condition \eqref{cond_SDE} is satisfied as well, so that there exists a constant $\tilde{K}>0$ such that
\begin{equation}\label{EM_weak}
\sup_{t\geq 0} \,\left| \E_x\left[ f(x_t)\right]-\E_x\left[f\left(X^{\delta}_{t}\right)\right] \right|\,\leq\, \tilde{K}\|f\|_{\CC^2_b}\cdot \left(\delta\,\,|x|\,+\, \delta^{1/2}\right)\ ,
\end{equation}
for any $f\in\CC^2_b(\R^N)$ and $\delta>0$ small enough. That is, the Euler-Maruyama scheme is a UiT approximation of \eqref{SDE_ito}.

\end{prop}

\begin{proof}[Proof of Proposition \ref{prop_EM}]
Throughout the proof, unless further specified, $C>0$ denotes a generic positive constant independent of $x$ and $\delta$.

By Proposition \ref{prop_weak}, if we prove that conditions \eqref{cond_SDE}, \eqref{cond_Euler} and \eqref{cond_bound_Phi} hold for the Euler-Maruyama scheme $X^\delta_{t_n}$, then \eqref{EM_weak} follows. The estimate \eqref{cond_SDE} holds under Assumption \ref{ass_expdecay2}, by Lemma \ref{lemma_expdecay2}. Thus, we only need to prove that conditions \eqref{cond_Euler} and \eqref{cond_bound_Phi} are satisfied under Assumption \ref{ass_globalLip}.
%It is enough to show that for $\delta>0$ small enough Assumption \ref{ass_weak} is satisfied, thus the statement follows by Proposition \ref{prop_weak}. 
More specifically, we show that, under Assumption \ref{ass_globalLip}, the hypotheses in Lemma \ref{standard_eg} hold, i.e.~we  show that condition \eqref{cond_Euler} is satisfied with
\begin{equation}\label{EM_phi}
    \phi(x,\delta)\,:=\, |x|\,\delta^2\,+\, \delta^{3/2} \ ,
\end{equation}
and that condition \eqref{M1} holds for $\delta>0$ small enough, i.e.
\begin{equation}\label{EM_moments}
    \E_x\left[\left|X^\delta_\delta \right|\right]\,\leq\, \e_\delta\, |x| +c_\delta\ ,
\end{equation}
with $\e_\delta:= \left(1 - 2\delta b_0 +2\delta^2 c_0\right)^{1/2}$ and $c_\delta = C\,\delta^{1/2}$. Thus, by Lemma \ref{standard_eg}, condition \eqref{cond_bound_Phi} is satisfied with
\begin{align*}
    \Phi(x,\delta)\,=\,\delta^2\,|x|\,+\, \frac{c_\delta\, \delta^2}{1-\e_\delta}\,+\, \delta^{3/2}\,\leq\,\delta^2\,|x|+C\delta^{3/2}\ ,
\end{align*}
for any $\delta>0$ small enough. Hence, the statement follows by Proposition \ref{prop_weak}. 

So, all we are left to do is to prove that \eqref{EM_phi} and \eqref{EM_moments} hold. We start with proving the first. By \eqref{25}, we have
\begin{align*}
    \E_x\left[f(x_\delta)-f\left(X^\delta_\delta\right)\right]\,=\,\E_x\left[ \int_0^\delta \left(\LL_{(X^\delta_r)}\,-\,\LL_{(x)}\right)\PP_{\delta-r}f(X^\delta_r)\,dr\right]\ .
\end{align*}

Using the definition of the operator $\LL_{(v)}$ \eqref{freezedL} with $v\in\R^N$, we obtain
\begin{align*}
    \E_x&\left[f(x_\delta)-f\left(X^\delta_\delta\right)\right]\\
    &\;=\,\E_x\bigg[ \int_0^\delta \langle U_0(X^\delta_r)-U_0(x),\,\nabla \PP_{\delta-r}f(X^\delta_r)\rangle\nonumber\\
    &\qquad\qquad+\,\sum_{k=1}^d \left(V_k(X^\delta_r)-V_k(x)\right)^T\,\nabla^2\PP_{\delta-r}f(X^\delta_r)\, \left(V_k(X^\delta_r)+V_k(x)\right)\,dr\bigg]\ .
\end{align*}

Therefore, by conditions \eqref{U0_globalLip}-\eqref{Vk_bounded}, we have
\begin{align}
    &\left|\E_x\left[f(x_\delta)-f\left(X^\delta_\delta\right)\right]\right|\nonumber\\
    &\qquad\leq\, C\,\E_x\left[\int_0^\delta \left(\left|\nabla \PP_{\delta-r} f(X^\delta_r)\right|\,+\, \left\|\nabla^2\PP_{\delta-r}f(X^\delta_r)\right\| \right)\,\left|X^\delta_r-x\right|\, dr\right]\ .\label{bound_EM}
\end{align}
By \eqref{cond_SDE}, $\left|\nabla \PP_{t} f(x)\right|+ \left\|\nabla^2\PP_{t}f(x)\right\|\leq K_0\,\|f\|_{\CC^2_b}$. Moreover, by construction of the Euler-Maruyama approximation \eqref{EM},
\begin{equation*}
    \E_x\left[|X_r^\delta-x|\right] \,\leq\, |U_0(x)|\, r\,+\, \sqrt{2}\,\sum_{k=1}^d |V_k(x)|\, r^{1/2}\,\leq\, C\,\left(|x|\,r\,+\, r^{1/2}\right)\ ,
\end{equation*}
for any $r\in[0,\delta]$ with $\delta<1$, since by \eqref{U0_globalLip} we deduce $|U_0(x)|\leq C(|x|+1)$. Therefore, integrating over $r\in[0,\delta]$, we can bound the LHS of \eqref{bound_EM} by
\begin{align*}
    \left|\E_x\left[f(x_\delta)-f\left(X^\delta_\delta\right)\right]\right|\,
    &\leq\,C\,\|f\|_{\CC^2_b}\cdot\E_x\left[\int_0^\delta \left|X^\delta_r-x\right|\, dr\right]\\
    &\leq\, C\,\|f\|_{\CC^2_b}\,\left(|x|\,\delta^2\,+\, \delta^{3/2}\right)\ .
\end{align*}

We are now left with proving that \eqref{EM_moments} holds. Again, by \eqref{EM} and by Assumption \ref{ass_globalLip}, we have  
\begin{align*}
    \E_x\left[ \left|X^\delta_\delta\right|^2 \right]\,&=\, |x|^2 \,+\, 2\delta\,\langle U_0(x),\,x\rangle\,+\, \delta^2\,\left|U_0(x)\right|^2\,+\, 2\delta\sum_{k=1}^d\left|V_k(x)\right|^2\\
    &\leq\, |x|^2 \,+2\delta\,\left(- b_0|x|^2+b_1\right)\,+\,C\delta^2\,\left(|x|^2+ 1\right)\,+\, 2\delta\,K\\
    &\leq\, \left(1 - 2\delta b_0 +C\delta^2 \right)\,|x|^2\,+\, C\delta\ .
\end{align*}
Since $1 - 2\delta b_0 +C\delta^2\in (0,1)$ for any $\delta>0$ small enough, we obtain \eqref{EM_moments}.

\end{proof}

\section{Split-Step and  Implicit Euler Schemes}\label{section_split}

In this section, we provide time-uniform weak error bounds for the split-step Euler scheme \eqref{SS2}-\eqref{SS1} and the implicit Euler scheme \eqref{general_implicit} under one-sided Lipschitz conditions (more precisely, under Assumption \ref{ass_onesidedLip} and Assumption \ref{ass_for_modified_expdecay} below). As illustrated in \cite{higham2002strong}, one key property of split-step and implicit Euler approximations is that, under locally Lipschitz assumptions on the drift of the SDE \eqref{SDE_ito}, they can be interpreted as standard  Euler-Maruyama schemes for a modified SDE (namely \eqref{modifiedSDE}), as reviewed below. Such a modified SDE has globally Lipschitz coefficients so, thanks to this interpretation, we will show that the time uniform estimates for split-step and implicit Euler schemes can be obtained by adapting the results in Section \ref{section_EM} to the modified SDE. However, it is important to stress that the results of Section \ref{section_EM} need to be applied with some care. Indeed, although it is true that the drift of the modified SDE is globally Lipschitz (see property \eqref{U0delta}), the Lipschitz constant is inversely proportional to $\delta$, so it is important to carefully keep track of the dependence on $\delta$ in all the estimates of this section. In particular, we will need to show that the modified SDE \eqref{modifiedSDE} is strongly exponentially stable with constants in \eqref{cond_SDE} independent of $\delta$.

In the remainder of this section, we first recall  the definition of split-step and implicit Euler schemes and their relation to the explicit Euler scheme for the modified SDE \eqref{modifiedSDE}. We  then explain the organisation of this section and the strategy of proof in more detail. 

Let $\{Z^\delta_{t_n}\}_{n\in \N}$ be the split-step Euler approximation with time-step $\delta$ of the SDE \eqref{SDE_ito}, namely
\begin{subequations}
\begin{align}
Z^\delta_{t_{n+1}}&=Z^{\star}_n+\sqrt{2} \sum_{k=1}^{d} V_{k}\left(Z_{n}^{\star}\right) \Delta B_{t_{n}}^{k}\ ,\label{SS2}\\
Z^{\star}_n&=Z^\delta_{t_n}+ U_0\left(Z^{\star}_n\right)\delta\ , \label{SS1}
\end{align}
\end{subequations}
with $Z_{0}^{\delta} =x$, $t_n= n\delta$ and $\Delta B_{t_n}=B_{t_{n+1}}-B_{t_n}$ (see \cite[Section 3.3]{higham2002strong} for further details). 
Similarly, let $\{Y^{\delta}_{t_n}\}_{n\in \N}$ be a standard implicit Euler approximation with time-step $\delta$, i.e.
\begin{align}
    Y_{t_{n+1}}^{\delta}=Y_{t_{n}}^{\delta}+\delta\,U_{0}\left(Y_{t_{n+1}}^{\delta}\right)\, +\,\sqrt{2} \sum_{k=1}^{d} V_{k}\left(Y_{t_{n}}^{\delta}\right) \Delta B_{t_{n}}^{k}&\ ,\label{general_implicit}
\end{align}
with $t_n$ and $\Delta B_{t_n}$ as above (see \cite[Section 5]{higham2002strong} for further details).

\begin{assumption}\label{ass_onesidedLip}
The vector fields $V_k$, $k=1,\dots,d$, satisfy conditions \eqref{Vk_globalLip} and \eqref{Vk_bounded}. The coefficient $U_0$ satisfies condition \eqref{cond_Lyap} and there exist constants $c_0,\,c_2,\,q>0$ such that
\begin{subequations}
    \begin{align}
        \langle U_0(x)-U_0(y),\,x-y\rangle\,&\leq\, c_0\,|x-y|^2\ ,\label{U0_onesidedLip}\\
        |U_0(x)-U_0(y)|^2\,&\leq\, c_2\,\left(1+|x|^{2q} +|y|^{2q}\right) \,|x-y|^2\ ,\label{cond_poly}
    \end{align}
    for every $x,\,y\in\R^N$. 
\end{subequations}
\end{assumption}

\begin{remark}
The vector fields $U_0$ of Example \ref{toy_eg} satisfy all the conditions of Assumption \ref{ass_onesidedLip}.
\end{remark}

\begin{lemma}\label{lemma_Fdelta}
Assume the SDE \eqref{SDE_ito} satisfies condition \eqref{U0_onesidedLip}. Let $\delta\in(0,\,\frac{1}{2\,c_0})$, with $c_0>0$ as in \eqref{U0_onesidedLip}. Then, there exists a unique solution $z^\star$ of the equation
\begin{equation}\label{implicit_eq}
    z^\star \,=\, y+U_0(z^\star)\,\delta\ ,
\end{equation}
for any $y\in\R^N$. Moreover, the function $F^\delta:\R^N\to\R^N$ defined by 
\begin{equation}\label{def_Fdelta}
    F^\delta\,:\, y\mapsto z^\star
\end{equation}
where $z^\star$ is the solution of \eqref{implicit_eq},
is smooth and
\begin{equation}\label{Fdelta}
    \left| F^\delta(y)-F^\delta(x)\right|^2\,\leq\, \frac{1}{1-2\delta\cdot c_0}\,|y-x|^2\ ,
\end{equation}
for all $y,x\in\R^N$.
\end{lemma}

\begin{proof}
Follows from the first part of the proof of \cite[Lemma 3.4]{higham2002strong}.

\end{proof}

In the same fashion as \cite{higham2002strong}, we define the modified vector fields
\begin{align}\label{def_Udelta_Vdelta}
    U^\delta_0(y):=U_0\left(F^\delta(y)\right)\, ,\qquad V^\delta_k(y):=V_k\left(F^\delta(y)\right)\, ,\quad k=1,\dots,d.
\end{align}
The split-step Euler scheme $Z^\delta_{t_n}$ with step $\delta$ \eqref{SS2}-\eqref{SS1} can be then rewritten as
\begin{equation*}
    Z^{\delta}_{t_{n+1}}= Z^\delta_{t_n}\,+\,U^\delta_0\left(Z^\delta_{t_n}\right)\,\delta\,+\, \sqrt{2}\sum_{k=1}^d V^\delta_k\left(Z^\delta_{t_n}\right)\,\Delta B^k_{t_n}\ .
\end{equation*}
In other words, the split-step scheme can be interpreted as an explicit Euler scheme $\overline{X}^\delta_{t_n}$ with step $\delta$ applied to the modified SDE
\begin{equation}\label{modifiedSDE}
    d \overline{x}_{t}^{\delta}=U_{0}^\delta\left(\overline{x}_{t}^{\delta}\right) d t+\sqrt{2} \sum_{k=1}^{d} V_{k}^\delta\left(\overline{x}_{t}^{\delta}\right) d B_{t}^{k}, 
\end{equation}
with initial state $\overline{x}^\delta_0=x=Z^\delta_0\in\R^N$. That is, if we consider the Euler scheme $\overline{X}^\delta_{t_n}$ defined by
\begin{equation}\label{modified_EM}
\overline{X}_{t_{n+1}}^{\delta}=\overline{X}_{t_{n}}^{\delta}+U^\delta_{0}\left(\overline{X}_{t_{n}}^{\delta}\right) \delta+\sqrt{2} \sum_{k=1}^{d} V^\delta_{k}\left(\overline{X}_{t_{n}}^{\delta}\right) \Delta B_{t_{n}}^{k}\ ,
\end{equation}
with $\overline{X}_{0}^{\delta}=x=Z^\delta_0$, then we have $Z^\delta_{t_n}=\overline{X}^\delta_{t_n}$ for every $n\in\N$.

Similarly, the implicit Euler scheme $\{Y^{\delta}_{t_n}\}_{n\in \N}$ for the SDE \eqref{SDE_ito}, with initial state $Y^\delta_0=x\in\R^N$ and time-step $\delta$, can be written in terms of the Euler-Maruyama scheme $\overline{X}^\delta_{t_n}$ \eqref{modified_EM} for the modified SDE \eqref{modifiedSDE} with initial state $\overline{x}^\delta_0 =\overline{X}_{0}^{\delta}=x-\delta U_0(x) =Y^\delta_0-\delta\,U_0(Y^\delta_0)$. More precisely, we have 
\begin{equation}\label{implicit_as_EM}
Y^{\delta}_{t_n}\,=\, F^{\delta}(\overline{X}^\delta_{t_n})\ .
\end{equation}
The identity \eqref{implicit_as_EM} can be shown by induction on $n\in\N$. Indeed, provided \eqref{implicit_as_EM} holds for $t_n$, we can see that
\begin{align*}
    Y^{\delta}_{t_{n+1}}-\delta \,U_0\left(Y^{\delta}_{t_{n+1}} \right)\,&=\, \overline{X}^\delta_{t_n}\,+\, \delta\, U_0\left(Y^{\delta}_{t_{n}} \right)\,+\, \sqrt{2}\sum_{k=1}^d V_k\left( Y^{\delta}_{t_{n}}\right)\, \Delta B^k_{t_n}\\
    &=\,  \overline{X}^\delta_{t_n}\,+\, \delta\, U_0^{\delta}\left(\overline{X}^\delta_{t_n}\right)\,+\, \sqrt{2}\sum_{k=1}^d V_k^{\delta}\left( \overline{X}^\delta_{t_n}\right)\, \Delta B^k_{t_n}\,=\, \overline{X}^\delta_{t_{n+1}} \ .
\end{align*}

This section is structured as follows. In Subsection \ref{subsection_modified} we present useful properties of the modified SDE (independent of the discretisation scheme at hand). In particular, we show that the solution of \eqref{modifiedSDE} converges weakly, uniformly in time, to the solution of the original SDE \eqref{SDE_ito} as $\delta\to 0$ (see Proposition \ref{prop_SDE_modified}). Moreover, we provide a criterion to ensure strong exponential stability for the modified SDE \eqref{modifiedSDE} with constants independent of $\delta$ (see Lemma \ref{lemma_modified_ExpDecay}). This is important because, if the SES condition \eqref{cond_SDE} holds with $\laa = \laa(\delta)\approx o(\delta^\alpha)$, $\alpha>0$, the UiT weak error bound would be \eqref{weak_conv} with denominator $1-e^{-\laa(\delta)\, \delta}\approx o(\delta^{1+\alpha})$, thus reducing the order of the weak error bound. 

In Subsection \ref{subsection_moments}, we prove the UiT weak convergence of the Euler-Maruyama scheme for the modified SDE \eqref{modifiedSDE} (see Proposition \ref{prop_modified}): the proof is done by using the paradigm of Proposition \ref{prop_EM}, hence  by using the time-uniform moment bounds of the Euler-Maruyama scheme for the modified SDE (Lemma \ref{lemma_alternative}) and the strong exponential stability for the modified SDE (Lemma \ref{lemma_modified_ExpDecay}). 
%we show that, under appropriate Lyapunov conditions, we have time-uniform moment bounds for the Euler-Maruyama approximation $\overline{X}^\delta_t$ (see Lemma \ref{lemma_alternative}). 

Finally, in Subsection \ref{subsection_weak}, recalling that we can interpret the split-step and implicit Euler schemes as Euler-Maruyama schemes for the modified SDE, we combine the UiT weak convergence of the Euler-Maruyama scheme for the modified SDE (Proposition \ref{prop_modified}) with the UiT weak convergence of the modified SDE \eqref{modifiedSDE} to the original SDE \eqref{SDE_ito} (Proposition \ref{prop_SDE_modified}) to prove UiT weak convergence for split-step and implicit Euler schemes (see Theorem \ref{thm_split} and Theorem \ref{thm_implicit}).

We denote by $\overline{\PP}_t^{\delta}$ the Markov semigroup associated to the modified SDE \eqref{modifiedSDE}, and by $\overline{\LL}^\delta$ the corresponding generator, that is
\begin{equation*}
    \overline{\LL}^\delta f(x)\,=\, \sum_{i=1}^{N} U_{0}^{\delta,i}(x)\, \partial_{i} f(x)+\sum_{k=1}^d\sum_{i, j=1}^{N} V_{k}^{\delta,i}(x) V_{k}^{\delta,j}(x) \,\partial_{i}\, \partial_{j} f(x) \ ,
\end{equation*}
for suitably smooth functions $f:\R^N\to\R$.

To avoid confusion, we summarise here the main notation:
\begin{itemize}
    \item $x_t,\;\PP_t,\; \LL$ denote respectively the solution of the SDE \eqref{SDE_ito}, the corresponding semigroup and the infinitesimal generator, respectively;
    \item $\overline{x}^\delta_t,\;\overline{\PP}^\delta_t,\; \overline{\LL}^\delta$ denote respectively the solution of the modified SDE \eqref{modifiedSDE}, the corresponding semigroup and the infinitesimal generator, respectively;
    %\item $X_t^\delta$ denotes the continuous-time interpolant of the explicit Euler-Maruyama approximation of the original SDE \eqref{SDE_ito} with step $\delta$;
    \item $\overline{X}_{t_n}^\delta$ denotes the Euler-Maruyama approximation \eqref{modified_EM} of the modified SDE \eqref{modifiedSDE};
    \item $\overline{X}^\delta_t$, $t\geq 0$, denotes the continuous-time interpolant of $\overline{X}^\delta_{t_n}$, namely
\begin{equation}\label{interpolant_overline}
    d\overline{X}^\delta_t\,=\, U_0^{\delta}\left(\overline{X}^\delta_{t_{n(t)}}\right)\,dt\,+\, \sum_{k=1}^d V_k\left(\overline{X}^\delta_{t_{n(t)}}\right)\, dB^k_t\ ,
\end{equation}
with $t_{n(t)}=t_i$ for $t\in[t_i,\,t_{i+1})$;
    \item $Z^\delta_{t_n}$ is the split-step Euler approximation \eqref{SS2}-\eqref{SS1} with step $\delta>0$ of the SDE \eqref{SDE_ito};
    \item $Y^{\delta}_{t_n}$ is the implicit Euler approximation \eqref{general_implicit} with step $\delta>0$ of the SDE \eqref{SDE_ito}.
    %\item $Y^{\delta}_t$ is the continuous-time interpolant of the tamed Euler approximation of the original SDE \eqref{SDE_ito} with step $\delta$.
\end{itemize}

\subsection{Properties of the modified SDE \eqref{modifiedSDE}}\label{subsection_modified}
Here, we present some properties of the modified SDE \eqref{modifiedSDE} which will be used in later sections, and we show that the solution of \eqref{modifiedSDE} converges weakly to the solution of the original SDE \eqref{SDE_ito}. Moreover, we provide a criterion (see Lemma \ref{lemma_modified_ExpDecay}) to ensure strong exponential stability of the semigroup associated to the modified SDE \eqref{modifiedSDE}, with constants independent of $\delta$.

\begin{lemma}\label{lemma_Udelta}
Under conditions \eqref{U0_onesidedLip}-\eqref{cond_poly} and \eqref{Vk_globalLip}-\eqref{Vk_bounded}, $U^\delta_0$ converges to $U_0$ and $V^\delta_k$ converges to $V_k$ as $\delta\to 0$ uniformly on compact sets, for $k=1,\dots,d$; moreover for $\delta\in(0,\,\frac{1}{2\,c_0})$ the vector fields $U^\delta_0$, $V^\delta_k$, $k=1,\dots,d$, are globally Lipschitz (with Lipschitz constant dependent on $\delta$), where $c_0$ is as in \eqref{U0_onesidedLip}. More precisely, the following bounds hold,
\begin{subequations}
\begin{align}
    \left| U_0^\delta(x)\right|\,&\leq\,\frac{1}{1-\delta\cdot c_0}\, \left| U_0(x)\right| \ , \label{|U0delta|}\\
    \left| U_0^\delta(x)-U_0^\delta(y)\right|\,&\leq\,\frac{1}{\delta}\,\left(1+\frac{1}{\sqrt{1-2\delta\cdot c_0}}\right)\, |x-y| \ , \label{U0delta}\\
    \sum_{k=1}^d\left| V_k^\delta(x)-V_k^\delta(y)\right|\,&\leq \,\frac{c_1}{\sqrt{1-2\delta\cdot c_0}}\,|x-y| \ , \label{Vdelta}\\
    \left| U_0(x)\,-\,U_0^\delta(x)\right|^2\,+\,\sum_{k=1}^d\left| V_k(x)\,-\,V_k^\delta(x)\right|^2\,&\leq\, \bar{c}\, \left( 1+|x|^{2+4q}\right)\, \delta^2\ ,\label{U0d-U0}
\end{align}
\end{subequations}
for any $x,y\in \R^N$.
\end{lemma}

\begin{proof}
For the proof of the bound \eqref{|U0delta|}, see the second part of the proof of \cite[Lemma 3.4]{higham2002strong}. To prove the bound \eqref{U0delta}, note that, by the basic identity \begin{equation}\label{Udelta}
    \delta\cdot U_0^\delta(x)\,=\, F^\delta(x)\,-\, x\ ,
\end{equation}
we have
\begin{equation*}
    \left| U_0^\delta(x)-U_0^\delta(y)\right|\,=\, \frac{1}{\delta}\,\left|F^\delta(x)-F^\delta(y)-x+y\right|\,\leq\,\frac{1}{\delta}\,\left|F^\delta(x)-F^\delta(y)\right|+\frac{1}{\delta}\left|x-y\right| \ .
\end{equation*}
Therefore, by \eqref{Fdelta}, we obtain the bound \eqref{U0delta}. Similarly, to prove the bound \eqref{Vdelta}, it is enough to observe the following
\begin{equation*}
    \sum_{k=1}^d\left| V_k^\delta(x)-V_k^\delta(y)\right|\,=\,\sum_{k=1}^d\left| V_k\left(F^\delta(x)\right)-V_k\left(F^\delta(y)\right)\right|\,\leq\, c_1\left|F^\delta(x)-F^\delta(y)\right|\ ,
\end{equation*}
since $V_k$, $k=1,\dots,d$, are globally Lipschitz \eqref{Vk_globalLip}. Hence, we conclude by \eqref{Fdelta}.

Finally, by \eqref{|U0delta|} and then condition \eqref{cond_poly}, there exists a constant $\bar{c}>0$ such that
\begin{equation*}
    \left|x-F^\delta(x)\right|^2\,=\,\delta^2\,\left|U^\delta_0(x)\right|^2\,\leq\, c\,\delta^2\,\left|U_0(x)\right|^2\,\leq\, \bar{c}\,\delta^2\,\left(1+|x|^{2+2q}\right)\ .
\end{equation*}
Moreover, by \eqref{Fdelta}, we have $|F^\delta(x)|\leq \bar{c}\,\left(1+|x|\right)$, for every $x\in\R^N$, for some constant $\bar{c}>0$. Hence, we obtain
\begin{align*}
    \left| U_0(x)\,-\,U_0^\delta(x)\right|^2\,&=\, \left| U_0(x)\,-\,U_0\left(F^\delta(x)\right)\right|^2\\
    &\leq\, c\, \left(1+|x|^{2q}+\left| F^{\delta}(x)\right|^{2q}\right)\,\left| x-F^{\delta}(x)\right|^2\\
    & \leq\, {c}\, \left(1+|x|^{2q}\right)\, \left( 1+|x|^{2+2q}\right)\, \delta^2\\
    &\leq\, \bar{c}\, \left( 1+|x|^{2+4q}\right)\, \delta^2\ .
\end{align*}
In the same fashion, by \eqref{Vk_globalLip}, 
\begin{align*}
    \sum_{k=1}^d\left| V_k(x)\,-\,V_k^\delta(x)\right|^2\,\leq\,c_1\,\left| x- F^\delta(x) \right|^2\,\leq\, \bar{c}\,\left(1+|x|^{2+2q}\right)\,\delta^2 \ .
\end{align*}
This proves the bound \eqref{U0d-U0}.

\end{proof}

\begin{assumption}\label{ass_for_modified_expdecay}
\begin{subequations}
Assume that there exist a constant $\beta\geq 1$ and a positive function $\lambda:\R^N\to\R$ such that $\lambda(x)>0$ for every $x\in\R^N$, $\inf_{x\in\R^N} \lambda(x) =:\lambda^\star >0$, and
\begin{align}\label{cond_nablaU0_formodified}
    -\beta\,\lambda(x)\, |v|^2\,\leq\, v^T\, \nabla U_0(x)\, v
    \,\leq \, -\lambda(x)\, |v|^2\ ,
\end{align}
for every $x,v\in\R^N$. Moreover, assume the following  holds: 

$\bullet$ there exists a constant $\alpha>0$ such that 
\begin{equation}\label{cond_expdecay2_2d_formodified}
    \sum_{j,i=1}^N \left|\partial_j \partial_i U_0(x)\right|\,\leq\, \a\,\lambda(x)\ ,\qquad x\in\R^N \,;
\end{equation}

 $\bullet$  there exist a constant $\rho\in(0,\frac{1}{5})$ such that
\begin{equation}\label{cond_expdecay_Vbis_formodified}
    \sum_{k=1}^d \|\nabla V_k(x)\|^2 \,\leq \,\frac{\rho}{N\,\beta^2}\,\lambda(x)\ ,\qquad x\in\R^N\ ,
\end{equation}
with $\beta\geq 1$ given by \eqref{cond_nablaU0_formodified}, and a constant $\hat K>0$ such that $\sup_k \left|V_k(x)\right|\leq \hat K$ and
\begin{equation}\label{cond_expdecay_V2_formodified}
    \sum_{k=1}^d \left(\left(\sum_{m,m'} \left| \partial_m \partial_{m'} V_k(x) \right|^2\right)^{1/2}+\a\,\left\|\nabla V_k(x) \right\|\right) \,\leq\,\frac{\rho}{\hat K\,\beta^2}\,\lambda(x)\ ,\qquad x\in\R^N\ ,
\end{equation}
\end{subequations}
where $\alpha>0$ is given by \eqref{cond_expdecay2_2d_formodified}.
\end{assumption}

\begin{example}
    For $N=d=1$, the drift $U_0(x)=-x^3-x$ satisfies conditions \eqref{cond_nablaU0_formodified} and \eqref{cond_expdecay2_2d_formodified} with $\lambda(x)=3x^2+1$ and $\alpha=2$ (and, obviously, $\beta=1$ since $d=1$).
\end{example}

\begin{lemma}[Strong exponential stability for the modified SDE]\label{lemma_modified_ExpDecay}
Under Assumption \ref{ass_for_modified_expdecay}, the bound \eqref{cond_SDE} is satisfied for the modified SDE \eqref{modifiedSDE} uniformly in $\delta$. Namely, there exist constants $\sigma,\, \hat{K}_0>0$ (independent of $\delta$) such that
    \begin{equation*}
         \|\bar{\PP}^\delta_t f\|_{\CC^2_b}\,\leq\,\hat{K}_0\,e^{-\sigma t}\, \| f\|_{\CC^2_b}\ ,
    \end{equation*}
    for every $\delta>0$ small enough, $t\geq 0$ and $f\in\CC^2_b(\R^N)$.
\end{lemma}

See Appendix \ref{appendix_expdecay} for the proof.

\begin{prop}\label{prop_SDE_modified}
Let $x_t$ be the solution of the SDE \eqref{SDE_ito} and let $\overline{x}_{t}^\delta$ be the solution of the modified SDE \eqref{modifiedSDE}.
If Assumption \ref{ass_onesidedLip} and Assumption \ref{ass_for_modified_expdecay} hold, then there exists a constant $K>0$ (independent of $t$, $\delta$) such that
\begin{equation*}
    \sup_{t\geq 0} \,\left| \E_x\left[ f\left(\overline{x}_{t}^\delta\right)\right]-\E_x\left[f\left(x_{t}\right)\right] \right|\,\leq\, \delta\,K\,\|f\|_{\CC^2_b}\,\left(1+|x|^{2+2q}\right)\ ,
\end{equation*}
for all $f\in \CC_b^2(\R^N)$ and $\delta >0$ small enough, with $q>0$ as in condition \eqref{cond_poly}.

\end{prop}

\begin{proof}
Denoting by $\overline{\LL}^\delta$ and $\overline{\PP}_t^\delta$ respectively the generator and semigroup associated to the modified SDE \eqref{modifiedSDE}, by Lemma \ref{lemma_P-Q} we can write
\begin{align*}
    &\E_x\left[ f\left(\overline{x}_{t}^\delta\right)\right]\,-\,\E_x\left[f\left(x_{t}\right)\right]\\
    &\qquad=\, \E_x \left[ \int_0^t \left( \overline{\LL}^\delta-\LL \right)\overline{\PP}^\delta_{t-s}f\left(x_s\right)\,ds \right] \\
    &\qquad=\, \E_x \bigg[ \int_0^t \langle  U_0^\delta\left(x_s\right) - U_0\left(x_s\right) ,\, \nabla \overline{\PP}^\delta_{t-s}f(x_s)\rangle \\
    &\qquad\qquad+ \sum_{k=1}^d \left( V^\delta_k\left(x_s\right) - V_k\left(x_s\right)  \right)^T\, \nabla^2\overline{\PP}^\delta_{t-s}f\left(x_s\right)\, \left( V^\delta_k\left(x_s\right) - V_k\left(x_s\right) \right)\, ds \bigg]\ .
\end{align*}
By Assumption \ref{ass_for_modified_expdecay}, we can apply Lemma \ref{lemma_modified_ExpDecay} and obtain
\begin{align*}
    &\left|\E_x\left[ f\left(\overline{x}_{t}^\delta\right)\right]-\E_x\left[f\left(x_{t}\right)\right]\right|\\
    &\qquad\leq\, \int_0^t \|\overline\PP^\delta_{t-s}f\|_{\CC^2_b} \, \E_x\left[ \left|U^\delta_0(x_s)-U_0(x_s)\right|+\sum_{k=1}^d\left|V^\delta_k(x_s)-V_k(x_s)\right|^2\right]\, ds\\
    &\qquad\leq\, \hat{K}_0\,\|f\|_{\CC^2_b}\,\int_0^t  e^{-\sigma(t-s)}\,\E_x\left[ \left|U^\delta_0(x_s)-U_0(x_s)\right|+\sum_{k=1}^d\left|V^\delta_k(x_s)-V_k(x_s)\right|^2 \right]\, ds\ .
\end{align*}

By \eqref{U0d-U0}, we have
\begin{align*}
    &\left|\E_x\left[ f\left(\overline{x}_{t}^\delta\right)\right]-\E_x\left[f\left(x_{t}\right)\right]\right|\\
    &\qquad\leq\,C\,\|f\|_{\CC^2_b}\, \int_0^t e^{-\sigma(t-s)} \,\E_x\left[ \left(\left(1+\left|x_s\right|^{1+2q}\right)\,\delta\,+\, \left(1+\left|x_s\right|^{2+2q}\right)\,\delta^2\right)\right]\,ds\\
    &\qquad\leq\,\delta\,C\,\|f\|_{\CC^2_b}\, \int_0^t e^{-\sigma(t-s)} \, \left(1+ \E_x\left[ \left|x_s\right|^{2+2q}\right]\right)\,ds\ ,
\end{align*}
where the constant $C>0$ may change from line to line. To conclude we apply Lemma \ref{lemma_xmoments} and obtain
\begin{align*}
    \left|\E_x\left[ f\left(\overline{x}_{t}^\delta\right)\right]-\E_x\left[f\left(x_{t}\right)\right]\right|\,\leq\,\delta\,C\,\|f\|_{\CC^2_b}\, \int_0^t \left(1+e^{-b_0 s}\left|x\right|^{2+2q}\right)\, e^{-\sigma (t-s)} \,ds\ .
\end{align*}
The statement follows by integrating.

\end{proof}

\subsection{UiT Weak Convergence for EM schemes for the modified SDE \eqref{modifiedSDE}}\label{subsection_moments}

In this subsection, we study the weak convergence of the explicit Euler scheme for the modified SDE \eqref{modifiedSDE} (see Proposition \ref{prop_modified}). To this end, we first provide uniform in time bounds for the moments of Euler approximation $\overline{X}^\delta_{t_n}$ of the modified SDE \eqref{modifiedSDE}.

\begin{lemma}\label{lemma_alternative}

Let the SDE \eqref{SDE_ito} satisfy Assumption \ref{ass_onesidedLip}. Then, the explicit Euler scheme $\overline{X}^\delta_{t_n}$ \eqref{modified_EM} (with initial condition $\overline{X}^\delta_0=x$) satisfies the following UiT moment bounds, 
\begin{align}
\sup_{n\in\N}\,\E_x\left[\left|\overline{X}^\delta_{t_n}\right|^{2q}\right]\,\leq\, 2\,|x|^{2q}\,+\,C_q\ ,
\end{align}
for every $q\in\R$, $q\geq 1$, and any $\delta \in(0,1)$, with constant ${C}_q>0$ independent of $\delta$ and $x$.

\end{lemma}

\begin{proof}
Applying condition \eqref{cond_Lyap} to $F^\delta(x)$ \eqref{def_Fdelta}, we have
\begin{align*}
    \langle U^\delta_0(x),\,F^\delta(x)\rangle\,=\, \langle U_0\left(F^\delta(x)\right),\, F^\delta(x)\rangle\,\leq\, -b_0\,\left|F^\delta(x)\right|^2\,+\, b_1\ .\label{Lyap_delta}
\end{align*}
Hence, recalling that $F^\delta(x)=x+\delta\,U_0^\delta(x)$ by definition, using Young's inequality and the above bound one gets
\begin{align*}
    \left|F^\delta(x)\right|^2\,&=\, \langle x,\,F^\delta(x) \rangle\,+\, \delta\cdot \langle U_0^\delta(x),\,F^\delta(x) \rangle\\
    &\leq\, \frac{1}{2\a}\,|x|^2\,+\,\frac{\a}{2}\,\left|F^\delta(x)\right|^2\,-\,b_0\delta \,\left|F^\delta(x)\right|^2\,+\, b_1\delta\ ,
\end{align*}
for any $\a>0$ possibly depending on $\delta$. Rearranging, we obtain
\begin{equation*}
    \left|F^\delta(x)\right|^2\,\leq\, \frac{1}{2\a\left(1+b_0\delta-\frac{\a}{2}\right)}\,|x|^2 + \frac{b_1\delta}{1+b_0\delta-\frac{\a}{2}}\ ,
\end{equation*}
provided $\a<2+2b_0\delta$. The coefficient of $|x|^2$ is minimised when $\a=1+b_0\delta$, and in this case we obtain
\begin{equation}\label{Fbound}
    \left|F^\delta(x)\right|^2\,\leq\, \frac{1}{\left(1+b_0\delta\right)^2}\,|x|^2 + \frac{2b_1\delta}{1+b_0\delta}\ .
\end{equation}

Now, from the explicit expression of $\overline{X}^\delta_\delta$ \eqref{modified_EM} at time $t_1=\delta$, we can write
\begin{align*}
    \E_x\left[ \left| \overline{X}^\delta_\delta\right|^{2k} \right]\,&=\, \E\left[ \left|F^\delta(x)+\sum_{i=1}^d V_i^\delta (x)B_\delta^i  \right|^{2k} \right]\\
    &=\,\sum_{\ell=0}^k \binom{2k}{2\ell} \left|F^\delta(x)\right|^{2\ell}\, \E\left[  \left|\sum_{i=1}^d V_i^\delta (x)B_\delta^i  \right|^{2k-2\ell}  \right]\\
    &\leq \sum_{\ell=0}^k \binom{2k}{2\ell}\,\left(\frac{1}{h^2}\,|x|^2 + 2b_1\,\delta\right)^{\ell}\cdot  K^{2k-2\ell}\,\sum_{i=1}^d \E\left[|B^i_\delta|^{2k-2\ell}\right]\ ,
\end{align*}
with $h:=1+b_0\delta>1$, where the last step follows by \eqref{Fbound} and condition \eqref{Vk_bounded} for the operators $V_i$. Moreover, since $B^i_\delta$ are (independent) Brownian motions,
\begin{equation*}
    \E\left[|B^i_\delta|^{2k-2\ell}\right]\,=\, \frac{(2k-2\ell)!}{(k-\ell)!\, 2^{k-\ell}}\, \delta^{k-\ell}\ ,
\end{equation*}
thus,
\begin{equation*}
    \E_x\left[ \left| \overline{X}^\delta_\delta\right|^{2k} \right]\,\leq\,\sum_{\ell=0}^k \binom{k}{\ell}\,\left(\frac{1}{h^2}\,|x|^2 + 2b_1\,\delta\right)^{\ell}\cdot  \frac{K^{2k-2\ell}}{2^{k-\ell}}\,\frac{2k!\,\ell !}{k!\,2\ell!\, }\, \delta^{k-\ell}\ .
\end{equation*}
To simplify the expression we introduce
\begin{equation*}
    c_k\,:=\, \frac{K^{2}}{2}\cdot \max_{\ell=0,\dots,k}\frac{2k!\,\ell !}{k!\,2\ell!\, }\ ,
\end{equation*}
thus for $k>1$ we can bound the $2k$-th moment as
\begin{align*}
    \E_x\left[ \left| \overline{X}^\delta_\delta\right|^{2k} \right]\,&\leq\, \left(\frac{1}{h^2}\,|x|^2 + (2b_1+c_k)\,\delta\right)^{k}\\
    &\leq\, \left(1+2^{k-1}\,\a\right)\,\frac{1}{h^{2k}}\,|x|^{2k}\,+\, \left(1+2^{k-1}\,\a^{-1}\right)\,(2b_1+c_k)^{k}\,\delta^k\ ,
\end{align*}
for any $\a>0$, where the last inequality follows by Lemma \ref{technical_lemma}. Choosing $\a=b_0\delta\,2^{-k+1}$, we have
\begin{equation*}
    \left(1+2^{k-1}\,\a\right)= h\ ,
\end{equation*}
and 
\begin{equation*}
     \left(1+2^{k-1}\,\a^{-1}\right)\,(2b_1+c_k)^{k}\,\delta^k\,\leq\, C_k\,\delta^{k-1}\ ,
\end{equation*}
where $C_k>0$ is independent of $\delta$. This gives 
\begin{align}
    \label{2kmoment}
    \E_x\left[ \left| \overline{X}^\delta_{\delta}\right|^{2k} \right]\,\leq\, \frac{1}{(1+b_0\delta)^{2k-1}}\,|x|^{2k}\,+\, {C}_k\,\delta^{k-1}\ ,
\end{align}
for any $k>1$. Whereas for $k=1$, we simply have
\begin{equation}
    \label{2moment}
    \E_x\left[ \left| \overline{X}^\delta_{\delta}\right|^{2} \right]\,\leq\, \frac{1}{(1+b_0\delta)^{2}}\,|x|^{2}\,+\, (2b_1+c_1)\,\delta\ .
\end{equation}
Iterating \eqref{2kmoment}, we then obtain UiT $2k$-th moment bounds for $k\in\N$, $k\geq 2$,
\begin{align*}
    \E_x\left[ \left| \overline{X}^\delta_{t_n}\right|^{2k} \right]\,&\leq\, \frac{1}{(1+b_0\delta)^{n(2k-1)}}\,|x|^{2k}\,+\, \frac{C_k\,\delta^{k-1}}{1-\frac{1}{(1+b_0\delta)^{2k-1}}}\\
    &\leq\,\frac{1}{(1+b_0\delta)^{n(2k-1)}}\,|x|^{2k}\,+\, {C}_k\,\delta^{k-2} \ ,
\end{align*}
for $\delta$ small enough and for every $n,k\in\N$ with $k\geq 2$ (the case for $k=1$ follows similarly from \eqref{2moment}). To extend this result to $2q$-th moment bounds with $q\in\R$, $q\geq 1$, choose $k\in\N$ such that $k-1<q\leq k$. By applying H\"older's inequality with $p=\frac{k}{q}\geq 1$, we obtain
\begin{align*}
    \E_x\left[ \left| \overline{X}^\delta_{t_n}\right|^{2q} \right]\,&\leq\,\E_x\left[ \left| \overline{X}^\delta_{t_n}\right|^{2k} \right]^{\frac{q}{k}}\,\leq\,\left(\frac{1}{(1+b_0\delta)^{n(2k-1)}}\,|x|^{2k}\,+\, {C}_k\,\delta^{k-2}\right)^{\frac{q}{k}}\\
    &\leq\,\frac{2}{(1+b_0\delta)^{\frac{q\cdot n(2k-1)}{k}}}\,|x|^{2q}\,+\, C_q\,\delta^{\frac{q(k-2)}{k}}\ ,
\end{align*}
in case $k\geq 2$, whereas
\begin{align*}
    \E_x\left[ \left| \overline{X}^\delta_{t_n}\right|^{2q} \right]\,\leq\,\frac{2}{(1+b_0\delta)^{ 2q \cdot n}}\,|x|^{2q}\,+\, C\ ,
\end{align*}
for $k=1$, where the constants $C_k,\,C_q,\,C>0$ can change from line to line. This proves the statement for any $\delta\in (0,1)$.

\end{proof}

\begin{prop}\label{prop_modified}
    Consider the Euler-Maruyama scheme $\overline{X}^\delta_{t_n}$ \eqref{modified_EM} with step $\delta>0$ for the solution $\overline{x}^\delta_t$ of the modified SDE \eqref{modifiedSDE}. Under Assumption \ref{ass_onesidedLip}, $\overline{X}^\delta_{t_n}$ satisfies conditions \eqref{cond_Euler} and \eqref{cond_bound_Phi}.

    If in addition Assumption \ref{ass_for_modified_expdecay} holds, then condition \eqref{cond_SDE} is satisfied as well, and there exists a constant $K>0$ such that
\begin{equation}\label{modified_UiT}
\sup_{n\in\N} \,\left| \E_x\left[ f(\overline{x}^\delta_{t_n})\right]-\E_x\left[f\left(\overline{X}^{\delta}_{t_n}\right)\right] \right|\,\leq\, K\|f\|_{\CC^2_b}\cdot \delta^{1/2}\, \left(1+|x|^{2q+1}\right)\   ,
\end{equation}
for any $f\in\CC^2_b(\R^N)$ and $\delta>0$ small enough, with $q>0$ as in condition \eqref{cond_poly}. 
\end{prop}

\begin{proof}
By Proposition \ref{prop_weak}, if we prove that conditions \eqref{cond_SDE}, \eqref{cond_Euler} and \eqref{cond_bound_Phi} hold for the modified Euler scheme $\overline{X}^\delta_{t_n}$, then \eqref{modified_UiT} follows. The estimate \eqref{cond_SDE} holds under Assumption \ref{ass_for_modified_expdecay}, by Lemma \ref{lemma_modified_ExpDecay}. Thus, we only need to prove that conditions \eqref{cond_Euler} and \eqref{cond_bound_Phi} are satisfied by $\overline{X}^\delta_{t_n}$ under Assumption \ref{ass_onesidedLip}. More precisely, we show that, under Assumption \ref{ass_onesidedLip}, condition \eqref{cond_Euler} is satisfied with
\begin{equation}\label{modified_phi}
    \phi(x,\delta)\,:=\, \delta^{3/2}\left(1+|x|^{2q+1}\right)\ ,
\end{equation}
with $q>0$ as in condition \eqref{cond_poly}, and that
\begin{equation}\label{modified_moments}
\sup_{n\in\N}\,\E_x\left[\left|\overline{X}^\delta_{t_n}\right|^{2q+1}\right]\,\leq\, 2\,|x|^{2q+1}\,+\,C_q\ ,
\end{equation}
for some constant $C_q>0$ independent of $x$ and $\delta$. Thus, by Lemma \ref{standard_eg}, we have that condition \eqref{cond_bound_Phi} is satisfied with
\begin{equation*}
    \Phi(x,\delta)\,=\, \delta^{3/2}\,\left(2\,|x|^{2q+1}+C_q\right)\ ,
\end{equation*}
and, by Proposition \ref{prop_weak}, we obtain \eqref{modified_UiT}. 

Since \eqref{modified_moments} follows directly from Lemma \ref{lemma_alternative}, we are left to prove \eqref{modified_phi}. Similarly to the proof of Proposition \ref{prop_EM}, we have 
\begin{align*}
    \E_x&\left[f\left(\overline{X}^\delta_\delta\right)-f(\overline{x}^\delta_\delta)\right]\\
    &\;=\,\E_x\bigg[ \int_0^\delta \langle U^\delta_0(x)-U^\delta_0(\overline{X}^\delta_r),\,\nabla \overline{\PP}^\delta_{s-r}f(\overline{X}^\delta_r)\rangle\nonumber\\
    &\qquad\qquad+\,\sum_{k=1}^d \left(V^\delta_k(x)-V^\delta_k(\overline{X}^\delta_r)\right)^T\,\nabla^2\overline{\PP}^\delta_{s-r}f(\overline{X}^\delta_r)\, \left(V^\delta_k(x)+V^\delta_k(\overline{X}^\delta_r)\right)\,dr\bigg]\ .
\end{align*}
By \eqref{Fdelta}, condition \eqref{cond_poly} is satisfied also for $U^\delta_0(x)=U_0(F^\delta(x))$, thus we obtain
\begin{align*}
     \left|\E_x\left[f\left(\overline{X}^\delta_\delta\right)-f(\overline{x}^\delta_\delta)\right]\right|&\\
     \leq\, &C\,\|f\|_{\CC^2_b}\int_0^\delta \E_x\left[\left( 1+|x|^q+ \left|\overline{X}^\delta_r\right|^q\right)\,\left|\overline{X}^\delta_r-x\right|+\left|\overline{X}^\delta_r-x\right|^2\right]\, dr\\
     \leq\, &\begin{multlined}[t]
        C\,\|f\|_{\CC^2_b}\int_0^\delta\bigg( \left[\E_x (1+|x|^{2q}+ \left|\overline{X}^\delta_r\right|^{2q})\right]^{1/2}\cdot\left[\E_x\left|\overline{X}^\delta_r-x\right|^2\right]^{1/2}\\
     +\E_x\left[\left|\overline{X}^\delta_r-x\right|^2\right]\bigg)\,dr\ ,
     \end{multlined}
\end{align*}
for some constant $C>0$, where the last inequality follows by Holder's inequality. By definition of the continuous-time interpolant $\overline{X}^\delta_r$ \eqref{interpolant_overline}, we have
\begin{align*}
     \E_x\left[\left|\overline{X}^\delta_r-x\right|^2\right]\,&=\, \E_x\left[\left|r\,U_0^\delta(x)+\sum_{k=1}^d V^\delta_k(x)\,B_r^k\right|^2\right]\\
     &\leq\,  r^2\,\left|U^\delta_0(x)\right|^2\,+\, K\,r\ .
\end{align*}
Moreover, using \eqref{interpolant_overline} and \eqref{U0delta}, we have
\begin{align*}
    \left[\E_x (1+|x|^{2q}+ \left|\overline{X}^\delta_r\right|^{2q})\right]^{1/2}\,&
    =\,  \left[\E_x\left( 1+|x|^{2q}+ \left|x+r\,U^\delta_0(x)+\sum_{k=1}^d V^\delta_k(x)\,B^k_r\right|^{2q}\right)\right]^{1/2}\\
    &\leq\, C_q\,\left[\E_x\left(1+|x|^{2q}+ r^{2q}\,\frac{|x|^{2q}}{\delta^{2q}}\,+\, \sum_{k=1}^d \left|B^k_r\right|^{2q} \right)\right]^{1/2}\\&\leq\,C_q\,\left(1+|x|^q\right)\ ,
\end{align*}
for any $\delta\in(0,1)$ and $r<\delta$, where the constant $C>0$ (independent of $\delta,\;r$ and $x$) can change from line to line. Combining all together, we obtain
\begin{align*}
\left|\E_x\left[f\left(\overline{X}^\delta_\delta\right)-f(\overline{x}^\delta_\delta)\right]\right|\,&\leq\, C\left(1+|x|^q\right)\int_0^\delta \left(r\,\left|U^\delta_0(x)\right|+r^{1/2}\right)\,dr\\
&\leq\,C\left(1+|x|^q\right)\cdot\left(\delta^2\,\left|U^\delta_0(x)\right|+\delta^{3/2}\right)\\
%&\leq\,C\left(1+|x|^q\right)\cdot\left(\delta^2 \,|x|^{q+1}+\delta^{3/2}\right)\\
&\leq\,C\,\delta^{3/2}\left(|x|^{2q+1}+1\right)\ ,
\end{align*}
by applying again \eqref{cond_poly} for $U^\delta_0$.

\end{proof}

\subsection{UiT Weak Convergence for Split-Step and Implicit Euler Schemes}\label{subsection_weak}

In this subsection, we show how to apply the UiT weak convergence for EM schemes for the modified SDE \eqref{modifiedSDE} (see Proposition \ref{prop_modified}) to obtain the UiT convergence of split-step and implicit Euler schemes for the original SDE \eqref{SDE_ito}, see Theorem \ref{thm_split} and Theorem \ref{thm_implicit}, respectively.

\begin{thm}[UiT weak convergence for the split-step scheme]\label{thm_split}
Let $\left\{Z^\delta_{t_n}\right\}_{n\in\N}$ be the split-step Euler approximation \eqref{SS2}-\eqref{SS1} for the solution $x_t$ of the SDE \eqref{SDE_ito}. 
Under Assumption \ref{ass_onesidedLip} and Assumption \ref{ass_for_modified_expdecay}, there exists a constant $K>0$ such that
\begin{equation*}
\sup_{n\in\N} \,\left| \E_x\left[ f(x_{t_n})\right]-\E_x\left[f\left(Z^{\delta}_{t_n}\right)\right] \right|\,\leq\, K\,\|f\|_{\CC^2_b}\cdot \delta^{1/2}\, \left(1+|x|^{2q+2}\right) ,
\end{equation*}
for any $f\in\CC^2_b(\R^N)$ and $\delta>0$ small enough.

\end{thm}

\begin{proof}
Recall that, under Assumption \ref{ass_onesidedLip}, the split-step Euler approximation $Z^{\delta}_{t_n}$ of the SDE \eqref{SDE_ito} coincides with the Euler-Maruyama approximation $\overline{X}^{\delta}_{t_n}$ of the modified SDE \eqref{modifiedSDE}, namely $Z^\delta_{t_n}=\overline{X}^\delta_{t_n}$ for every $n\in\N$, provided $\overline{X}_{0}^{\delta}=\overline{x}^\delta_0=Z^\delta_0=x$. Hence,
\begin{align*}
    \left| \E_x\left[ f(x_{t_n})\right]-\E_x\left[f\left(Z^{\delta}_{t_n}\right)\right] \right|&=\left| \E_x\left[ f(x_{t_n})\right]-\E_x\left[f\left(\overline{X}^{\delta}_{t_n}\right)\right] \right|\\
    &\leq \left| \E_x\left[ f(x_{t_n})\right]-\E_x\left[f(\overline{x}^\delta_{t_n})\right]\right|+\left|\E_x\left[f(\overline{x}^\delta_{t_n})\right]-\E_x\left[f\left(\overline{X}^{\delta}_{t_n}\right)\right] \right|\\
    &\leq\,K\,\|f\|_{\CC^2_b}\,\delta\left(1+|x|^{2+2q}\right)\,+\,K\|f\|_{\CC^2_b}\cdot \delta^{1/2}\, \left(1+|x|^{2q+1}\right)\ ,
\end{align*}
respectively by Proposition \ref{prop_SDE_modified} and Proposition \ref{prop_modified}.

\end{proof}

\begin{thm}[UiT weak convergence for the Implicit Euler scheme]\label{thm_implicit}
Let $\left\{Y^{\delta}_{t_n}\right\}_{n\in\N}$ be the implicit Euler approximation \eqref{general_implicit} for the solution $x_t$ of the SDE \eqref{SDE_ito}. Under Assumption \ref{ass_onesidedLip} and Assumption \ref{ass_for_modified_expdecay}, there exists a constant $K>0$ such that
\begin{equation*}
\sup_{n\in\N} \,\left| \E_x\left[ f(x_{t_n})\right]-\E_x\left[f\left(Y^{\delta}_{t_n}\right)\right] \right|\,\leq\, K\,\|f\|_{\CC^2_b}\cdot \delta^{1/2}\, \left(1+|x|^{(2q+1)\cdot(q+1)}\right)\ ,
\end{equation*}
for any $f\in\CC^2_b(\R^N)$ and $\delta>0$ small enough. 
\end{thm}

\begin{proof}
Recall that, under Assumption \ref{ass_onesidedLip},  $Y^{\delta}_{t_n}=F^\delta\left(\overline{X}^\delta_{t_n}\right)$, for any $n\in \N$, where $\overline{X}^\delta_{t_n}$ is the Euler-Maruyama discretisation of the modified SDE \eqref{modifiedSDE} with initial state $\overline{X}^\delta_0=Y^\delta_0 - \delta\, U_0(Y^\delta_0)$, $Y^\delta_0=x$ (see \eqref{implicit_as_EM}). Hence, denoting $x^\star:=x-\delta\,U_0(x)$ and $f^\star:=f\circ F^{\delta}$, we have
\begin{subequations}
\begin{align}
    \left| \E_x\left[ f(x_{t_n})\right] -\E_x\left[f\left(Y^{\delta}_{t_n}\right)\right] \right|\nonumber\,=\,&\left| \E_x\left[ f(x_{t_n})\right]-\E_{x^\star}\left[f^\star(\overline{X}^{\delta}_{t_n})\right] \right|\nonumber\\
    \leq\,&\left| \E_x\left[ f(x_{t_n})\right]-\E_x\left[ f(\overline{x}^\delta_{t_n})\right]\right|\label{bound1}\\
    &+\left|\E_x\left[ f(\overline{x}^\delta_{t_n})\right]-\E_{x^\star}\left[ f(\overline{x}^\delta_{t_n})\right]\right|\label{bound2}\\
    &+ \left|\E_{x^\star}\left[ f(\overline{x}^\delta_{t_n})\right]-\E_{x^\star}\left[ f^\star(\overline{x}^\delta_{t_n})\right]  \right| \label{bound3}\\
    &+ \left| \E_{x^\star}\left[ f^\star(\overline{x}^\delta_{t_n})\right] -\E_{x^\star}\left[f^\star(\overline{X}^{\delta}_{t_n})\right]\right|\ .
    \label{bound4}
\end{align}
\end{subequations}
%In what follows the constant $C>0$ can change from line to line.

By Proposition \ref{prop_SDE_modified}, \eqref{bound1} is bounded by
\begin{align*}
    \left|\E_x\left[ f(x_{t_n})\right]-\E_x\left[ f(\overline x^{\delta}_{t_n})\right]  \right| \,&\leq\,\delta\,K\,\|f\|_{\CC^2_b}\,\left(1+|x|^{2+2q}\right)\ .
\end{align*}

Moreover, recalling that $x-x^\star=\delta\,U_0(x) $ and, by condition \eqref{cond_poly}, $\left| U_0(x)\right|\leq c\, \left(1+|x|^{1+q}\right)$, we see that the term \eqref{bound2} is bounded by
\begin{align*}
    \left| \E_x\left[ f(\overline x^{\delta}_{t_n})\right]-\E_{x^\star}\left[ f(\overline x^{\delta}_{t_n})\right]\right|=\big| \overline\PP^\delta_{t_n} f(x)\,-\, \overline\PP^\delta_{t_n} f(x^\star)\big|\,
    &\leq\, \left\|\nabla \overline\PP^\delta_{t_n} f \right\|_{\infty} \left|x-x^\star\right|\\
    &\leq\, C\,\left\| \overline\PP^\delta_{t_n} f \right\|_{\CC^2_b}\,\delta\,\left(1+|x|^{1+q}\right)\\
    &\leq\,  C\,e^{-\sigma\, n \delta}\,\left\| f \right\|_{\CC^2_b}\,\delta\,\left(1+|x|^{1+q}\right) \ ,
\end{align*}
where the last inequality follows by Lemma \ref{lemma_modified_ExpDecay}, having $t_n=n\delta$.

Similarly, we can bound \eqref{bound3} by
\begin{align*}
   \left|\E_{x^\star}\left[ f(\overline x^{\delta}_{t_n})\right]-\E_{x^\star}\left[ f^\star(\overline x^{\delta}_{t_n})\right]\right|
    &\leq \, \left\|\nabla f \right\|_{\infty}\E_{x^\star}\left[ \left|\overline x^{\delta}_{t_n}-F^\delta(\overline x^{\delta}_{t_n}) \right| \right] \\
    &\leq\, C\, \left\|\nabla f \right\|_{\infty}\delta\, \E_{x^\star}\left[1+\left|\overline x^{\delta}_{t_n}\right|^{1+q}\right]\\
    &\leq\, C\,\left\| f \right\|_{\CC^2_b}\, \delta\,\left( 1 +|x^\star|^{1+q} \right) \\
    &\leq\, C\,\left\| f \right\|_{\CC^2_b}\, \delta\,\left( 1 +|x|^{(1+q)^2} \right) \ ,
\end{align*}
for any $f\in \CC^2_b(\R^N)$, where the penultimate step follows by Lemma \ref{lemma_xmoments}, since the hypotheses of Lemma \ref{lemma_xmoments} hold also for the modified SDE \eqref{modifiedSDE}. Indeed, if the Lyapunov condition \eqref{cond_Lyap} for $U_0$ holds, $U^\delta_0$ satisfies a similar Lyapunov condition. More precisely, since $F^\delta(x)=x+\delta\, U^\delta_0(x)$ by construction,
\begin{align*}
   \langle U^\delta_0(x),\,x\rangle\,+\,\delta\,|U^\delta_0(x)|^2 \,&=\,\langle U^\delta_0(x),\,F^{\delta}(x)\rangle\,=\,\langle U_0(F^{\delta}(x)),\,F^{\delta}(x)\rangle\\
    &\leq\, -b_0\,\left(|x|^2+\delta^2|U^\delta_0(x)|^2\,+\,2\delta\,\langle U^\delta_0(x),\,x\rangle\right)\,+\, b_1\ .
\end{align*}
Rearranging, we obtain that the modified vector field $U_0^\delta$ satisfies the Lyapunov condition
\begin{equation*}\label{lyap_delta}
    \langle U^\delta_0(x),\,x\rangle\,\leq\, -\frac{{b}_0}{2}|x|^2\,+\, b_1\ ,
\end{equation*}
for any $\delta\in(0,\,\frac{1}{2\,b_0})$ and $x\in\R^N$, where we have used the basic facts $\frac{b_0}{1+2\delta b_0}>\frac{b_0}{2}$ and $\frac{b_1}{1+2\delta b_0}<b_1$ for simplicity.

Finally, by Proposition \ref{prop_modified}, we can bound \eqref{bound4} by
\begin{align*}
    \left| \E_{x^\star}\left[ f^\star(\overline{x}^\delta_{t_n})\right] -\E_{x^\star}\left[f^\star(\overline{X}^{\delta}_{t_n})\right]\right|\,&\leq\, C\|f^\star\|_{\CC^2_b}\cdot \delta^{1/2}\, \left(1+|x^\star|^{2q+1}\right)\\
    &\leq\, C\,\|F^\delta\|_{\CC^2_b}\,\|f\|_{\CC^2_b}\cdot \delta^{1/2}\, \left(1+|x|^{(2q+1)\cdot(q+1)}\right)\   ,
\end{align*}
for some constant $K>0$. Under Assumption \ref{ass_for_modified_expdecay}, the hypotheses of Lemma \ref{lemma_derivativesF} are satisfied, thus $\|F^\delta\|_{\CC^2_b}\leq C$ uniformly in $\delta>0$. Combining all together we obtain Theorem \ref{thm_implicit}.

\end{proof}

\section{Truncated Tamed Euler Schemes}\label{section_tamed}

In this section, we discuss the UiT weak convergence for a Truncated Tamed Euler (TTE) scheme which we introduce below, see \eqref{tamed}. 

\begin{assumption}\label{ass_tamed}
We assume that there exist constants $q\geq 0$ and $b_0,b_1,c_2>0$ such that the local Lipshitz condition  \eqref{cond_poly} and the following inequality are satisfied (for the same constant $q\geq 0$):
\begin{align}
    \langle U_0(x),x\rangle\,\leq\, -b_0 |x|^{q+2} +b_1\ .\label{tamed_Lyap}
\end{align}
We further assume that conditions \eqref{Vk_globalLip}-\eqref{Vk_bounded} hold for $V_k$, $k=1,\dots,d$.

%For simplicity we assume $q\in \N$, $q\geq 2$. The case $q\in \R$, $q\geq 2$, can be treated in the same fashion as Section \ref{section_split} (see e.g.~Proposition \ref{prop_modified}).
    
\end{assumption}

\begin{remark}
    Note that the conditions in Assumption \ref{ass_tamed} are satisfied for both the examples illustrated in Example \ref{toy_eg}, with $q=2$.
\end{remark}

Under Assumption \ref{ass_tamed}, we define the TTE scheme $\{J^\delta_{t_n}\}_{n\in\N}$ with time-step $\delta$ by
\begin{align}
    J^\delta_{t_{n+1}}\,=\, J^\delta_{t_n}\,+\, \frac{\delta\,U_0(J^\delta_{t_n})}{1\,+\, \delta\alpha\,| J^\delta_{t_n}|^q}\,+\,\sqrt{2}\sum_{k=1}^d V_k(J^\delta_{t_n})\,\Delta B^k_n\ , \qquad
    J^\delta_0\,=\, x\ ,\label{tamed}
\end{align}
with $t_n=n\delta$, $\Delta B_{t_n} = B_{t_{n+1}}-B_{t_n}$, $\a>0$ arbitrary constant, and $q\geq 0$ as introduced in Assumption \ref{ass_tamed}. The TTE scheme \eqref{tamed} with $\a=1$ has been already introduced in \cite{sabanis2016euler} (see also \cite{hu2018convergence}).

\begin{note}\label{note:section5} We make several remarks on the scheme \eqref{tamed}. 
\begin{itemize}
\item To explain the choice of `truncation', i.e. of the drift term in \eqref{tamed} set $N=1$ and suppose $U_0$ grows polynomially as $x^{q+1}$ (see condition \eqref{cond_poly}); then the derivative of $U_0$ grows like $x^q$. So, morally, what we are doing is dividing by $(1+ \delta |U_0'(x)|)$; that is, we are applying a sort of `Newton correction' to the drift, hence the reason why we also refer to the scheme as {\em Newton-tamed Euler}.   
\item Regarding the choice of constant $\alpha$, in Figure  \ref{fig_tamed} we  illustrate  how the behaviour of the truncated tamed Euler scheme \eqref{tamed} is affected by different choices of $\a$ when the norm of the initial datum is large. While the study of the `optimal' choice  of $\alpha$ is outside of the scope of this paper,    the analysis of this section shows that  the constant $\alpha$ needs to be chosen so that the inequalities \eqref{tamed_edelta} in the proof of Proposition \ref{prop_moments_tamed} hold  (such inequalities are instrumental to proving  uniform bounds on the moments of the scheme, we come back to this in the next bullet point).  In the example considered in  Figure \ref{fig_tamed}, \eqref{tamed_edelta} does not hold for $\alpha=1$ (for that example the lower bound    appearing in \eqref{tamed_edelta} can be computed explicitly, since Assumption \ref{ass_tamed} holds for $b_0\leq 1$ and $c_0 = 1$); and indeed,  when $\alpha=1$ the scheme exhibits  steep/non-smooth behaviour, see Fig. \ref{fig_tamed}, right panel. Moreover, comparing different choices of $\alpha$, it seems that the TTE scheme is more accurate as the constant $\alpha$ is smaller, provided it satisfies \eqref{tamed_edelta}.
\item It is natural to compare the scheme \eqref{tamed} with the standard tamed Euler scheme $\{\hat J^\delta_{t_n}\}_{n\in\N}$ with time-step $\delta$ for the SDE \eqref{SDE_ito} (see \cite{hutzenthaler2012strong, brehier2020approximation}), namely with the scheme 
\begin{align}
   \hat J^\delta_{t_{n+1}}\,=\, \hat J^\delta_{t_n}\,+\, \frac{\delta\,U_0(\hat J^\delta_{t_n})}{1\,+\, \delta\,|U_0(\hat J^\delta_{t_n})|}\,+\,\sqrt{2}\sum_{k=1}^d V_k(\hat J^\delta_{t_n})\,\Delta B^k_n\ , \qquad
    \hat J^\delta_0\,=\, x\ .\label{standard_tamed}
\end{align}

The basic idea for the choice of the drift $\hat{U}^\delta_0(x):=\frac{U_0(x)}{1+\delta|U_0(x)|}$ in \eqref{standard_tamed} is that for any given $\delta>0$ the  drift $\hat{U}^\delta_0$ is bounded even if $U_0$ is not, and it converges to $U_0$ as $\delta$ goes to zero. In contrast, the drift in \eqref{tamed} is Lipschitz and grows linearly. Related to this, if we consider our running example $U_0(x) = -x^3-x$, it is easy to check that, for an appropriate choice of $\alpha$,  $|x|^2$ is a Lyapunov function for the scheme \eqref{tamed} (which implies that the second moment of the scheme is bounded){\footnote{To carry out this calculation observe that in this simple case, by taking $d=1$ and $V_1(x)=1$,  the generators for the scheme \eqref{tamed} and \eqref{standard_tamed} are given, respectively,  by $\mathcal L_{TTE} = -\frac{(x^3+x)}{1+\alpha \delta x^2} \partial_x +\partial_{xx} $ and $\mathcal L_C = -\frac{(x^3+x}{1+\alpha \delta |x^3+x|} \partial_x +\partial_{xx}$, respectively, and recall that a function $f=f(x)$ is a Lyapunov function for a generator $\mathcal L$ if there exist constants $a>0,b\geq 0$ such that $\mathcal L f\leq - af+b$.} } but it can never be a Lyapunov function for  \eqref{standard_tamed};  in general the function $|x|^q$  is a Lyapunov function for the scheme \eqref{tamed} but not for \eqref{standard_tamed}. Of course this does not prove that it is not possible to establish uniform in time control on the moments of the classical tamed scheme. However we observe that, consistently, in \cite{brehier2020approximation} the author proves  moment bounds for the tamed Euler scheme \eqref{standard_tamed} that hold with at most polynomial dependence with respect to the time horizon, and applies this result to prove weak error bounds that depend polynomially on time. We do not know how to  improve these error bounds for standard tamed.\\
Our proof of the moment bounds for TTE substantially uses a Lyapunov argument (at least for the second moment). Because of the observations we have just made, that proof clearly won't work for the standard tamed scheme. 
\end{itemize}
\end{note}

\begin{figure}
\begin{center}
\includegraphics[width=0.45\textwidth]{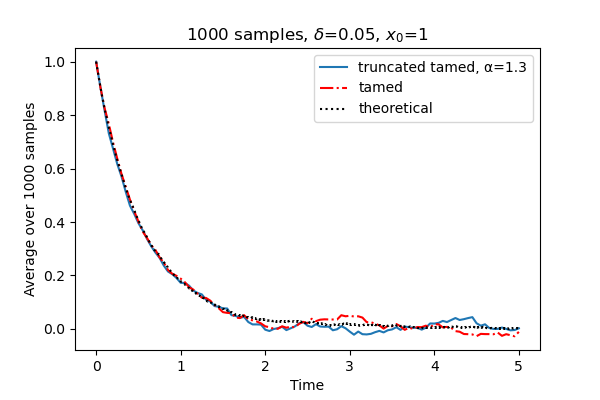}\quad\includegraphics[width=0.45\textwidth]{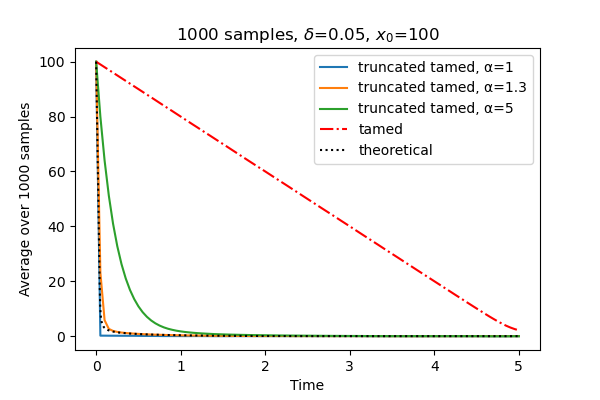}
\end{center}
\caption{\label{fig_tamed} Illustration of standard tamed Euler \eqref{standard_tamed} and truncated tamed Euler schemes \eqref{tamed} (for different constants, namely $\a=1, 1.3, 5$) with step $\delta=0.05$.  We plot  time-approximations of $\E[x_t]$, with $x_t$ solution of $dx_t = -(x^3+x)dt+dB_t$, with initial state $x_0=1$ (left) and $x_0=100$ (right), taking averages over 1000 samples. The solution is not explicitly computable so the benchmark  `theoretical line' is obtained using the standard tamed Euler scheme \eqref{standard_tamed} with $\delta=5\times 10^{-4}$ and using $10^4$ samples. For small initial data the choice of $\alpha$ in TTE does not matter and all the schemes behave substantially in the same way; for large initial data $\alpha$ in TTE needs to be chosen compatibly with \eqref{tamed_edelta},  see Note \ref{note:section5}. The choice $\alpha=1$ violates \eqref{tamed_edelta}, giving rise to the incorrect non-smooth behaviour of the right panel (blue line).  Within the  range of $\alpha$ satisfying \eqref{tamed_edelta}, TTE seems to always outperform the standard tamed scheme.}
\end{figure}

\begin{prop}[Time-uniform second moment bound]\label{prop_moments_tamed}
Under Assumption \ref{ass_tamed}, there exists $\a>0$ such that for every given initial datum $x\in\R^N$ and for $\delta\in(0,\delta_c)$ with $\delta_c>0$ small enough, the truncated tamed Euler scheme \eqref{tamed} with constant $\a$ satisfies the following,
    \begin{equation*}
       \sup_{n\in\N} \E_x\left[\left|J^\delta_{t_n}\right|^{2k} \right]\,\leq\, |x|^{2k}\,+\, C_k \ ,
    \end{equation*}
for every $k\in\N$, for some constant $C_k>0$ independent of $\delta$ and $x$.
\end{prop}

\begin{proof}
    First consider the case $k=1$.  We consider the case $|x|\leq 1$ and $|x|>1$ separately. If $|x|\leq 1$, by the definition of the scheme \eqref{tamed} we simply have
    \begin{align*}
        \E_x\left[\left|J^\delta_{\delta}\right|^2 \right]\,\leq\, \tilde{C}_0\ ,
    \end{align*}    
    for some constant $\tilde{C}_0>0$. Now, let  $|x|>1$. Analogously to what we have done in \eqref{interpolant}-\eqref{freezedL},  consider the continuous time interpolant $J_t^{\delta}$ of the scheme \eqref{tamed} (we don't rewrite the full interpolant and below use analogous notation to the one in \eqref{freezedL}). From Ito formula, for $0\leq t \leq \delta $ we then obtain:

\begin{align*}
    \varphi(J_t^{\delta}) = \varphi(x) + \int_0^t \mathcal L_{(x)} \varphi(J_s^{\delta}) ds 
    + \sqrt{2} \int_0^t \sum_{k=1}^d \sum_{i=1}^N V_k^i(x) \partial_i \varphi(J_s^{\delta}) dB_s^k \,.
\end{align*}
Using the above with $\varphi(x)=|x|^2$ and for $t=\delta$, we have
\begin{align*}
    |J_{\delta}^{\delta}|^2 & = |x|^2 +2 
    \int_0^{\delta} \frac{\langle U_0(x), J_s^{\delta} \rangle}{1+\delta \alpha |x|^q} ds+ 
    2\int_0^{\delta} \sum_k^d\sum_{i,j}^N
    V_k^i(x) V_k^j(x) ds \\  
    &+ \sqrt{2} \int_0^{\delta} \sum_{k}^d \sum_{i}^N 2 V_k^i (J_s^{\delta})^i  dB_s^k \,.
\end{align*}
From \eqref{Vk_bounded}, and by letting $I_{\delta}$ be the martingale term on the second line of the above, we then have
\begin{align}\label{use1Michela}
    |J_{\delta}^{\delta}|^2 & \leq |x|^2 +2 
    \int_0^{\delta} \frac{\langle U_0(x), J_s^{\delta}-x \rangle}{1+\delta \alpha |x|^q} ds+  2 
    \int_0^{\delta} \frac{\langle U_0(x), x \rangle}{1+\delta \alpha |x|^q} ds + C \delta + I_{\delta},  
\end{align}
    where $C$ is a generic constant independent of $x$ and ${\delta}$. Since $0\leq s \leq \delta$, from the definition of the interpolant $J_t^{\delta}$, we have 
    $$
    \mathbb E |J_s^{\delta}-x| \leq \frac{\delta U_0(x)}{1+\delta \alpha |x|^q} +C\delta  \,.$$  
    Hence, taking expectation on both sides of \eqref{use1Michela},  using the above and condition \eqref{tamed_Lyap}, we obtain
    \begin{align}\label{use2Michela}
       \mathbb E |J_{\delta}^{\delta}|^2 & \leq |x|^2  - \frac{2b_0 |x|^{q+2}}{1+\delta \alpha |x|^q} \delta + \frac{ 2b_1}{1+\delta \alpha |x|^q}\delta + 2 \frac{|U_0(x)|}{1+\delta \alpha |x|^q} \frac{\delta |U_0(x)| +C\delta}{1+\delta \alpha |x|^q}\delta + C\delta \,.
    \end{align}
 Finally,   by condition \eqref{cond_poly}, we know that there exist constants $c_0,c_1>0$ such that
    \begin{equation} \label{tamed_bound}
        \left| U_0(x)\right|^2\,\leq\,c_0 |x|^{2q+2}\,+\, c_1\,(1+|x|^{2q})  \,.
    \end{equation}
   Thus, combining \eqref{use2Michela} and \eqref{tamed_bound} we reach
    \begin{align*}\label{tamed_basic_bound}
      \mathbb E |J_{\delta}^{\delta}|^2 & \leq |x|^2   \left( 1-\frac{2b_0|x|^q}{1+\delta \alpha |x|^q} \delta + \frac{2 c_0 |x|^{2q}}{(1+\delta \alpha |x|^q)^2}\delta^2  \right) + C\delta
    \end{align*}
    
    Let
    \begin{equation*}
        g_\delta(x) :=   1-\frac{2b_0|x|^q}{1+\delta \alpha |x|^q} \delta + \frac{2 c_0 |x|^{2q}}{(1+\delta \alpha |x|^q)^2}\delta^2 
    \end{equation*}
As $|x|$ tends to infinity the function $g_{\delta}$ tends to the value $a= 1-2b_0/\alpha +2c_0/\alpha^2$ (which is independent of $\delta$). Moreover, it can be seen that upon choosing $\alpha$ so that
\begin{equation}\label{tamed_edelta}
        0<1+\frac{2c_0}{\alpha^2} - \frac{2b_0}{\alpha}<1
    \end{equation}
  that is
  $$
  \alpha>\max\left\{\frac{c_0}{b_0}, 2b_0-\frac{2c_0}{\alpha}\right\}, 
  $$
 one has  $0<g_{\delta}\leq 1$ for every $x, \delta$. Since the function is symmetric about the $y$ axis it needs only be studied for $x$ say positive, which is what we do next. 
With some help from matlab (or with some caluclation) one can see that the global maximum of the function is attained only for $x=0$, where $g_{\delta}(0)=1$. For $x>0$,  depending on the value of $a$ and of the various parameters,  there are two cases. Either $g_{\delta}$ is monotonically decreasing towards $a$ or, after decreasing towards its global minimum, the function becomes monotonically increasing towards $a$.  In the former case we can bound $g_{\delta}$ from above with a constant $0<\e=1-a<1$, independent of $\delta$ (and it seems to us that, upon choosing $\alpha$ appropriately, one should always be able to be in this scenario). In the latter case, since we are really only interested in the behaviour of $g_{\delta}$ for $x>1$, we upper bound $g_{\delta}$ with $g_{\delta}(x)\leq \max\{g_{\delta}(1), a\} = \e_{\delta}$. We therefore obtain    
    \begin{equation*}
\E_x\left[\left|J^\delta_{\delta}\right|^2 \right]\,\leq\, \max\left\{\tilde{C}_0,\, \tilde\e |x|^2\,+\,(C+K)\delta \right\}\, 
    \end{equation*}
    where $\tilde \e$ is either $\e$ or $\e_{\delta}$. 
    By a standard recursion argument we obtain 
    \begin{align*}
\sup_{n\in\N}\E_x\left[\left|J^\delta_{n\delta}\right|^2 \right]\,
 \leq \max\left\{\tilde{C}_0,\, |x|^2 \right\} \,+\, (C+K) \frac{\delta}{1-\tilde \e}   \,.
    \end{align*}
The case $k>1$ can be dealt with similarly (and with a little combinatorics, analogous to the one in Lemma \ref{lemma_alternative}), so we omit detail. 
\end{proof}

 \begin{thm}[UiT weak convergence for truncated tamed Euler schemes]\label{thm_weak_tamed}
Consider the truncated tamed Euler scheme $J^\delta_{t_n}$ \eqref{tamed} with step $\delta>0$ for the solution $x_t$ of the SDE \eqref{SDE_ito}. Under Assumption \ref{ass_tamed}, $J^\delta_{t_n}$ satisfies conditions \eqref{cond_Euler} and \eqref{cond_bound_Phi}. 

If in addition Assumption \ref{ass_expdecay2} holds, then condition \eqref{cond_SDE} is satisfied as well, and there exists a constant $K>0$ such that
\begin{equation}\label{tamed_UIT}
\sup_{n\in\N} \,\left| \E_x\left[ f(x_{t_n})\right]-\E_x\left[f\left(J^{\delta}_{t_n}\right)\right] \right|\,\leq\, K\|f\|_{\CC^2_b}\,\left(\delta\,|x|^{2q}\,+\,\delta^{1/2}\right)\ ,
\end{equation}
for any $f\in\CC^2_b(\R^N)$ and $\delta>0$ small enough. 
\end{thm}

\begin{proof}
By Proposition \ref{prop_weak}, if we prove that conditions \eqref{cond_SDE}, \eqref{cond_Euler} and \eqref{cond_bound_Phi} hold for the truncated tamed Euler scheme $J^\delta_{t_n}$, then \eqref{tamed_UIT} follows. The estimate \eqref{cond_SDE} holds under Assumption \ref{ass_expdecay2}, by Lemma \ref{lemma_expdecay2}. Thus, we are only left to prove that conditions \eqref{cond_Euler} and \eqref{cond_bound_Phi} are satisfied by $J^\delta_{t_n}$ under Assumption \ref{ass_tamed}. More precisely, we show that, under Assumption \ref{ass_tamed}, Assumption \ref{ass_weak} is satisfied and the statement follows by Proposition \ref{prop_weak}. That is, we want to show that condition \eqref{cond_Euler} is satisfied with
\begin{equation}\label{tamed_phi}
    \phi(x,\delta)\,:=\,\delta^2|x|^{2q+2}+\delta^{3/2}\ ,
\end{equation}
where $q\geq 2$ is given by Assumption \ref{ass_tamed}, and that
\begin{equation}\label{tamed_moments}
\sup_{n\in\N}\,\E_x\left[\left|J^\delta_{t_n}\right|^{2q+2}\right]\,\leq\, |x|^{2q+2}\,+\,C_{q}\ .
\end{equation}
Therefore, similarly to Lemma \ref{standard_eg}, condition \eqref{cond_bound_Phi} is satisfied with
\begin{equation*}
    \Phi(x,\delta)\,=\, \delta^2|x|^{2q+2}\,+\, C_{q}\,\delta^{2}\,+\,\delta^{3/2}\ .
\end{equation*}
By Proposition \ref{prop_weak}, we then obtain the statement. 

The bound \eqref{tamed_moments} is given by Proposition \ref{prop_moments_tamed}. Thus, we are left to prove \eqref{tamed_phi}. In what follows, the constant $C>0$ may change from line to line.

Denoting by $J^\delta_r$, $r\in(0,\delta)$, the continuous-time interpolant of $J^\delta_{t_n}$, i.e.
\begin{equation}\label{tamed_interpolant}
    dJ^\delta_r\,=\, \frac{U_0(x)}{1+\delta\a|x|^q}\,dr\,+\, \sum_{k=1}^d V_k(x)\, dB^k_r\ ,\qquad X^\delta_0\,=\,x\ ,
\end{equation}
applying Lemma \ref{lemma_P-Q}, we have 
\begin{align*}
   \E_x\left[f(x^\delta_\delta)-f\left(J^\delta_\delta\right)\right]\,
    =&\,\E_x\bigg[ \int_0^\delta \langle U_0(J^\delta_r)\,-\, \frac{U_0(x)}{1+\delta\alpha|x|^q},\,\nabla {\PP}^\delta_{s-r}f(J^\delta_r)\rangle\nonumber\\
    &+\,\sum_{k=1}^d \left(V_k(J^\delta_r)-V_k(x)\right)^T\,\nabla^2{\PP}^\delta_{s-r}f(J^\delta_r)\, \left(V_k(J^\delta_r)-V_k(x)\right)\,dr\bigg]\ .
\end{align*}
Note that
\begin{align*}
    \left|U_0(J^\delta_r)\,-\, \frac{U_0(x)}{1+\delta\alpha|x|^q}\right|\,&\leq\,\left|U_0(J^\delta_r)\,-\, U_0(x)\right|\,+\, \left|U_0(x)\,-\, \frac{U_0(x)}{1+\delta\alpha|x|^q}\right|\\
    &\leq \,\left|U_0(J^\delta_r)\,-\, U_0(x)\right|\,+\, \delta\alpha|x|^q\,|U_0(x)|\\
    &\leq\, C\,\left(1+|x|^{q}+|J^\delta_r|^{q}\right)\,|x-J^\delta_r|\,+\, C\,\a\delta\left(|x|^{2q+1}\,+\,1\right)\ ,
\end{align*}
by condition \eqref{cond_poly}. Thus, we obtain
\begin{align*}
    \big|\E_x\big[f\big(J^\delta_\delta\big)\,&-\,f(x^\delta_\delta)\big]\big|\\
     \leq\, &C\,\|f\|_{\CC^2_b}\int_0^\delta \E_x\left[\left(1+|x|^{q}+|J^\delta_r|^{q}\right)\,|x-J^\delta_r|\,+\, \delta\alpha\,|x|^{2q+1}+\left|J^\delta_r-x\right|^2\right]\, dr\\
     \leq\, &\begin{multlined}[t]
    C\,\|f\|_{\CC^2_b}\int_0^\delta\bigg( \delta\,|x|^{2q+1}\,+\,\left[\E_x \left( 1+|x|^{2q}+ \left|J^\delta_r\right|^{2q}\right)\right]^{1/2}\cdot\left[\E_x\left|J^\delta_r-x\right|^2\right]^{1/2}\\
     +\E_x\left[\left|J^\delta_r-x\right|^2\right]\bigg)\,dr\ ,
     \end{multlined}
\end{align*}
where the last inequality follows by Holder's inequality. By definition of the continuous-time interpolant $J^\delta_r$ \eqref{tamed_interpolant}, we have
\begin{align*}
     \E_x\left[\left|J^\delta_r-x\right|^2\right]\,&=\, \E_x\left[\left|r\,\frac{U_0(x)}{1+\delta\,\alpha |x|^q}+\sum_{k=1}^d V_k(x)\,B_r^k\right|^2\right]\\
     &\leq\, r^2\,\left|U_0(x)\right|^2\,+\, K\,r\\
     &\leq\, C\, r^2 (|x|^{2q+2}+1)\,+\,K r\ .
\end{align*}
Moreover, using again \eqref{tamed_interpolant}, we have
\begin{align*}
    \E_x\left[ 1+|x|^{2q}+ \left|J^\delta_r\right|^{2q}\right]^{1/2}\,&=\,  \E_x\left[ 1+|x|^{2q}+ \left|x+r\,\frac{U_0(x)}{1+\delta\alpha |x|^q}\,+\,\sum_{k=1}^d V_k(x)\,B^k_r\right|^{2q}\right]^{1/2}\\
    &\leq\, C\,\E_x\left[1+|x|^{2q}\,+ \,\frac{r^{2q}}{\delta^{2q}}\,|x|^{2} \,+\, \sum_{k=1}^d \left|B^k_r\right|^{2q} \right]^{1/2}\\&\leq\,C\,\left(1+|x|+|x|^{q}\right)\ ,
\end{align*}
for any $\delta\in(0,1)$ and $r<\delta$. Combining all together, we obtain
\begin{align*}
\left|\E_x\left[f\left(J^\delta_\delta\right)-f(x_\delta)\right]\right|\,&\leq\, C\,\delta^2\,|x|^{2q+1}\,+\, C\left(1+|x|+|x|^{q}\right)\int_0^\delta \left(r\,|x|^{q+1}+r^{1/2}\right)\,dr\\
&\leq\,C\delta^2\,|x|^{2q}\,+\,C\left(1+|x|+|x|^{q}\right)\cdot\left(\delta^2\,|x|^{q+1}+\delta^{3/2}\right)\\
&\leq\,C\,\left(\delta^2|x|^{2q+2}+\delta^{3/2}\right)\ ,
\end{align*}
by applying again \eqref{tamed_bound}. This gives \eqref{tamed_phi} and concludes the proof.

\end{proof}
\bigskip

{\bf Acknowledgments.} M. Ottobre and L. Angeli gratefully acknowlegde the support  of the Leverhulme grant RPG--2020--095. 
The authors are extremely grateful to an anonymous referee who read the manuscript extremely carefully, giving plenty of useful feedback and spotting a number of typos and to Can Huang (Xiamen University) who found a mistake in an earlier version of the proof of Proposition 5.3. 
\appendix
\section{Proofs of exponential decay}\label{appendix_expdecay}

Given a $(N\times N)$-matrix $B$ and a function $A:\R^N\to \R^{N\times N}$, we denote by $\|B\|:=\sqrt{\sum_{i,j}|b_{ij}|^2}$ the Frobenius norm, and, with a slight abuse of notation, $\|A\|:=\sup_{x\in\R^N}\|A(x)\|$. Moreover, we define the seminorm
\begin{equation*}
    \|F^\delta\|_{\CC^2_b}\,=\, \sup_{x\in\R^N} \left(\left(\sum_{i=1}^N \left|\partial_i F^\delta(x)\right|^2\right)^{1/2}\,+\, \left(\sum_{i,j=1}^N \left|\partial_j\partial_i F^\delta(x)\right|^2\right)^{1/2}\right).
\end{equation*}
In some of the proofs we also use the 2-norm of a matrix, namely $\|B\|_2 :=\sup_{v\neq 0} \frac{|Bv|}{|v|}$ (recall that $\|B\|_2\leq \|B\|$).

\begin{proof}[Proof of Lemma \ref{lemma_expdecay1}]
We adapt the proof of \cite[Proposition 4.5]{crisan2022poisson}.

Fix $f\in\CC^2_b(\R^N)$. In order to prove the statement, it is enough to show that there exists a constant $\gamma>0$ such that
\begin{equation}\label{cond_nabla}
    \partial_s \PP_{t-s}|\nabla \PP_s f(x)|^2\leq -2\gamma\, \PP_{t-s} |\nabla \PP_s f(x)|^2\ ,
\end{equation}
for any $0\leq s\leq t$ and $x\in\R^N$.

Indeed, this gives 
\begin{equation*}
    |\nabla \PP_t f(x)|^2\leq e^{-2\gamma t}\, \PP_{t} |\nabla f(x)|^2\ ,
\end{equation*}
by Gronwall's lemma. By the positivity property of Markov semigroups, 
\begin{equation*}
    \sup_{t\geq 0}\,\sup_{x\in\R^N}\PP_{t} |\nabla f(x)|^2\,\leq\, \|\nabla f\|^2_{\infty}\ ,
\end{equation*}
the statement then follows. 

So, we are left to prove \eqref{cond_nabla}. Writing
\begin{equation*}
    \partial_s \PP_{t-s}|\nabla \PP_s f(x)|^2 \,=\, \E_x\left[\partial_s |\nabla \PP_s f(x_{t-s})|^2-\LL |\nabla \PP_s f(x_{t-s})|^2\right]\ ,
\end{equation*}
where $\LL$ and $x_t$ are respectively the generator and the process associated with the semigroup $\PP_t$, we see that \eqref{cond_nabla} holds if 
\begin{equation}\label{bound_gamma}
    \partial_s |\nabla \PP_s f(x)|^2-\LL |\nabla \PP_s f(x)|^2\,\leq\, -2\gamma\,  |\nabla \PP_s f(x)|^2\ ,
\end{equation}
for every $x\in\R^N$ and $s\geq 0$. Expanding the terms at the LHS, we can write
\begin{align}
    \partial_s |\nabla \PP_s f(x)|^2\,&=\,\partial_s\left(\sum_{i=1}^N |\partial_i\PP_s f(x)|^2\right) \nonumber  \\ &=\,2\sum_{i=1}^N \left(\partial_i \PP_s f(x)\right)\,\left(\partial_i \LL \PP_s f(x)\right),\label{partial_s_nabla}
\end{align}
with
\begin{align*}
    \partial_i \LL \PP_s f(x)\,=\, \langle \partial_i U_0(x),\,\nabla \PP_s f(x)\rangle+\langle  U_0,\,\nabla \partial_i \PP_s f(x)\rangle
    +\,\partial_i \sum_{k=1}^d V^T_k(x) \left(\nabla^2 \PP_s f(x)\right)V_k(x)\ ,
\end{align*}
and
\begin{align}\label{LL_nabla}
    \LL |\nabla \PP_s f(x)|^2\,&=\, \langle U_0(x),\,\nabla \left( |\nabla \PP_s f(x)|^2\right)\rangle\,+\,\sum_{k=1}^d V^T_k(x) \left(\nabla^2 \left(|\nabla \PP_s f(x)|^2\right)\right)V_k(x)\ ,
\end{align}
where
\begin{align*}
   \langle U_0(x),\,\nabla \left( |\nabla \PP_s f(x)|^2\right)\rangle\,&=\, \langle U_0(x),\,\nabla \left( \sum_{i=1}^N |\partial_i\PP_s f(x)|^2\right)\rangle \\
   &=\, 2\sum_{i=1}^N \left(\partial_i \PP_s f(x)\right)\,\langle  U_0(x),\,\nabla \partial_i \PP_s f(x)\rangle\ ,
\end{align*}
and
\begin{align*}
    \nabla^2 \left(|\nabla \PP_s f(x)|^2\right)\,&=\,\nabla^2 \left(\sum_{i=1}^N |\partial_i\PP_s f(x)|^2\right)  \\
    &=\,2\sum_{i=1}^N \left(\nabla \partial_i \PP_s f(x)\right)\,\otimes\,\left(\nabla \partial_i \PP_s f(x)\right)\,+\,2\sum_{i=1}^N \left(\partial_i \PP_s f(x)\right)\,\nabla^2 \partial_i \PP_s f(x)\ .
\end{align*}
Therefore, combining \eqref{partial_s_nabla} and \eqref{LL_nabla}, we have
\begin{align}
    \partial_s |\nabla \PP_s f(x)|^2& - \LL  |\nabla \PP_s f(x)|^2\nonumber \\
    =\,&2\sum_{i=1}^N \left(\partial_i \PP_s f(x)\right)\, \langle \partial_i U_0(x),\,\nabla \PP_s f(x)\rangle-\,2\sum_{k=1}^d\sum_{i=1}^N \langle \nabla \partial_i \PP_s f(x),\,V_k(x) \rangle^2\nonumber\\
    &+\,4\sum_{i=1}^N\sum_{k=1}^d \left(\partial_i \PP_s f(x)\right)\,\left( V^T_k(x) \left(\nabla^2 \PP_s f(x)\right)\partial_i V_k(x)\right)\ .\label{ds-L}
\end{align}

The first term in the decomposition \eqref{ds-L} is simply
\begin{align*}
2\sum_{i=1}^N \left(\partial_i \PP_s f(x)\right)\, \langle \partial_i U_0(x),\,\nabla \PP_s f(x)\rangle\,=\,2\,\nabla \PP_s f(x)\, \nabla U_0(x)\,\nabla \PP_s f(x)\ .
\end{align*}
Moreover, using the basic fact
\begin{equation*}
    u^T \nabla^2 \PP_s f(x) v\,=\, \sum_{j=1}^N \left( u^T  \nabla^2 \PP_s f(x)\right)_j v_j\,=\, \sum_{j=1}^N  u^T  \partial_j\nabla \PP_s f(x) v_j\,=\, \sum_{j=1}^N \langle u, \partial_j \nabla\PP_s f(x)\rangle\, v_j\ ,
\end{equation*}
for every $u,v\in\R^N$, we can rewrite the last term in \eqref{ds-L} as
\begin{align*}
    4\sum_{i=1}^N&\sum_{k=1}^d  V^T_k(x)\left(\nabla^2 \PP_s f(x)\right)\left(\partial_i \PP_s f(x)\,\partial_i V_k(x)\right)\\
    &=\,4\sum_{i,j=1}^N\sum_{k=1}^d \langle V_k(x),\,\partial_j\nabla\PP_s f(x) \rangle\,\left(\partial_i \PP_s f(x)\,\partial_i V_k(x)\right)_j\\
    &\leq\,2\sum_{i,j=1}^N\sum_{k=1}^d \left(\frac{1}{N}\langle V_k(x),\,\partial_j\nabla\PP_s f(x) \rangle^2\,+\,N\left(\partial_i \PP_s f(x)\,\partial_i V_k(x)\right)_j^2\right)\ ,
\end{align*}
where the last step follows by Young's inequality.
Writing
\begin{equation*}
    2\sum_{i,j=1}^N\sum_{k=1}^d \frac{1}{N}\langle V_k(x),\,\partial_j\nabla\PP_s f(x) \rangle^2\,=\, 2\sum_{j=1}^N\sum_{k=1}^d \langle V_k(x),\,\partial_j\nabla\PP_s f(x) \rangle^2\ ,
\end{equation*}
and
%\begin{align*}
%    2\sum_{i,j=1}^N\sum_{k=1}^d N\left(\partial_i \PP_s f(x)\,\partial_i V_k(x)\right)_j^2\,&=\, 2N\, \sum_{i=1}^N\sum_{k=1}^d \left|\partial_i \PP_s f(x)\,\partial_i V_k(x)\right|^2\\
%    &\leq \, 2N\,\sum_{k=1}^d  |\nabla \PP_s f(x)|^2\, \|\nabla V_k(x)\|^2\ ,
%\end{align*}
\begin{align*}
    2\sum_{i,j=1}^N\sum_{k=1}^d N\left(\partial_i \PP_s f(x)\,\partial_i V_k(x)\right)_j^2\,=\, 2N\, \sum_{i=1}^N\sum_{k=1}^d \left|\partial_i \PP_s f(x)\,\partial_i V_k(x)\right|^2\ ,
    %&\leq \, 2 \,\left(\left|\lambda_N(x)\right|-\gamma\right) \sum_{i=1}^N \left|\partial_i \PP_s f(x)\right|^2\ ,
\end{align*}
we obtain that \eqref{ds-L} is bounded by
\begin{multline}\label{need_it_later}
     \partial_s |\nabla \PP_s f(x)|^2 -\, \LL  |\nabla \PP_s f(x)|^2\\
     \leq\,2\,\left(\nabla \PP_s f(x)\right)^T \nabla U_0(x) \left(\nabla \PP_s f(x)\right)\,+\,2N\, \sum_{i=1}^N\sum_{k=1}^d \left|\partial_i \PP_s f(x)\,\partial_i V_k(x)\right|^2\ .
     %&\leq \,2\,\left(\nabla \PP_s f(x)\right)^T \nabla U_0(x) \left(\nabla \PP_s f(x)\right)\,+\,2\,\left(\left|\lambda_N(x)\right|-\gamma\right) \left|\nabla \PP_s f(x)\right|^2\ ,
\end{multline}
Hence, applying condition \eqref{cond_nablaU0} and condition \eqref{cond_expdecay_V}, we obtain
\begin{align*}
    \partial_s |\nabla \PP_s f(x)|^2 -\, \LL  |\nabla \PP_s f(x)|^2 \,&\leq\, -2\,\lambda(x)\,\left|\nabla \PP_s f(x)\right|^2 \, +\, 2\,\left(\lambda(x)-\gamma\right) \left|\nabla \PP_s f(x)\right|^2\\
    &=\, -2\gamma\,\left|\nabla \PP_s f(x)\right|^2\ .
\end{align*}

Namely, \eqref{bound_gamma} holds with $-2\gamma <0$, and this concludes the proof.

\end{proof}

\begin{proof}[Proof of Lemma \ref{lemma_expdecay2}]
First, note that if we prove
\begin{align}\label{second_D}
     |\nabla \PP_t f(x)|^2+c\,\|\nabla^2 \PP_t f(x)\|^2\,\leq\, e^{-\sigma t} \PP_{t} \left(|\nabla f(x)|^2+c\,\|\nabla^2 f(x)\|^2\right),
\end{align}
for some $c>0$, it ensures the following bound,
\begin{equation*}
    c\,\|\nabla^2 \PP_t f(x)\|^2\leq \max\{1,c\}\, e^{-\sigma t} \PP_{t} \left(|\nabla f(x)|^2+\|\nabla^2 f(x)\|^2\right)\ .
\end{equation*}
Hence, by the positivity property of Markov semigroups, 
\begin{equation*}
    \sup_{t\geq 0}\,\sup_{x\in\R^N} \PP_{t} \left(|\nabla f(x)|^2+\|\nabla^2 f(x)\|^2\right)\,\leq\, \|\nabla f\|^2_{\infty}+\|\nabla^2 f\|^2_{\infty}\ ,
\end{equation*} 
this proves the statement, with $\tilde{K}_0 = \frac{\max\{1,c\}}{c}$.

    Thus we have to show \eqref{second_D}. Similarly to the proof of Lemma \ref{lemma_expdecay1}, it is enough to prove the following bound,
    \begin{multline}\label{to_prove_expdecay2}
    \partial_s \left(|\nabla \PP_s f(x)|^2+c\,\|\nabla^2 \PP_s f(x)\|^2\right)-\LL \left(|\nabla \PP_s f(x)|^2+c\,\|\nabla^2 \PP_s f(x)\|^2\right)\\
    \leq\, -\sigma  \left(|\nabla \PP_s f(x)|^2+c\,\|\nabla^2 \PP_s f(x)\|^2\right)\ ,
\end{multline}
for every $x\in\R^N$ and $s\geq 0$. Similarly to the proof of Lemma \ref{lemma_expdecay1}, we obtain \eqref{need_it_later}, namely
\begin{align}
     \partial_s |\nabla \PP_s f&(x)|^2 -\, \LL  |\nabla \PP_s f(x)|^2\nonumber\\
     &\leq\,2\,\left(\nabla \PP_s f(x)\right)^T \nabla U_0(x) \left(\nabla \PP_s f(x)\right)\,+\,2N\, \sum_{i=1}^N\sum_{k=1}^d \left|\partial_i \PP_s f(x)\,\partial_i V_k(x)\right|^2 \nonumber \\
     &\leq \,2\left((\rho-1)\,\lambda(x)-\gamma\right) \, \left|\nabla \PP_s f(x)\right|^2\ ,\label{from_previous_lemma}
\end{align}
where the latter bound follows by condition \eqref{cond_nablaU0} and condition \eqref{cond_expdecay_Vbis}.

So, we are left to bound
\begin{equation*}
    \partial_s \|\nabla^2 \PP_s f(x)\|^2-\LL \|\nabla^2 \PP_s f(x)\|^2\ .
\end{equation*}
We can write
\begin{align}
    \partial_s \|\nabla^2 \PP_s f(x)\|^2= \partial_s \left(\sum_{i,j=1}^N \left|\partial_i\partial_j \PP_s f(x)\right|^2\right) = 2\sum_{i,j=1}^N \left(\partial_i\partial_j \PP_s f(x)\right)\,\left(\partial_i\partial_j \LL \PP_s f(x)\right)\ ,\label{partial_s_nabla2}
\end{align}
with
\begin{align*}
    \partial_i\partial_j \LL \PP_s f(x)\,=\, \partial_i\partial_j \langle U_0(x),\,\nabla \PP_t f(x)\rangle \,+\,\partial_i\partial_j \sum_{k=1}^d V^T_k(x) \left(\nabla^2 \PP_t f(x)\right)V_k(x)\ .
\end{align*}

Moreover,
\begin{align}\label{LL_nabla2}
    \LL \|\nabla^2 \PP_s f(x)\|^2\,&=\, \langle U_0(x),\,\nabla \|\nabla^2 \PP_s f(x)\|^2\rangle\,+\,\sum_{k=1}^d V^T_k(x) \left(\nabla^2 \|\nabla^2 \PP_s f(x)\|^2\right)V_k(x)\ ,
\end{align}
where
\begin{align*}
   \langle U_0(x),\,\nabla \|\nabla^2 \PP_s f(x)\|^2\rangle\,&=\, \langle U_0(x),\,\nabla \left( \sum_{i,j=1}^N |\partial_j\partial_i\PP_s f(x)|^2\right)\rangle \\
   &=\, 2\sum_{i,j=1}^N \left(\partial_j\partial_i \PP_s f(x)\right)\,\langle  U_0(x),\,\nabla \partial_j\partial_i \PP_s f(x)\rangle\ ,
\end{align*}
and
\begin{align*}
    \nabla^2&\|\nabla^2 \PP_s f(x)\|^2\,=\,\nabla^2 \left(\sum_{i,j=1}^N |\partial_j\partial_i\PP_s f(x)|^2\right)  \\
    &=\,2\sum_{i=1}^N \left(\nabla \partial_j\partial_i \PP_s f(x)\right)\,\otimes\,\left(\nabla \partial_j\partial_i \PP_s f(x)\right)\,+\,2\sum_{i=1}^N \left(\partial_j\partial_i \PP_s f(x)\right)\,\nabla^2 \partial_j\partial_i \PP_s f(x)\ .
\end{align*}
Therefore, combining \eqref{partial_s_nabla2} and \eqref{LL_nabla2}, we have
\begin{subequations}
\begin{align}
    \partial_s \|\nabla^2 \PP_s f(x)\|^2& - \LL  \|\nabla^2 \PP_s f(x)\|^2\nonumber \\
    =\,&2\sum_{i,j=1}^N \left(\partial_j\partial_i \PP_s f(x)\right)\cdot \langle \partial_j\partial_i U_0(x),\,\nabla \PP_s f(x)\rangle\label{dsL_term1}\\
    &+2\sum_{i,j=1}^N \left(\partial_j\partial_i \PP_s f(x)\right)\cdot\langle\partial_i U_0(x),\,\nabla \partial_j \PP_s f(x)\rangle\label{dsL_term2}\\
    &-\,2\sum_{k=1}^d\sum_{i,j=1}^N \left(\langle \nabla \partial_j\partial_i \PP_s f(x),\,V_k(x) \rangle\right)^2\label{dsL_term3}\\
    &+\,4\sum_{i,j=1}^N\sum_{k=1}^d \left(\partial_j\partial_i \PP_s f(x)\right)\,\left( V^T_k(x) \left(\nabla^2 \PP_s f(x)\right)\partial_j\partial_i V_k(x)\right)\label{dsL_term4}\\
    &+\,4\sum_{i,j=1}^N\sum_{k=1}^d \left(\partial_j\partial_i \PP_s f(x)\right)\,\left( \partial_j V^T_k(x) \left(\nabla^2 \PP_s f(x)\right)\partial_i V_k(x)\right)\label{dsL_term5}\\
    &+\,4\sum_{i,j=1}^N\sum_{k=1}^d \left(\partial_j\partial_i \PP_s f(x)\right)\,\left( V^T_k(x) \left(\partial_j\nabla^2 \PP_s f(x)\right)\partial_i V_k(x)\right)
    \ .\label{dsL_term6}
\end{align}
\end{subequations}
First, we consider the terms \eqref{dsL_term2}, \eqref{dsL_term3} and \eqref{dsL_term6}, and we proceed in the same fashion of Lemma \ref{lemma_expdecay1}. More precisely, we can rewrite \eqref{dsL_term6} as
\begin{align*}
    4\sum_{i,j=1}^N\sum_{k=1}^d&  V^T_k(x) \left(\partial_j\nabla^2 \PP_s f(x)\right)\,\left(\partial_j\partial_i \PP_s f(x)\,\partial_i V_k(x)\right)\\    &=4\sum_{i,j,\ell=1}^N\sum_{k=1}^d\langle V_k(x),\,\partial_\ell\partial_j\nabla \PP_s f(x) \rangle \,\left(\partial_j\partial_i \PP_s f(x)\,\partial_i V_k(x)\right)_\ell  \\
    &\leq 2\sum_{i,j,\ell=1}^N\sum_{k=1}^d \left( \frac{1}{N}\langle V_k(x),\,\partial_\ell\partial_j\nabla\PP_s f(x) \rangle^2\,+\,N\left(\partial_j\partial_i \PP_s f(x)\,\partial_i V_k(x)\right)_\ell^2\right) \ ,
\end{align*}
where the last bound follows by Young's inequality. Writing
\begin{align*}
    2\sum_{i,j,\ell=1}^N\sum_{k=1}^d  \frac{1}{N}\langle V_k(x),\,\partial_\ell\partial_j\nabla\PP_s f(x) \rangle^2\,=\, 2\sum_{i,j=1}^N\sum_{k=1}^d \langle V_k(x),\,\partial_i\partial_j\nabla\PP_s f(x) \rangle^2\ ,
\end{align*}
and
\begin{align*}
    2\sum_{i,j,\ell=1}^N\sum_{k=1}^d N\left(\partial_j\partial_i \PP_s f(x)\,\partial_i V_k(x)\right)_\ell^2\, &=\, 2N\sum_{i,j=1}^N \sum_{k=1}^d  \left|\partial_j\partial_i \PP_s f(x)\,\partial_i V_k(x)\right|^2\\
    &\leq\, 2\left(\rho\,\lambda(x)-\gamma\right)\,\sum_{j=1}^N |\nabla \partial_j \PP_s f(x)|^2\ ,
\end{align*}
where the last step follows by \eqref{cond_expdecay_Vbis}, we obtain
\begin{align*}
    \eqref{dsL_term2}+\eqref{dsL_term3}+ \eqref{dsL_term6}\,\leq\, &2\sum_{j=1}^N \left(\nabla\partial_j \PP_s f(x)\right)^T\, \nabla U_0(x)\,\nabla\partial_j \PP_s f(x)\\
    &\,+\, 2\left(\rho\,\lambda(x)-\gamma\right)\,\sum_{j=1}^N |\nabla \partial_j \PP_s f(x)|^2\\
    \leq\,& -2\left((1-\rho)\,\lambda(x)+\gamma\right)\,\|\nabla^2 \PP_s f(x)\|^2\ ,
\end{align*}
by condition \eqref{cond_nablaU0}. Moreover, we have
\begin{align*}
    \eqref{dsL_term1}\,&\leq\, 2\,\left|\nabla \PP_s f(x)\right|\sum_{i,j=1}^N \left|\partial_j\partial_i \PP_s f(x)\right|\, \left| \partial_j\partial_i U_0(x)\right|   \\    
    &\leq\,\e^{-1} \sum_{i,j=1}^N \left|\partial_j\partial_i \PP_s f(x)\right|^2\, \left| \partial_j\partial_i U_0(x)\right|\,+\, \e\, \left|\nabla \PP_s f(x)\right|^2 \sum_{i,j=1}^N  \left| \partial_j\partial_i U_0(x)\right| \\
    &\leq\, \e^{-1}\,\|\nabla^2 \PP_s f(x)\|^2\,\sum_{i,j=1}^N  \left| \partial_j\partial_i U_0(x)\right|\,+\, \e\,\left|\nabla \PP_s f(x)\right|^2 \sum_{i,j=1}^N  \left| \partial_j\partial_i U_0(x)\right| \ ,
\end{align*}
for any $\e>0$, where the second bound follows by Young's inequality. Thus, by condition \eqref{cond_expdecay2_2d},
\begin{equation*}
    \eqref{dsL_term1}\,\leq\,\e^{-1}\,\a\left(1+\lambda(x)\right)\,\|\nabla^2 \PP_s f(x)\|^2\,+\, \e\,\a\left(1+\lambda(x)\right)\,\left|\nabla \PP_s f(x)\right|^2 \ .
\end{equation*}

Finally, by conditions \eqref{cond_expdecay_V2} and \eqref{cond_expdecay_Vbis}, we can bound respectively \eqref{dsL_term4} and \eqref{dsL_term5} by
\begin{align*}
    \eqref{dsL_term4} \,\leq\, 4\left(\rho\,\lambda(x)-\gamma\right)\, \left\|\nabla^2\PP_s f(x)\right\|^2 \ ,
\end{align*}
and, similarly,
\begin{align*}
    \eqref{dsL_term5} \,& =\, 4\sum_{i,j=1}^N\sum_{k=1}^d \left(\partial_j\partial_i \PP_s f(x)\right)\,\left( \partial_j V^T_k(x) \left(\nabla^2 \PP_s f(x)\right)\partial_i V_k(x)\right)\\
    &\leq\, 4\,\left(\sup_{i,j=1,\dots,N}\sum_{k=1}^d \left|\partial_i V^T_k(x) \right|\,\left|\partial_j V^T_k(x) \right|\right)\, \sum_{i,j=1}^N \left|\partial_j\partial_i \PP_s f(x)\right|\cdot \left\|\nabla^2 \PP_s f(x)\right\| \\
    &\leq\, 4\left(\rho\,\lambda(x)-\gamma \right)\,\left\|\nabla^2\PP_s f(x)\right\|^2\ ,
\end{align*}
where the last inequality follows by writing
\begin{equation*}
    \sup_{i,j=1,\dots,N}\sum_{k=1}^d \left|\partial_i V^T_k(x) \right|\,\left|\partial_j V^T_k(x) \right|\,\leq\, \sup_{i=1,\dots,N}\sum_{k=1}^d \left|\partial_i V^T_k(x) \right|^2\ ,
\end{equation*}
by Cauchy-Schwarz inequality, and then applying condition \eqref{cond_expdecay_Vbis}.

Thus, combining together all the terms of the decomposition, from \eqref{dsL_term1} to \eqref{dsL_term6}, we obtain
\begin{align*}
    \partial_s\|\nabla^2 \PP_s &f(x)\|^2 - \LL  \|\nabla^2 \PP_s f(x)\|^2\\ \leq\,&\e^{-1}\,\a\left(1+\lambda(x)\right)\,\|\nabla^2 \PP_s f(x)\|^2\,+\, \e\,\a\left(1+\lambda(x)\right)\,\left|\nabla \PP_s f(x)\right|^2 \\
    &-2\left((1-\rho)\lambda(x)+\gamma\right)\,\|\nabla^2 \PP_s f(x)\|^2\,
    +4\left(\rho\,\lambda(x)-\gamma\right)\, \left\|\nabla^2\PP_s f(x)\right\|^2\\&+\,4\left(\rho\,\lambda(x)-\gamma \right)\,\left\|\nabla^2\PP_s f(x)\right\|^2\\
    =\,& \left((\e^{-1}\a-2+10\rho )\,\lambda(x)  \,+\,  \e^{-1}\a-10\gamma\right)\,\|\nabla^2 \PP_s f(x)\|^2\\
    &+\,\e\,\a\left(1+\lambda(x)\right)\,\left|\nabla \PP_s f(x)\right|^2
\end{align*}
Now, to prove \eqref{to_prove_expdecay2}, we combine the above bound with \eqref{from_previous_lemma} and we write
\begin{equation*}
    \partial_s \left(|\nabla \PP_s f(x)|^2+c\,\|\nabla^2 \PP_s f(x)\|^2\right)-\LL \left(|\nabla \PP_s f(x)|^2+c\,\|\nabla^2 \PP_s f(x)\|^2\right)\,\leq\, A_1 + A_2\ ,
\end{equation*}
with
\begin{align*}
    A_1\, :=&\,2\left((\rho-1)\,\lambda(x)-\gamma\right) \, \left|\nabla \PP_s f(x)\right|^2\,+\, c\e\,\a\left(1+\lambda(x)\right)\,\left|\nabla \PP_s f(x)\right|^2\\
    =&\, \left((2\rho-2+c\e\a)\lambda(x)\,- 2\gamma+c\e\a\right)\,\left|\nabla \PP_s f(x)\right|^2 ,
\end{align*}
and
\begin{align*}
    A_2\,:=\,& c\,\left((\e^{-1}\a-2+10\rho )\,\lambda(x)  \,+\,  \e^{-1}\a-10\gamma\right)\,\|\nabla^2 \PP_s f(x)\|^2\ .
\end{align*}
Provided $\rho<\frac{1}{5}$, we can choose $\e>0$ so that $\e^{-1}\a\leq 2-10\rho$ and $\e^{-1}\a<10\gamma$. Thus, there exists a constant $\sigma_2>0$ such that $A_2\leq - c\,\sigma_2\,\left\|\nabla^2 \PP_s f(x)\right\|^2<0 $. 

Then, once we have chosen $\e>0$, we can choose $c>0$ such that $c\e\a\leq 2(1-\rho)$ and $c\e\a<2\gamma$. Thus, there exists a constant $\sigma_1>0$ such that $A_1\leq -\sigma_1\,\left|\nabla\PP_s f(x)\right|^2<0 $.

Choosing $\sigma=\min\{\sigma_1,\sigma_2\}$, we obtain \eqref{to_prove_expdecay2}. Therefore, \eqref{second_D} holds and this concludes the proof.

\end{proof}

\begin{proof}[Proof of Lemma \ref{lemma_modified_ExpDecay}]
    We first prove that, if Assumption \ref{ass_for_modified_expdecay} holds for $U_0$ and $V_k$, $k=1,\dots,N$, then similar conditions to Assumption \ref{ass_expdecay2} hold for $U^\delta_0$, $V^\delta_k$, $k=1,\dots,N$. However, the main difference is that the constants (namely $\a,\;\rho$ and $\gamma$) of these conditions for $U^\delta_0$ and $V^\delta_k$ depend on $\delta>0$. Hence, if we apply directly Lemma \ref{lemma_expdecay1} and Lemma \ref{lemma_expdecay2} we would obtain that there exist constants $\sigma_\delta,\, K_\delta>0$ (depending on $\delta$) such that
    \begin{equation*}
         \|\PP^\delta_t f\|_{\CC^2_b}\,\leq\,K_\delta\,e^{-\sigma_\delta t}\, \| f\|_{\CC^2_b}\ .
    \end{equation*}
    Therefore, to show that in fact the constants $K$ and $\sigma$ in the exponential decay of the first two derivatives of $\PP^\delta_t$ do not depend on $\delta$, we will go through the proof of Lemma \ref{lemma_expdecay1} and Lemma \ref{lemma_expdecay2} again, to highlight all the dependencies on $\delta$.

    First of all, note that by the proof of Lemma \ref{lemma_derivativesF} we know that \eqref{1dF_bound} and \eqref{2dF_bound} hold. Applying condition \eqref{cond_expdecay2_2d_formodified} to \eqref{2dF_bound}, we obtain 
    \begin{align}
        \left|\partial_j\partial_i F^\delta (x)\right|\,&\leq\, \delta\,\left( 1+\delta \lambda(y)\right)^{-3}\,\left(\sum_{m,m'=1}^N \left|\partial_m\partial_{m'} U_0(y)\right|^2 \right)^{1/2}\nonumber\\
        &\leq\, \frac{\delta\,\a\,\lambda(y)}{\left( 1+\delta \lambda(y)\right)^{3}}\ ,\label{2dFbis}
    \end{align}
    for every $i,j=1,\dots,N$, $\delta\in(0,1)$, where $y=F^\delta(x)$.\\

    To simplify the presentation, we write $y=F^\delta(x)$ and $A(y):=\1\,-\,\delta\, \nabla U_0(y)$ throughout the proof.

    STEP 1: we show that, under condition \eqref{cond_nablaU0_formodified}, $\nabla U^\delta_0$ satisfies
    \begin{equation}
         v^T\,\nabla U^\delta_0(x)\,v\,\leq\,  -\,\frac{\lambda(F^\delta(x))}{\beta^2\,\left(1+\delta\, \lambda(F^\delta(x))\right)}\,|v|^2\ ,\label{1dUdelta_bound_withdelta}
    \end{equation}
    thus condition \eqref{cond_nablaU0} holds with positive function $\tilde{\lambda}_\delta(x):=\frac{\lambda(F^\delta(x))}{\beta^2\,\left(1+\delta\, \lambda(F^\delta(x))\right)}$, depending on $\delta$ (and $F^\delta$). To show this, note that 
    \begin{equation*}
        \nabla U^\delta_0(x)\,=\,\nabla U_0(y)\big|_{y=F^\delta(x)}\,\nabla F^\delta(x)\,=\, \nabla U_0(y)\big|_{y=F^\delta(x)}\,\left(\1\,-\,\delta\, \nabla U_0(y)\big|_{y=F^\delta(x)} \right)^{-1}\ ,
    \end{equation*}
    by \eqref{dF}. We have
    \begin{align}
        v^T\,\nabla U^\delta_0(x)\,v\,&=\, (A(y)\,A(y)^{-1}v)^T \,\nabla U_0(y)\, A(y)^{-1}v\nonumber\\
        &=\, (A(y)^{-1}v)^T\,\left(\1\,-\,\delta\, \nabla U_0(y)\right)^{T}\,\nabla U_0(y)\, (A(y)^{-1}v)\nonumber\\
        &\leq\, -\lambda(y)\,|A(y)^{-1}v|^2\,-\, \delta\, \left|\nabla U_0(y)\,A(y)^{-1}v\right|^2\ ,\label{step1_bound}
    \end{align}
    where the last step follows by \eqref{cond_nablaU0_formodified}. Moreover, again by \eqref{cond_nablaU0_formodified}, we see that
    \begin{equation*}
        \left(1+\delta \,\lambda(y)\right)\,|v|^2\,\leq\, v^T A(y) v\,\leq\, \left(1+\delta \,\beta\, \lambda(y)\right)\,|v|^2\ ,
    \end{equation*}
    in particular $\inf_{v\neq 0} \frac{|A(y)v|}{|v|}\,\geq\, 1+\delta \,\lambda(y)$ and thus $\|A(y)^{-1}\|_2\,\leq\, \left(1+\delta \,\lambda(y)\right)^{-1}$. Similarly, we can see that $\inf_{v\neq 0} \frac{|A(y)^{-1}v|}{|v|}\,\geq\, \left(1+\delta \,\beta\, \lambda(y)\right)^{-1}$ and $\inf_{v\neq 0} \frac{|\nabla U_0(y) v|}{|v|}\,\geq\, \lambda(y)$ by \eqref{cond_nablaU0_formodified}. Thus, from \eqref{step1_bound} we obtain
    \begin{align*}
        v^T\,\nabla U^\delta_0(x)\,v\,&\leq\, -\,\frac{\lambda(y)}{\left(1+\delta\,\beta\,\lambda(y)\right)^2}\,|v|^2\,-\, \frac{\delta\,\lambda(y)^2}{\left(1+\delta\,\beta\, \lambda(y)\right)^2}\,|v|^2\nonumber\\
        &=\, -\,\frac{\lambda(y)\cdot\left(1+\delta\,\lambda(y)\right)}{\left(1+\delta\,\beta\, \lambda(y)\right)^2}\,|v|^2\nonumber\\
        &\leq\,  -\,\frac{\lambda(F^\delta(x))}{\beta^2\,\left(1+\delta\, \lambda(F^\delta(x))\right)}\,|v|^2\ ,
    \end{align*}
    where the last step follows by noting that $\beta\geq 1$ in \eqref{cond_nablaU0_formodified} and recalling that $y=F^\delta(x)$. This proves \eqref{1dUdelta_bound_withdelta}. \\

    STEP 2: we prove that \begin{equation}\label{2dUdelta_bound}
        \sum_{j,i=1}^N \left|\partial_j \partial_i U_0^\delta(x)\right|\,\leq\, \frac{\tilde{\a}\,\lambda(F^\delta(x)) }{\left(1+\delta \lambda(F^\delta(x))\right)^2}\ ,
    \end{equation}
    with $\tilde{\a}=N^2(1+\beta \sqrt{N})\a$, for every $\delta\in(0,1)$ and $x\in\R^N$. Thus, condition \eqref{cond_expdecay2_2d} is satisfied for the second derivatives of $U^\delta_0$, and, more precisely, we show that the second derivatives are bounded uniformly in $x\in\R^N$ for any given $\delta\in(0,1)$. Indeed, writing
    \begin{equation*}
        \left(\partial_i U^\delta_0(x)\right)_k\,=\, \langle \nabla (U_0(y))_k\big|_{y=F^\delta(x)},\,\partial_i F^\delta(x)\rangle\ ,
    \end{equation*}
    denoting $(v)_k:=\langle v,e_k\rangle$, with $\{e_1,\dots,e_N\}$ standard basis of $\R^N$, we have
    \begin{equation*}
        \left(\partial_j\partial_i U^\delta_0(x) \right)_k\,=\, \langle \nabla^2 (U_0(y))_k\big|_{y=F^\delta(x)}\,\partial_j F^\delta(x),\,\partial_i F^\delta(x)\rangle\,+\,\langle \nabla (U_0(y))_k\big|_{y=F^\delta(x)},\,\partial_j\partial_i F^\delta(x)\rangle\ .
    \end{equation*}

    By Cauchy-Schwarz and writing $\nabla (U_0(y))_k\,=\, \nabla U_0(y)\,e_k$, we have
    \begin{align*}
        \left|\partial_j \partial_i U^\delta_0(x)\right|\,\leq\, &\left|\partial_j F^\delta(x)\right|\,\left|\partial_i F^\delta(x)\right|\,\left(\sum_k \left\|\nabla^2 (U_0(y))_k\right\|^2\right)^{1/2}\\
        &+\, \left|\partial_j\partial_i F^\delta(x)\right|\,\left(\sum_k \left|\nabla U_0(y)\,e_k\right|^2\right)^{1/2}\\
        \leq\,& \left|\partial_j F^\delta(x)\right|\,\left|\partial_i F^\delta(x)\right|\,\left(\sum_{m,m'} \left|\partial_m \partial_{m'} U_0(y)\right|^2\right)^{1/2}\,+\,\beta \sqrt{N}\,\lambda(y)\, \left|\partial_j\partial_i F^\delta(x)\right|\ ,
    \end{align*}
    where the last step follows by noting that \eqref{cond_nablaU0_formodified} gives $|\nabla U_0(y)v|\leq \beta\,\lambda(y)\,|v|$. By \eqref{1dF_bound} and \eqref{2dFbis} and applying condition \eqref{cond_expdecay2_2d_formodified}, we obtain
    \begin{align*}
        \left|\partial_j \partial_i U^\delta_0(x)\right|\,\leq\,\frac{\a\,\lambda(y)}{\left(1+\delta \lambda(y)\right)^2}\,
        +\, \frac{\beta \sqrt{N}\,\lambda(y)\cdot \delta\,\a\,\lambda(y)}{\left( 1+\delta \lambda(y)\right)^{3}}\,\leq\, \frac{(1+\beta \sqrt{N})\a\,\lambda(y) }{\left(1+\delta \lambda(y)\right)^2}  \ .
    \end{align*}
    Hence, \eqref{2dUdelta_bound} holds for every $\delta\in(0,1)$ and $x\in\R^N$.\\
    
    STEP 3: We now show that the first and second derivatives of the modified vector fields $V^\delta_k$, $k=1,\dots,d$ can be bounded respectively by
    \begin{align}
      \sup_{i=1,\dots,N}\sum_{k=1}^d  \left| \partial_i V^\delta_k(x)\right|^2 \, &\leq\,\frac{\frac{\rho}{N\,\beta^2}\,\lambda(y)}{1+\delta\,\lambda(y)}\ ,\label{1dVdelta_bound}\\
      \sup_{i,j=1,\dots,N}\sum_{k=1}^d \left|V^\delta_k(x)\right|\, \left| \partial_i \partial_j V^\delta_k(x)\right|\,&\leq\, \frac{\frac{\rho}{\beta^2}\,\lambda(x)}{\left(1+\delta\,\lambda(y)\right)^2}\ .\label{2dVdelta_bound}
    \end{align}
    
    By definition of $V^\delta_k(x)=V_k(F^\delta(x))$, we see that
    \begin{equation*}
         \left| \partial_i V^\delta_k(x)\right|\,=\, \left| \nabla V_k(y)\big|_{y=F^\delta(x)} \partial_i F^\delta(x) \right| \,\leq\, \frac{ \left\|\nabla V_k(y)\right\|}{1+\delta\,\lambda(y)}\ ,
    \end{equation*}
    %by condition \eqref{cond_expdecay_Vbis_formodified}
    by \eqref{1dF_bound}.

    Moreover, in the same fashion as STEP 2, we can write
    \begin{align*}
        \left|\partial_j \partial_i V^\delta_k(x)\right|\,\leq\, &\left|\partial_j F^\delta(x)\right|\,\left|\partial_i F^\delta(x)\right|\,\left(\sum_\ell \left\|\nabla^2 (V_k(y))_\ell\right\|^2\right)^{1/2}\\
        &+\, \left|\partial_j\partial_i F^\delta(x)\right|\,\left(\sum_\ell \left|\nabla V_k(y)\,e_\ell\right|^2\right)^{1/2}\\
        \leq\,& \frac{1}{\left(1+\delta\,\lambda(y)\right)^2}\,\left(\sum_{m,m'} \left| \partial_m \partial_{m'} V_k(y) \right|^2\right)^{1/2}\,+\,\frac{\delta\,\a\,\lambda(y)\,\left\|\nabla V_k(y) \right\|}{\left( 1+\delta \lambda(y)\right)^{3}} \ ,
    \end{align*}
    for every $k=1,\dots,d$, where the last step follows by \eqref{2dFbis}. 

    Hence, under assumptions \eqref{cond_expdecay_Vbis_formodified} and \eqref{cond_expdecay_V2_formodified}, we obtain \eqref{1dVdelta_bound} and \eqref{2dVdelta_bound}.

    STEP 4: We finally prove the exponential decay of $\|\nabla \PP^\delta_t f\|^2_\infty+\|\nabla^2 \PP^\delta_t f\|^2_\infty$. In the exact same fashion as the proof of Lemma \ref{lemma_expdecay1}, we obtain \eqref{need_it_later} for the semigroup associated to the modified SDE \eqref{modifiedSDE}, namely
    \begin{multline*}
     \partial_s |\nabla \PP^\delta_s f(x)|^2 -\, \LL^\delta  |\nabla \PP^\delta_s f(x)|^2\\
     \leq\,2\,\left(\nabla \PP^\delta_s f(x)\right)^T \nabla U^\delta_0(x) \left(\nabla \PP^\delta_s f(x)\right)\,+\,2N\, \sum_{i=1}^N\sum_{k=1}^d \left|\partial_i \PP^\delta_s f(x)\,\partial_i V^\delta_k(x)\right|^2\ ,
\end{multline*}
where we denote by $\LL^\delta$ the generator associated to $\PP^\delta_s$.

Thus, by \eqref{1dUdelta_bound_withdelta} and \eqref{1dVdelta_bound}, we obtain
\begin{align}
    \partial_s |\nabla \PP^\delta_s f(x)|^2 -\, \LL^\delta  |\nabla \PP^\delta_s f(x)|^2\,\leq\, \frac{2\left(\rho-1\right)\,\lambda(y)}{\beta^2\,\left(1+\delta\,\lambda(y)\right)}\,\left|\nabla \PP^\delta_s f(x)\right|^2\ .\label{firstderivatives_modified}
\end{align}
Since $\rho\leq 1$ and $\delta\in(0,1)$, recalling that $\inf_{x\in\R^N} \lambda(x) =:\lambda^\star >0$, we have
\begin{align*}
    \partial_s |\nabla \PP^\delta_s f(x)|^2 -\, \LL^\delta  |\nabla \PP^\delta_s f(x)|^2\,&\leq\, \frac{2\left(\rho-1\right)\,\lambda(y)}{\beta^2\,\left(1+\lambda(y)\right)}\,\left|\nabla \PP^\delta_s f(x)\right|^2\\
    &\leq\, \frac{2}{\beta^2}\,\left(\rho-1\right)\,\frac{\lambda^\star}{1+\lambda^\star}\,\left|\nabla \PP^\delta_s f(x)\right|^2\ ,
\end{align*}

hence, we obtain in the same fashion as Lemma \ref{lemma_expdecay1} the exponential decay for the first derivatives, i.e.
\begin{equation*}
         \|\nabla \PP^\delta_t f\|^2_{\infty}\,\leq\,e^{-2\gamma t}\,\|\nabla f\|^2_{\infty}\ ,
    \end{equation*}

with $\gamma=\frac{2}{\beta^2}\left(\rho-1\right)\frac{\lambda^\star}{1+\lambda^\star}>0$.

Moreover, proceeding similarly to the proof of Lemma \ref{lemma_expdecay2}, we obtain exactly \eqref{dsL_term1}-\eqref{dsL_term6} for the semigroup of the modified SDE \eqref{modifiedSDE}, namely
\begin{subequations}
\begin{align}
    \partial_s \|\nabla^2 \PP^\delta_s f(x)\|^2& - \LL^\delta  \|\nabla^2 \PP^\delta_s f(x)\|^2\nonumber \\
    =\,&2\sum_{i,j=1}^N \left(\partial_j\partial_i \PP^\delta_s f(x)\right)\cdot \langle \partial_j\partial_i U^\delta_0(x),\,\nabla \PP^\delta_s f(x)\rangle\label{modified_dsL_term1}\\
    &+2\sum_{i,j=1}^N \left(\partial_j\partial_i \PP^\delta_s f(x)\right)\cdot\langle\partial_i U^\delta_0(x),\,\nabla \partial_j \PP^\delta_s f(x)\rangle\label{modified_dsL_term2}\\
    &-\,2\sum_{k=1}^d\sum_{i,j=1}^N \left(\langle \nabla \partial_j\partial_i \PP^\delta_s f(x),\,V^\delta_k(x) \rangle\right)^2\label{modified_dsL_term3}\\
    &+\,4\sum_{i,j=1}^N\sum_{k=1}^d \left(\partial_j\partial_i \PP^\delta_s f(x)\right)\,\left( V^\delta_k(x)^T \left(\nabla^2 \PP^\delta_s f(x)\right)\partial_j\partial_i V^\delta_k(x)\right)\label{modified_dsL_term4}\\
    &+\,4\sum_{i,j=1}^N\sum_{k=1}^d \left(\partial_j\partial_i \PP^\delta_s f(x)\right)\,\left( \partial_j V^\delta_k(x)^T \left(\nabla^2 \PP^\delta_s f(x)\right)\partial_i V^\delta_k(x)\right)\label{modified_dsL_term5}\\
    &+\,4\sum_{i,j=1}^N\sum_{k=1}^d \left(\partial_j\partial_i \PP^\delta_s f(x)\right)\,\left( V^\delta_k(x)^T \left(\partial_j\nabla^2 \PP^\delta_s f(x)\right)\partial_i V^\delta_k(x)\right)
    \ .\label{modified_dsL_term6}
\end{align}
\end{subequations}
We then proceed in the exact same fashion as proof of Lemma \ref{lemma_expdecay2}, but applying the bounds \eqref{1dUdelta_bound_withdelta}-\eqref{2dVdelta_bound} instead of those in Assumption \ref{ass_expdecay2}.
More precisely, by \eqref{1dVdelta_bound} and \eqref{1dUdelta_bound_withdelta},
\begin{align*}
    \eqref{modified_dsL_term2}+\eqref{modified_dsL_term3}+ \eqref{modified_dsL_term6}\,\leq\, &2\sum_{j=1}^N \left(\nabla\partial_j \PP^\delta_s f(x)\right)^T\, \nabla U^\delta_0(x)\,\nabla\partial_j \PP^\delta_s f(x)\\
    &\,+\, 2N\sum_{i,j=1}^N \sum_{k=1}^d  \left|\partial_j\partial_i \PP^\delta_s f(x)\,\partial_i V^\delta_k(x)\right|^2\\
    %\leq\, &2\sum_{j=1}^N \left(\nabla\partial_j \PP^\delta_s f(x)\right)^T\, \nabla U^\delta_0(x)\,\nabla\partial_j \PP^\delta_s f(x)\\
    %&\,+\, 2\left(\rho\,\frac{\lambda(F^\delta(x))}{\beta^2\left(1+\lambda(F^\delta(x))\right)}\,-\,\gamma\right)\,\sum_{j=1}^N |\nabla \partial_j \PP^\delta_s f(x)|^2\\
    \leq\,& 2\,(\rho-1)\,\frac{\lambda(F^\delta(x))}{\beta^2\left(1+\delta\,\lambda(F^\delta(x))\right)}\,\|\nabla^2 \PP^\delta_s f(x)\|^2\ .
\end{align*}
By condition \eqref{2dUdelta_bound},
\begin{equation*}
    \eqref{modified_dsL_term1}\,\leq\, \frac{\e^{-1}\,\tilde{\a}\,\lambda(F^\delta(x)) }{\left(1+\delta\, \lambda(F^\delta(x))\right)^2}\,\|\nabla^2 \PP^\delta_s f(x)\|^2\,+\, \frac{\e\,\tilde{\a}\,\lambda(F^\delta(x)) }{\left(1+\delta\, \lambda(F^\delta(x))\right)^2}\,\left|\nabla \PP^\delta_s f(x)\right|^2\ ,
\end{equation*}
for any $\e>0$, which we can choose according to $\delta$ too. Finally, by \eqref{2dVdelta_bound} and \eqref{1dVdelta_bound}, we have
\begin{equation*}
    \eqref{modified_dsL_term4}\,\leq\, 4\,\rho\,\frac{\lambda(F^\delta(x))}{\beta^2\left(1+\delta\,\lambda(F^\delta(x))\right)}\,\left\| \nabla^2 \PP^\delta_s f(x)\right\|^2\ ,
\end{equation*}
and
\begin{equation*}
    \eqref{modified_dsL_term5}\,\leq\, 4\rho\,\frac{\lambda(F^\delta(x))}{\beta^2\left(1+\delta\,\lambda(F^\delta(x))\right)}\,\left\| \nabla^2 \PP^\delta_s f(x)\right\|^2\ .
\end{equation*}

Thus, using the above inequalities for \eqref{modified_dsL_term1} to \eqref{modified_dsL_term6}, we obtain
\begin{align*}
    \partial_s \|\nabla^2 & \PP^\delta_s f(x)\|^2 - \LL^\delta  \|\nabla^2 \PP^\delta_s f(x)\|^2\\
    \leq\,&2(\rho-1)\,\frac{\lambda(F^\delta(x))}{\beta^2\left(1+\delta\,\lambda(F^\delta(x))\right)}\,\|\nabla^2 \PP^\delta_s f(x)\|^2\\
    &+\,\frac{\e^{-1}\,\tilde{\a}\,\lambda(F^\delta(x)) }{\left(1+\delta\, \lambda(F^\delta(x))\right)^2}\,\|\nabla^2 \PP^\delta_s f(x)\|^2\,+\, \frac{\e\,\tilde{\a}\,\lambda(F^\delta(x)) }{\left(1+\delta\, \lambda(F^\delta(x))\right)^2}\,\left|\nabla \PP^\delta_s f(x)\right|^2\\
    &+\,8\rho\,\frac{\lambda(F^\delta(x))}{\beta^2\left(1+\delta\,\lambda(F^\delta(x))\right)}\,\left\| \nabla^2 \PP^\delta_s f(x)\right\|^2\\
    =\,& \left((10\rho-2)\,\frac{\lambda(F^\delta(x))}{\beta^2\left(1+\delta\,\lambda(F^\delta(x))\right)}\,+\, \frac{\e^{-1}\,\tilde{\a}\,\lambda(F^\delta(x)) }{\left(1+\delta \lambda(F^\delta(x))\right)^2}\right)\,\|\nabla^2 \PP^\delta_s f(x)\|^2\\&+\,\frac{\e\,\tilde{\a}\,\lambda(F^\delta(x)) }{\left(1+\delta \lambda(F^\delta(x))\right)^2}\,\left|\nabla \PP^\delta_s f(x)\right|^2 \ .
\end{align*}
Combining this with \eqref{firstderivatives_modified}, we obtain
\begin{equation*}
    \partial_s \left(|\nabla \PP^\delta_s f(x)|^2+c\,\|\nabla^2 \PP^\delta_s f(x)\|^2\right)-\LL^\delta \left(|\nabla \PP^\delta_s f(x)|^2+c\,\|\nabla^2 \PP^\delta_s f(x)\|^2\right)\,\leq\, A_1 + A_2\ ,
\end{equation*}
with
\begin{align*}
    A_1\,=\, \left(2({\rho}-1)\,\frac{\lambda(F^\delta(x))}{\beta^2\,(1+\delta\,\lambda(F^\delta(x)))}\,+\,c\,\frac{\e\,\tilde{\a}\,\lambda(F^\delta(x)) }{\left(1+\delta \lambda(F^\delta(x))\right)^2}\right)\,\left|\nabla \PP^\delta_s f(x)\right|^2\ ,
\end{align*}
and
\begin{align*}
    A_2\,=\, c\,\left((10\rho-2)\,\frac{\lambda(F^\delta(x))}{\beta^2\left(1+\delta\,\lambda(F^\delta(x))\right)}\,+\, \frac{\e^{-1}\,\tilde{\a}\,\lambda(F^\delta(x)) }{\left(1+\delta \lambda(F^\delta(x))\right)^2}\right)\,\|\nabla^2 \PP^\delta_s f(x)\|^2\ .
\end{align*}
Note that, if $\rho\in(0,\frac{1}{5})$, then 
\begin{align*}
    A_1\,&\leq\, \frac{2\rho-2+c\e\,\tilde{\a}\,\beta^2}{\beta^2\,\left(1+\delta \lambda(F^\delta(x))\right)^2}\,\lambda(F^\delta(x))\,\left|\nabla \PP^\delta_s f(x)\right|^2\ ,\\
    A_2\,&\leq\, c\,\frac{10\rho-2+\e^{-1}\tilde{\a}\,\beta^2}{\beta^2\,\left(1+\delta \lambda(F^\delta(x))\right)^2}\,\lambda(F^\delta(x))\,\|\nabla^2 \PP^\delta_s f(x)\|^2\ .
\end{align*}
We can then choose $\e>0$ such that $10\rho-2+\e^{-1}\tilde{\a}\,\beta^2=:-\gamma_1 <0$, so that 
$$A_2\leq -\,\frac{c\,\gamma_1\,\lambda^\star}{\beta^2(1+\lambda^\star)^2}\,\|\nabla^2 \PP^\delta_s f(x)\|^2\ ,$$
for every $\delta\in(0,1)$, where $\lambda^\star=\inf_x \lambda(x)>0$. Once we have chosen $\e>0$, we can then choose $c>0$ such that $2\rho-2+c\e\,\tilde{\a}\,\beta^2=:-\gamma_2<0$. This way we have
\begin{equation*}
    A_1\,\leq\, -\,\frac{\gamma_2\,\lambda^\star}{\beta^2(1+\lambda^\star)^2}\,\left|\nabla \PP^\delta_s f(x)\right|^2\ ,
\end{equation*}
for every $\delta\in(0,1)$. Hence, we obtain
\begin{multline*}
    \partial_s \left(|\nabla \PP^\delta_s f(x)|^2+c\,\|\nabla^2 \PP^\delta_s f(x)\|^2\right)-\LL^\delta \left(|\nabla \PP_s f(x)|^2+c\,\|\nabla^2 \PP_s f(x)\|^2\right)\\
    \leq\, -\sigma  \left(|\nabla \PP_s f(x)|^2+c\,\|\nabla^2 \PP_s f(x)\|^2\right)\ ,
\end{multline*}
with $\sigma=\frac{\lambda^\star\,\min\left\{\gamma_1,\,\gamma_2 \right\}}{(1+\lambda^\star)^2}>0$, which is analogous to \eqref{to_prove_expdecay2} for the original process, therefore we can conclude in the same fashion as Lemma \ref{lemma_expdecay2}.

\end{proof}

\section{Technical Lemmas}

\begin{lemma}\label{lemma_P-Q}
    Let $x_t,\,y_t$ be two Feller processes with Markov semigroup $\PP$ and $\QQ$, respectively. Denoting by $\LL^\PP$ and $\LL^\QQ$ the corresponding Markov generators, the following identity holds,
    \begin{equation*}
        \PP_t f(x)-\QQ_t f(x)\,=\,  \int_0^t \E_x\left[\left(\LL^\PP-\LL^\QQ\right)\PP_{t-s} f(y_s)\right]\,ds\ ,
    \end{equation*}
    for every $f\in\CC^2_b(\R^N)$, $x\in\R^N$, and $t\geq 0$.
\end{lemma}

\begin{proof}
By definition of Markov generator, we can write
    \begin{align*}
        \E_x\left[\phi_t(y_t)\right]\,=\, \phi_0(x)\,+\, \int_0^t \E_x\left[\partial_s\phi_s(y_s)\,+\,\LL^\QQ\phi_s(y_s) \right]\,ds\ ,
    \end{align*}
for every time-dependent function $\varphi_\cdot\in \CC^2_b([0,\infty)\times\R^N)$. Fixing the final time $t\geq 0$ and choosing $\varphi_s = \PP_{t-s} f$ with $f\in\CC^2_b(\R^N)$, for every $s\in[0,t]$, we obtain
\begin{align*}
    \QQ_t f(x)\,=\,\E_x\left[f(y_t)\right]\,&=\, \PP_t f(x)\,+\, \int_0^t \E_x\left[\partial_s\PP_{t-s} f(y_s)\,+\,\LL^\QQ \PP_{t-s} f(y_s) \right]\,ds\\
    &=\, \PP_t f(x)\,+\, \int_0^t \E_x\left[-\LL^\PP\PP_{t-s} f(y_s)\,+\,\LL^\QQ \PP_{t-s} f(y_s) \right]\,ds\ .
\end{align*}
This gives the statement.

\end{proof}

\begin{lemma}\label{technical_lemma}
For every $a,b>0$, $k\in\N$, $k>1$, and $\a>0$, the following bound holds,
\begin{equation}\label{technical_bound}
    (a+b)^k\,\leq\,\left(1+{2^{k-1}}\,\a  \right)\,a^k \,+\, \left(1+2^{k-1}\,\a^{-1} \right)\,b^k \ .
\end{equation}
\end{lemma}

\begin{proof}
We recall the generalised Young's inequality: for any $a,b, \eta>0$ and  $\zeta, \tilde{\zeta}>0$ such that $\zeta>1$ and $\zeta^{-1}+\tilde{\zeta}^{-1}=1$, we have
$$
ab \leq \frac{a^{\zeta}\eta^{\zeta}}{\zeta}+\frac{b^{\tilde{\zeta}}}{\tilde{\zeta}\eta^{\tilde{\zeta}}} \,.
$$
The statement is now a straightforward consequence of  the above.  Indeed, we have 
\begin{align*}
    (a+b)^k\,=\,\sum_{\ell=0}^k \binom{k}{\ell} a^\ell b^{k-\ell}\,\leq\, a^k+b^k + \sum_{\ell=1}^{k-1} \binom{k}{\ell}\,\left( \frac{\eta_\ell^{p_\ell}}{p_\ell}\, a^k + \frac{1}{\eta_\ell^{q_\ell}\cdot q_\ell}\, b^k \right)\ , 
\end{align*}
with $p_\ell:=k/\ell$, $q_\ell:=k/(k-\ell)$, and for any arbitrary $\eta_\ell>0$. Rearranging, since $q_\ell=p_{k-\ell}$ and $\binom{k}{k-\ell}=\binom{k}{\ell}$ for every $\ell=1,\dots,k-1$, we can write
\begin{align*}
    \sum_{\ell=1}^{k-1} \binom{k}{\ell}\,\left( \frac{\eta_\ell^{p_\ell}}{p_\ell}\, a^k + \frac{1}{\eta_\ell^{q_\ell}\cdot q_\ell}\, b^k \right)\,&
=\, \sum_{\ell=1}^{k-1} \binom{k}{\ell}\, \frac{\eta_\ell^{p_\ell}}{p_\ell}\, a^k \,+\, \sum_{\ell=1}^{k-1} \binom{k}{k-\ell}\, \frac{1}{\eta^{p_{k-\ell}} p_{k-\ell} }\, b^k \\
    &=\, \sum_{\ell=1}^{k-1} \binom{k}{\ell}\, \frac{\eta_\ell^{p_\ell}}{p_\ell}\, a^k \,+\, \sum_{\ell=1}^{k-1} \binom{k}{\ell}\, \frac{1}{\eta_\ell^{p_\ell} p_\ell}\, b^k\ .
\end{align*}
Thus, 
\begin{equation*}
    (a+b)^k\,\leq\, \left(1+\sum_{\ell=1}^{k-1} \binom{k}{\ell}\,\frac{\eta_\ell^{p_\ell}}{p_\ell}  \right)\,a^k \,+\, \left(1+\sum_{\ell=1}^{k-1} \binom{k}{\ell}\,\frac{1}{\eta_\ell^{p_\ell}\,p_\ell}  \right)\,b^k\ .
\end{equation*}
Choosing $\eta_\ell$ such that $\eta_\ell^{p_\ell}=\a$ for every $\ell=1,\dots,k-1$, since
\begin{equation*}
    \sum_{\ell=1}^{k-1}\binom{k}{\ell}\,\frac{1}{p_{\ell}}\,=\,\sum_{\ell=1}^{k-1}\binom{k}{\ell}\,\frac{\ell}{k}\, =\, \sum_{\ell=1}^{k-1}\frac{(k-1)!}{(\ell-1)!(k-\ell)!}\,=\,\sum_{\ell=0}^{k-2}\binom{k-1}{\ell} \,=\,2^{k-1}-1\ ,
\end{equation*}
we obtain the statement.

\end{proof}

\begin{lemma}\label{lemma_xmoments}
Let conditions \eqref{Vk_bounded} and \eqref{cond_Lyap} hold. Then, for any $q>1$, the $q$-th moment of $x_t$ is bounded by
\begin{align*}
    \E_x\left[\left|x_t\right|^{q} \right]\,\leq\, 2\,e^{-b_0 t}  |x|^{q} +{C}_q\ ,
\end{align*}
for any $x\in\R^N$, $t>0$, with $b_0$ as in condition \eqref{cond_Lyap} and with constant ${C}_q>0$ independent of $x$ and $t$.
\end{lemma}

\begin{proof}
Let $m\in\N$ and $f(x):=|x|^{2m}$. Then we can write
\begin{align*}
    \langle U_0(x),\,\nabla f(x)\rangle\,+\,\sum_{k=1}^d V_k(x)^T&\,\nabla^2 f(x)\,V_k(x)\\
    &\,= 2m\,|x|^{2m-2}\langle U_0(x),\,x\rangle+ 2m\,|x|^{2m-2}\sum_{k=1}^d \left|V_k(x)\right|^2\\
    &\qquad+2m(2m-2)|x|^{2m-4} \sum_{k=1}^d \left|\langle V_k(x),x \rangle\right|^2\\
    &\,\leq\,2m\,|x|^{2m-2}\,\left(-b_0|x|^2 +b_1+K(2m-1)\right)\ ,
\end{align*}
by Cauchy-Schwarz, \eqref{cond_Lyap} and \eqref{Vk_bounded}. For $m>1$, we apply the generalised Young's inequality with $p=\frac{m}{m-1}$ and $q=m$,
\begin{equation*}
    2m\,\left(b_1+K(2m-1)\right)\,|x|^{2m-2}\,\leq\, 2(m-1)\a^{\frac{m}{m-1}}\,|x|^{2m}\,+\, 2\left(\frac{b_1+K(2m-1)}{\a}\right)^m\ ,
\end{equation*}
for any $\a>0$. Choosing $\a= b_0^{\frac{m-1}{m}}$, for $f(x)=|x|^{2m}$ we obtain
\begin{align}
    \langle U_0(x),\,\nabla f(x)\rangle\,+\,\sum_{k=1}^d V_k(x)^T\,\nabla^2 f(x)\,V_k(x)\,\leq\, -2{b}_0\, |x|^{2m}\,+\, C_m\ ,\label{bound_2m}
\end{align}
for some constant $C_m>0$ depending only on $m,\;K,\;b_1$ and ${b}_0$. Note that the same bound holds for $m=1$ too, with $C_1=b_1+K$.

Now, for any $m\in\N$ and $t>0$, by It\^{o}'s formula,
\begin{multline*}
\E_x\left[e^{2{b}_0t}f(x_t) \right]-f(x)=\\
 \E_x \left[ \int_0^t e^{2{b}_0 s} \left(2{b}_0 f(x_s)\,+\, \langle U_0(x_s),\nabla f(x_s) \rangle+\sum_{k=1}^d V_k(x)^T\,\nabla^2 f(x)\,V_k(x)\right)\,ds \right]\ .
\end{multline*}
Choosing $f(x)=\left|x\right|^{2m}$, by \eqref{bound_2m} we obtain
\begin{align*}
\E_x\left[e^{2{b}_0t}\left|x_t\right|^{2m} \right]\,\leq\,|x|^{2m} + C_m\,\int_0^t e^{2{b}_0 s} ds\,=\,  |x|^{2m} + \frac{C_m}{2{b}_0}\left(e^{2{b}_0 t}-1\right)\ .
\end{align*}
To conclude note that, by Holder's inequality with $p=\frac{2m}{q}$, choosing $m\in\N$ such that $2(m-1)<q\leq 2m$ (i.e.~$m$ is the smallest integer such that $p\geq 1$), we can write
\begin{align*}
    \E_x\left[ e^{\frac{q\,{b}_0t}{m}}\left|x_t\right|^q \right]\,&\leq\, \E_x\left[ e^{2{b}_0t}\left|x_t\right|^{2m} \right]^{\frac{q}{2m}}\,\leq\,\left(|x|^{2m}+\frac{C_m}{2b_0}\left(e^{2b_0 t}-1\right)\right)^{\frac{q}{2m}}\\
    &\leq\, 2^\frac{q}{2m}\,|x|^q\,+\, C_q\left(e^{2b_0 t}-1\right)^{\frac{q}{2m}} \ ,
\end{align*}
for some constant $C_q$ independent of $x,t,\delta$.

Thus,
\begin{equation*}
\E_x\left[\left|x_t\right|^{q} \right]\leq  e^{-\frac{q\,{b}_0t}{m}}\,2^\frac{q}{2m}\,|x|^q\,+\, C_q\,\left(1-e^{-2b_0 t}\right)^{\frac{q}{2m}} \ ,
\end{equation*}
for every $t\geq 0$, and this proves the statement (since $\frac{q}{2m} >\frac{1}{2}$, for $q>1$).

\end{proof}

In the lemma below we write ``uniformly in $\delta$" in short, meaning uniformly over  $\delta \in (0, 1/(2c_0))$, which is the range of $\delta$'s where the map $F^{\delta}$ is well defined.

\begin{lemma}\label{lemma_derivativesF}
    Assume that there exists a positive function $\lambda(x)>0$, $x\in\R^N$, such that
    \begin{equation}\label{cond_1dU0}
         v^T \nabla U_0(x) v\,\leq\, -\lambda(x)\,|v|^2\ ,
    \end{equation}
    for all $x,v\in\R^N$. Then, $F^\delta$ \eqref{def_Fdelta} satisfies
    \begin{equation}\label{1dF_bound}
        \left|\partial_i F^\delta (x)\right|\,\leq\,  \frac{1}{1+\delta \lambda(F^\delta(x))}\,\leq\, 1\ ,
    \end{equation}
    for every $i\in 1,\dots,N$, uniformly in $\delta>0$ and $x\in\R^N$.

    Moreover, if we further assume that there exists a constant $C>0$ independent of  $\delta>0$ (but possibly dependent on $N$) such that
    \begin{equation}\label{cond_2dU0}
        \left(\sum_{i,j=1}^N\left| \partial_i \partial_j U_0(x)\right|^2\right)^{1/2}\,\leq\, \frac{C}{\delta}\,\left( 1\,+\,\delta\lambda(x) \right)^3\ , 
    \end{equation}
    for every $x\in\R^N$ and $\delta>0$, then,
    \begin{equation}\label{2dF_bound}
    \left|\partial_j\partial_i F^\delta (x)\right|\,\leq\,\frac{\delta}{\left( 1\,+\,\delta\lambda(F^\delta(x)) \right)^3}\,\left(\sum_{m,m'=1}^N\left| \partial_m \partial_{m'} U_0(y)\big|_{y=F^\delta(x)}\right|^2\right)^{1/2}\,\leq\, C\ ,
    \end{equation}
    for every $i,j=1,\dots,N$, uniformly in $\delta>0$, for some constant $C>0$. In particular, $\|F^\delta\|_{\CC^2_b}\leq (1+C)$ uniformly in $\delta>0$. 
\end{lemma}

\begin{note}
    Note that, for every given $x\in\R^N$, the minimum of $\frac{1}{\delta}\,\left(1+\delta \lambda(x)\right)^3$ is $\frac{27}{4}\,\lambda(x)$, obtained at $\delta=(2\lambda(x))^{-1}$. Therefore, under the following more practical condition,
    \begin{equation*}
        \left(\sum_{i,j=1}^N\left| \partial_i \partial_j U_0(x)\right|^2\right)^{1/2}\,\leq\, \widetilde{C}\,\left(1+\lambda(x)\right)\ ,\qquad x\in\R^N\ ,
    \end{equation*}
    for some $\widetilde{C}>0$, with $\lambda(x)>0$ as in \eqref{cond_1dU0}, we can see that condition \eqref{cond_2dU0} is satisfied for every $\delta\in (0,1)$.

    Moreover, note that under Assumption \ref{ass_for_modified_expdecay}, all the hypotheses of Lemma \ref{lemma_derivativesF} are satisfied.
\end{note}

\begin{proof}[Proof of Lemma \ref{lemma_derivativesF}]
    We start with bounding the first-order derivatives. By definition of $F^\delta(x)\in\R^N$, we can write
\begin{equation*}
    \partial_i F^\delta(x)\,=\, e_i\,+\,\delta\,\nabla U_0(y)\big|_{y=F^\delta(x)} \,\partial_i F^\delta(x)\ ,
\end{equation*}
where $(e_i)_{i=1,\dots,N}$ is the canonical orthonormal basis on $\R^N$. Hence, noting that under condition \eqref{cond_1dU0} the matrix $\1\,-\,\delta\, \nabla U_0(y)$ is invertible for every given $y\in\R^N$ and $\delta>0$, we obtain
\begin{equation}\label{dF}
    \partial_i F^\delta(x)\,=\,\left(\1\,-\,\delta\, \nabla U_0(y)\big|_{y=F^\delta(x)} \right)^{-1} e_i\ .
\end{equation}

Denoting $y=F^\delta(x)$ and $A(y):=\1\,-\,\delta\, \nabla U_0(y)$ to simplify the presentation, we obtain $|\partial_i F^\delta(x)|\leq \left\|A(y)^{-1}\right\|_2$, with
\begin{equation*}
    \|A(y)^{-1}\|_2\,=\, \max_{v\neq 0}\frac{|A(y)^{-1} v|}{|v|}\,=\, \max_{w\neq 0} \frac{| w|}{|Aw|}\ ,
\end{equation*}
    and, by condition \eqref{cond_1dU0}, 
    \begin{equation*}
        |A(y) \, w|\,\geq\, \frac{|w^T A(y) w|}{|w|}\,=\,\frac{|w|^2 -\delta w^T \nabla U_0(y) w}{|w|}\,\geq\, (1+\delta \lambda(y))|w|\ .
    \end{equation*}
    Therefore, we obtain
    \begin{equation*}
        |\partial_i F^\delta(x)|\,\leq\, \|A(y)^{-1}\|_2\,\leq\,  \frac{1}{1+\delta \lambda(F^\delta(x))}\ ,
    \end{equation*}
    for every $\delta>0$, which proves \eqref{1dF_bound}.

Now, consider the second-order derivatives,
\begin{align*}
    \left|\partial_j\partial_i F^\delta (x)\right|^2\,&=\, \sum_{k=1}^N \left|\left(\partial_j\partial_i F^\delta (x)\right)_k\right|^2\ ,
\end{align*}
where $(v)_k=\langle e_k, v\rangle$ denotes the $k$-th component of a vector $v\in\R^N$. Noting that
\begin{equation*}
    \left(\partial_i F^\delta(x)\right)_k\,=\, \delta_{i,k} + \delta\,\langle \nabla\left( U_0(y)\right)_k\big|_{y=F^\delta(x)} ,\, \partial_i F^\delta (x)  \rangle\ ,
\end{equation*}
where $\delta_{i,k}$ denotes the Kronecker delta, we can write
\begin{multline*}
     \left(\partial_j\partial_i F^\delta(x)\right)_k\,=\,\partial_j\partial_i \left(F^\delta(x)\right)_k\,=\,\delta\,\langle \nabla^2\left( U_0(y)\right)_k\big|_{y=F^\delta(x)} \,\partial_j F^\delta(x),\, \partial_i F^\delta (x)  \rangle\\
     +\delta\,\langle \nabla\left( U_0(y)\right)_k\big|_{y=F^\delta(x)} ,\, \partial_j\partial_i F^\delta (x)  \rangle\ .
\end{multline*}

Rearranging,
\begin{equation*}
    \langle A(y)\,e_k,\, \partial_i\partial_j F^\delta(x) \rangle\,=\, \delta\,\langle \nabla^2\left( U_0(y)\right)_k\big|_{y=F^\delta(x)} \,\partial_j F^\delta(x),\, \partial_i F^\delta (x)  \rangle\, . 
\end{equation*}
 By \eqref{dF}, we have
\begin{equation*}
    \langle A(y) \,e_k,\, \partial_i\partial_j F^\delta(x) \rangle\,=\, \delta\,\langle \nabla^2\left( U_0(y)\right)_k\big|_{y=F^\delta(x)} \,A(y)^{-1}e_j,\, A(y)^{-1}e_i  \rangle\ .
\end{equation*}
Therefore,
\begin{align*}
    \left| \langle e_k,\,\partial_i \partial_j F^\delta(x) \rangle \right|\,&\leq\, \|A(y)^{-1}\|_2 \,\left|  \langle A(y)\,e_k,\,\partial_i \partial_j F^\delta(x) \rangle \right|\\
    &\leq\,\delta\,\|A(y)^{-1}\|_2^3 \,\left\| \nabla^2 \left(U_0(y)\right)_k\big|_{y=F^\delta(x)}\right\|_2\ ,
\end{align*}
for every $k,i,j=1,\dots,N$. By \eqref{1dF_bound} (and recalling that $\|\cdot\|_2\leq\|\cdot\|$), we can write
\begin{align*}
    \left|\partial_i \partial_j F^\delta(x) \right|\,&=\,\left(\sum_{k=1}^N \left| \langle e_k,\,\partial_i \partial_j F^\delta(x) \rangle \right|^2\right)^{1/2} \nonumber\\
    &\leq\,\frac{\delta}{\left( 1\,+\,\delta\lambda(F^\delta(x)) \right)^3}\,\left(\sum_{k=1}^N\left\| \nabla^2 \left(U_0(y)\right)_k\big|_{y=F^\delta(x)}\right\|^2\right)^{1/2}\nonumber\\
    &=\,\frac{\delta}{\left( 1\,+\,\delta\lambda(F^\delta(x)) \right)^3}\,\left(\sum_{m,m'=1}^N\left| \partial_m \partial_{m'} U_0(y)\big|_{y=F^\delta(x)}\right|^2\right)^{1/2} \ ,
\end{align*}
thus we see that condition \eqref{cond_2dU0} is sufficient to ensure \eqref{2dF_bound}.
    
\end{proof}

\bibliographystyle{abbrvnat}
\bibliography{ms}

\end{document}